\documentclass[11pt,reqno]{amsart}
\usepackage{amssymb,mathrsfs,color}
\usepackage{graphicx}
\usepackage{xcolor} 
\usepackage{tensor}
\usepackage{slashed}
\usepackage{upgreek}
\usepackage{cancel}
\usepackage[margin=1in, marginpar=.6in]{geometry} 

\usepackage{enumitem}

\allowdisplaybreaks				

\usepackage{mathtools}
\usepackage{amssymb,mathrsfs,color,cite}
\usepackage{tikz}
\usepackage{pgfplots}
\usetikzlibrary{snakes}
\usetikzlibrary{intersections, pgfplots.fillbetween}
\usepackage[titletoc,title]{appendix}
\usepackage{hyperref}
\mathtoolsset{showonlyrefs}

\definecolor{green}{rgb}{0,0.8,0} 

\newcommand{\Blue}[1]{\begingroup\color{blue} #1\endgroup}

\newtheorem{theorem}{Theorem}[section]

\newtheorem{lemma}[theorem]{Lemma}
\newtheorem{proposition}[theorem]{Proposition}
\theoremstyle{definition}
\newtheorem{definition}[theorem]{Definition}

\theoremstyle{remark}
\newtheorem{remark}[theorem]{Remark}
\numberwithin{equation}{section}

\makeatletter
\newcommand{\nrm}{\@ifstar{\nrmb}{\nrmi}}
\newcommand{\nrmi}[1]{\Vert{#1}\Vert}
\newcommand{\nrmb}[1]{\left\Vert{#1}\right\Vert}
\newcommand{\abs}{\@ifstar{\absb}{\absi}}
\newcommand{\absi}[1]{\vert{#1}\vert}
\newcommand{\absb}[1]{\left\vert{#1}\right\vert}
\newcommand{\brk}{\@ifstar{\brkb}{\brki}}
\newcommand{\brki}[1]{\langle{#1}\rangle}
\newcommand{\brkb}[1]{\left\langle{#1}\right\rangle}
\newcommand{\set}{\@ifstar{\setb}{\seti}}
\newcommand{\seti}[1]{\{#1\}}
\newcommand{\setb}[1]{\left\{ #1\right\}}
\makeatother

\newcommand{\VERT}[1]{{\left\vert\kern-0.25ex\left\vert\kern-0.25ex\left\vert #1 
    \right\vert\kern-0.25ex\right\vert\kern-0.25ex\right\vert}}

\DeclareMathOperator{\sgn}{sgn\,}
\DeclareMathOperator{\supp}{supp}

\newcommand{\aleq}{\lesssim}

\newcommand{\ud}{\mathrm{d}}
\newcommand{\rd}{\partial}

\newcommand{\0}{\emptyset}

\newcommand{\alp}{\alpha}

\newcommand{\gmm}{\gamma}

\newcommand{\dlt}{\delta}

\newcommand{\Tht}{\Theta}

\newcommand{\omg}{\omega}

\newcommand{\bfT}{{\bf T}}

\newcommand{\bfdlt}{\boldsymbol{\delta}}

\newcommand{\bfeta}{\boldsymbol{\eta}}

\newcommand{\bbP}{\mathbb P}

\newcommand{\bbR}{\mathbb R}
\newcommand{\bbS}{\mathbb S}

\newcommand{\bbZ}{\mathbb Z}

\newcommand{\calA}{\mathcal A}
\newcommand{\calB}{\mathcal B}
\newcommand{\calC}{\mathcal C}
\newcommand{\calD}{\mathcal D}
\newcommand{\calE}{\mathcal E}
\newcommand{\calF}{\mathcal F}
\newcommand{\calG}{\mathcal G}
\newcommand{\calH}{\mathcal H}

\newcommand{\calL}{\mathcal L}
\newcommand{\calM}{\mathcal M}
\newcommand{\calN}{\mathcal N}
\newcommand{\calO}{\mathcal O}
\newcommand{\calP}{\mathcal P}
\newcommand{\calQ}{\mathcal Q}
\newcommand{\calR}{\mathcal R}
\newcommand{\calS}{\mathcal S}
\newcommand{\calT}{\mathcal T}

\newcommand{\calX}{\mathcal X}
\newcommand{\calY}{\mathcal Y}
\newcommand{\calZ}{\mathcal Z}

\newcommand{\frkf}{\mathfrak f}

\newcommand{\uX}{\underline{X}}					
\newcommand{\secondff}{ {\mathrm{I\!I}} }				

\newcommand{\rsmet}{\mathring{\slashed{g}}}			
\newcommand{\ucalC}{\underline{\calC}}				
\newcommand{\uH}{\underline{H}}					
\newcommand{\uZ}{\underline{\calZ}}				

\title[]{Stability of the $3$-dimensional catenoid \\ for the hyperbolic vanishing mean curvature equation}

\definecolor{green}{rgb}{0,0.8,0} 

\newcommand{\pmat}[1]{\begin{pmatrix} #1 \end{pmatrix}}

\newcommand{\Del}[1]{}

\newcommand{\mand}{{\ \ \text{and} \ \  }}

\newcommand{\mif}{{\ \ \text{if} \ \ }}

\newcommand{\angles}[2]{\langle #1,#2\rangle}

\newcommand{\jap}[1]{\langle #1\rangle}

\newcommand{\zed}{\mathfrak{z}}
\newcommand{\txtg}{{\text g}}

\newcommand{\bfOmega}{\boldsymbol{\Omega}}

\newcommand{\abar}{{\underline a}}

\renewcommand{\hbar}{{\underline h}}

\newcommand{\ellbar}{{\underline \ell}}

\newcommand{\Hbar}{{\underline H}}

\newcommand{\Lbar}{{\underline L}}

\newcommand{\Xbar}{{\underline X}}

\newcommand{\alphabar}{\underline{\alpha}}
\newcommand{\betabar}{\underline{\beta}}
\newcommand{\gammabar}{\underline{\gamma}}

\newcommand{\xibar}{\underline{\xi}}

\newcommand{\pibar}{\underline{\uppi}}

\newcommand{\varphibar}{\underline{\varphi}}

\newcommand{\dota}{{\dot a}}

\newcommand{\dotf}{{\dot f}}

\newcommand{\dotu}{{\dot u}}
\newcommand{\dotv}{{\dot v}}

\newcommand{\dotF}{{\dot F}}

\newcommand{\dotK}{{\dot K}}

\newcommand{\dotP}{{\dot P}}

\newcommand{\frakf}{\mathfrak f}

\newcommand{\frakm}{\mathfrak m}

\newcommand{\tilf}{{\tilde{f}}}

\newcommand{\tilg}{{\tilde{g}}}

\newcommand{\tilk}{{\tilde{k}}}

\newcommand{\tilq}{{\tilde{q}}}

\newcommand{\tilr}{{\tilde{r}}}

\newcommand{\tilt}{{\tilde{t}}}

\newcommand{\tilC}{{\tilde{C}}}

\newcommand{\tilF}{{\tilde{F}}}

\newcommand{\tilH}{{\tilde{H}}}

\newcommand{\tilM}{{\tilde{M}}}

\newcommand{\tilN}{{\tilde{N}}}

\newcommand{\tilP}{{\tilde{P}}}

\newcommand{\tilR}{{\tilde{R}}}

\newcommand{\tilS}{{\tilde{S}}}

\newcommand{\tilV}{{\tilde{V}}}

\newcommand{\scB}{{\mathscr{B}}}
\newcommand{\scC}{{\mathscr{C}}}

\newcommand{\scE}{{\mathscr{E}}}

\newcommand{\scQ}{{\mathscr{Q}}}

\newcommand{\ringa}{{\mathring a}}
\newcommand{\ringb}{{\mathring b}}
\newcommand{\ringc}{{\mathring c}}

\newcommand{\ringN}{{\mathring N}}

\newcommand{\ringpi}{\mathring{\uppi}}

\newcommand{\ringSigma}{\mathring{\Sigma}}

\newcommand{\ringPhi}{\mathring{\Phi}}

\newcommand{\ringpsi}{\mathring{\psi}}
\newcommand{\ringepsilon}{\mathring{\epsilon}}

\newcommand{\vecf}{{\vec f}}

\newcommand{\vecu}{{\vec u}}
\newcommand{\vecv}{{\vec v}}

\newcommand{\vecF}{{\vec F}}
\newcommand{\vecG}{{\vec G}}
\newcommand{\vecH}{{\vec H}}

\newcommand{\vecK}{{\vec K}}

\newcommand{\vecN}{{\vec N}}
\newcommand{\vecO}{{\vec O}}
\newcommand{\vecP}{{\vec P}}

\newcommand{\vecZ}{{\vec Z}}

\newcommand{\vecphi}{\vec{\phi}}

\newcommand{\vecfy}{\vec{\varphi}}
\newcommand{\vecPhi}{\vec{\Phi}}

\newcommand{\vecpsi}{\vec{\psi}}
\newcommand{\vecPsi}{\vec{\Psi}}
\newcommand{\vecomega}{\vec{\omega}}
\newcommand{\vecOmega}{\vec{\Omega}}
\newcommand{\barcalC}{{\underline{\calC}}}
\newcommand{\barcalH}{{\underline{\calH}}}

\newcommand{\ringsg}{\mathring{\slashed{g}}}

\newcommand{\ringsDelta}{{\mathring{\slashed{\Delta}}}}
\newcommand{\sDelta}{{\slashed{\Delta}}}
\newcommand{\snabla}{{\slashed{\nabla}}}

\newcommand{\ringsnabla}{{\mathring{\slashed{\nabla}}}}

\newcommand{\dotxi}{{\dot{\xi}}}
\newcommand{\ddotxi}{{\ddot{\xi}}}

\newcommand{\tilpsi}{{\tilde{\psi}}}
\newcommand{\fy}{\varphi}

\newcommand{\tilrho}{{\tilde{\rho}}}
\newcommand{\tilvarphi}{{\tilde{\varphi}}}

\newcommand{\dotell}{{\dot{\ell}}}
\newcommand{\dotphi}{{\dot{\phi}}}
\newcommand{\dotpsi}{{\dot{\psi}}}
\newcommand{\dotfy}{{\dot{\fy}}}

\newcommand{\ddotell}{\ddot{\ell}}

\newcommand{\dotwp}{{\dot{\wp}}}

\newcommand{\ringpartial}{{\mathring{\partial}}}

\newcommand{\far}{{\mathrm{far}}}
\newcommand{\near}{{\mathrm{near}}}

\newcommand{\pert}{{\mathrm{pert}}}
\newcommand{\main}{{\mathrm{main}}}
\newcommand{\Int}{{\mathrm{int}}}
\newcommand{\Ext}{{\mathrm{ext}}}
\newcommand{\tilchi}{{\tilde{\chi}}}
\newcommand{\callP}{\calP}

\newcommand{\temp}{{\mathrm{temp}}}

\newcommand{\tilcalP}{{\widetilde{\calP}}}

\newcommand{\tilDelta}{{\tilde{\Delta}}}

\newcommand{\tilbfOmegabar}{{\tilde{\underline{\bfOmega}}}}

\renewcommand{\min}{{\mathrm{min}}}
\newcommand{\hyp}{{\mathrm{hyp}}}
\newcommand{\flatt}{{\mathrm{flat}}}

\newcommand{\Err}{{\mathrm{Err}}}
\newcommand{\ext}{{\mathrm{ext}}}

\newcommand{\tilpartial}{{\tilde{\partial}}}
\newcommand{\tiluptau}{{\tilde{\uptau}}}
\newcommand{\tiluprho}{{\tilde{\uprho}}}
\newcommand{\tiluptheta}{{\tilde{\uptheta}}}

\newcommand{\vecbfOmega}{{\vec{\bfOmega}}}
\newcommand{\vecUpomega}{{\vec{\Upomega}}}
\newcommand{\RbfT}{{\bfT}}
\newcommand{\tilupsi}{{\tilde{\uppsi}}}
\newcommand{\tilupphi}{{\tilde{\upphi}}}
\newcommand{\tilupfy}{{\tilde{\upvarphi}}}
\newcommand{\upfy}{\upvarphi}
\newcommand{\euc}{{\mathrm{euc}}}
\newcommand{\cat}{{\mathrm{cat}}}

\newcommand{\bfUpomega}{{\boldsymbol{\Upomega}}}
\newcommand{\bfupsigma}{{\boldsymbol{\upsigma}}}

\newcommand{\trap}{{\mathrm{trap}}}

\newcommand{\Reigenfunctioncutoffscale}{{R_\ctf}}
\newcommand{\callO}{\calO}
\newcommand{\bsUpsigma}{{\boldsymbol{\Upsigma}}}
\newcommand{\barbsUpsigma}{{\underline{\bsUpsigma}}}
\newcommand{\tilbsUpsigma}{{\tilde{\boldsymbol{\Upsigma}}}}

\newcommand{\HHbar}{\Hbar}

\newcommand{\glbl}{{\mathrm{glbl}}}
\newcommand{\calPstato}{\calP_h^\mathrm{stat}}
\newcommand{\stat}{{\mathrm{stat}}}

\newcommand{\vectilupphi}{\vec{\tilde{\upphi}}}
\newcommand{\fybar}{\varphibar}
\newcommand{\vecvarepsilon}{{\vec{\varepsilon}}}
\newcommand{\NT}{2}

\newcommand{\ctf}{{\mathrm{ctf}}}

\newcommand{\mrI}{{\mathrm{I}}}
\newcommand{\Dar}{D_{\mathrm{rbx}}}
\newcommand{\tiltilS}{{\tilde{\tilS}}}
\newcommand{\tiltilk}{{\tilde{\tilk}}}
\newcommand{\Rled}{{\tilR}}
\newcommand{\dottilF}{{\dot{\tilF}}}
\newcommand{\bsupvarepsilon}{{\boldsymbol{\upvarepsilon}}}
\usepackage{accents}
\newcommand{\tiltilP}{\tilde{\raisebox{0pt}[0.85\height]{$\tilde{P}$}}}
\begin{document}

\begin{abstract}
We prove that the $3$-dimensional catenoid is asymptotically stable as a solution to the hyperbolic vanishing mean curvature equation in Minkowski space, modulo suitable translation and boost (i.e., modulation) and with respect to a codimension one set of initial data perturbations. The modulation and the codimension one restriction on the initial data are necessary (and optimal) in view of the kernel and the unique simple eigenvalue, respectively, of the stability operator of the catenoid. 

The $3$-dimensional problem is more challenging than the higher (specifically, $5$ and higher) dimensional case addressed in the previous work of the authors with J.~L\"uhrmann, due to slower temporal decay of waves and slower spatial decay of the catenoid. To overcome these issues, we introduce several innovations, such as a proof of Morawetz- (or local-energy-decay-) estimates for the linearized operator with slowly decaying kernel elements based on the Darboux transform, a new method to obtain Price's-law-type bounds for waves on a moving catenoid, as well as a refined profile construction designed to capture a crucial cancellation in the wave-catenoid interaction. In conjunction with our previous work on the higher dimensional case, this paper outlines a systematic approach for studying other soliton stability problems for $(3+1)$-dimensional quasilinear wave equations.
\end{abstract}

\thanks{S.-J.~Oh was supported by a Sloan Research Fellowship and a NSF CAREER Grant DMS-1945615. S. Shahshahani was supported by the Simons Foundation grant 639284. The authors thank J. L\"uhrmann and N. Tang for helpful discussions, and The Erwin Schr\"odinger Institute for their hospitality during a visit where part of this research was conducted.}

\author{Sung-Jin Oh}%
\author{Sohrab Shahshahani}%
\maketitle
\tableofcontents

\section{Introduction}
The \emph{hyperbolic vanishing mean curvature equation} (HVMC equation) is a quasilinear wave equation that is the Lorentzian generalization of the minimal hypersurface equation. Catenoids, which are among the simplest non-flat minimal surfaces in Euclidean space, furnish time-independent solutions of the HVMC equation. The  nonlinear asymptotic stability of catenoids as solutions to the HVMC in $\bbR^{1+(n+1)}$ was studied in \cite{LuOS1} for $n\geq 5$.  In this work we extend the analysis of \cite{LuOS1} to dimension $n=3$. We prove the nonlinear asymptotic stability, up to suitable translations and Lorentz boosts in the ambient space, of the $3$ dimensional catenoid with respect to a ``codimension-$1$'' set of initial data perturbations without any symmetry assumptions. The codimension condition is sharp in view of the unstable mode of the linearized operator for the HVMC equation. The need for adjusting the translation and Lorentz boost parameter (that is, \emph{modulation}) comes from the Lorentz invariance of the problem which leads to the presence of a non-trivial kernel for the linearized operator. Our result also extends the pioneering work of Donninger--Krieger--Szeftel--Wong \cite{DKSW}, which considers the same problem in radial symmetry for $n \geq 2$, albeit for~$n=3$.

Compared to higher dimensions, the $n=3$ case entails a number of essential challenges that will be discussed in detail below. Most significantly, to deal with the slow spatial decay of the catenoid metric to the flat metric, we introduce a refined profile construction for the modulated catenoid. Our refined profile construction incorporates the leading order contribution from this slow spatial decay into the profile and captures a crucial cancellation between the modulated catenoid and the remainder. In addition the slow spatial decay calls for modifications in the proof of local energy decay. Finally, the slower time decay\footnote{Here and elsewhere in this paper, the time decay rate is with respect to an asymptotically null foliation (alternatively, in terms of the time decay in a compact spatial region).} of waves in lower dimensions calls for a new scheme for the proof of time decay for the perturbation. To this end we build on ideas and tools developed in the literature surrounding late time tail estimates for linear waves, for instance in the context of Price's law. Many of the challenges in the present work and \cite{LuOS1} are brought on by the quasilinear nature of the evolution equations. We are hopeful that the innovations introduced in this paper, in conjunction with those from \cite{LuOS1}, will have further applications in the study of stability of solitons in other quasilinear hyperbolic equations on $(3+1)$-dimensional spacetimes, for instance for well-known topological solitons such as the Skyrmion for the Skyrme model.

\subsection{Stability problems for the hyperbolic vanishing mean curvature equation}
We begin by describing the hyperbolic vanishing mean curvature equation. 
Let $(\bbR^{1+4},\bfeta)$ be the $1+4$ dimensional Minkowski space with the standard metric
\begin{align*}
\begin{split}
\bfeta=-\ud X^0\otimes \ud X^0+\ud X^1\otimes \ud X^1+\dots + \ud X^{4}\otimes \ud X^{4}.
\end{split}
\end{align*}
Let $\calM$ be a $4$ dimensional connected orientable manifold without boundary. We consider embeddings $\Phi:\calM\to \bbR^{1+4}$ such that the pull-back metric $\Phi^\ast\bfeta$ is Lorentzian (that is, $\Phi(\calM)$ is timelike), and which satisfy 
\begin{align}\label{eq:HVMC1}
\begin{split}
\Box_{\Phi^\ast \bfeta}\Phi=0.
\end{split}
\end{align}
The vector $\Box_{\Phi^\ast \bfeta}\Phi$  is the \emph{mean curvature vector} of the hypersurface $\Phi(\calM)$ in $\bbR^{1+4}$, and equation~\eqref{eq:HVMC1} is the requirement that this hypersurface have vanishing mean curvature (VMC).  Embeddings satisfying these requirements are called \emph{(timelike) maximal} and equation~\eqref{eq:HVMC1} is referred to as the HVMC equation. When there is no risk of confusion, by a slight abuse of notation, we will identify $\calM$ with its image $\Phi(\calM)$ and simply refer to $\calM$ as a hypersurface of $\bbR^{1+4}$. The HVMC equation is the hyperbolic analogue of the elliptic minimal surface equation (or the parabolic mean curvature flow). Variationally, it asrises as the Euler-Lagrange equations for the area functional 
\begin{align}\label{eq:area1}
\begin{split}
\calA(\Phi)=\int_{\calM}\sqrt{|\det \Phi^\ast\bfeta|}.
\end{split}
\end{align}

The Cauchy problem for \eqref{eq:HVMC1} can be described as follows. Given a coordinate patch $U\subseteq \calM$ with coordinates $s=(s^0,\dots,s^3)$, let $U_0:=\{s\in U~\vert~ s^0=0\}\subset U$. Consider two functions $\Phi_0,\Phi_1\colon U_0\to \bbR^{1+4}$ such that $\Phi_0$ is an embedding, $\Phi_0^\ast \bfeta$ is Riemannian, and the metric
\begin{align*}
g_{\mu\nu}:=\begin{cases}
\bfeta(\partial_{\mu}\Phi_0,\partial_\nu\Phi_0),\quad &\mu,\nu=1,2,3\\
\bfeta(\Phi_1,\partial_\nu\Phi_0),\quad &\mu=0,\nu=1,2, 3\\
\bfeta(\partial_\mu\Phi_0,\Phi_1),\quad &\mu=1,2, 3, \nu=0\\
\bfeta(\Phi_1,\Phi_1),\quad &\mu=\nu=0
\end{cases}
\end{align*}
satisfies $\sup_{U_0}\det g<0$. We seek a neighborhood $V\subseteq U$ of $U_0$ such that there is a timelike embedding $\Phi\colon V\to \bbR^{1+4}$ satisfying \eqref{eq:HVMC1}, as well as $\Phi\vert_{U_0}=\Phi_0$ and $\partial_{0}\Phi\vert_{U_0}=\Phi_1$. It is shown in \cite{AC2} that this problem admits a solution $\Phi$ and that any two solutions $\Phi$ and $\Psi$ are related by a diffeomorphism. In the present work we are interested in manifolds $\calM$ that can be written as direct products $\calM=\bbR\times M$. In this case we use $(t,x)$ to denote points in $\bbR\times M$. Given $\Phi_0:M\to \{0\}\times \bbR^{4}\subseteq \bbR^{1+4}$ and a family of future directed timelike vectors $\Phi_1:M\to \bbR^{1+4}$, by finite speed of propagation and standard patching arguments, the result of \cite{AC2} implies the existence of an interval $I\ni 0$, and a unique solution $\Phi\colon I\times M\to\bbR^{1+4}$ to \eqref{eq:HVMC1} such that $\Phi(t,M)\subseteq \{t\}\times\bbR^{4}$, $\Phi(0)=\Phi_0$, and $\partial_t\Phi(0)=\Phi_1$.

We next turn to questions related to global (in time) dynamics of solutions to \eqref{eq:HVMC1}. In the context of the Cauchy problem formulated on $\bbR\times M$, one can ask if the local solution extends from $I\times M$ to all of $\bbR\times M$, and if so, how it behaves as $t\to\pm\infty$. A special class of maximal hypersurfaces for which the global dynamics are easily described are the products of  Riemannian VMC surfaces in $\bbR^{4}$ with $\bbR$. We refer to such product solutions as \emph{stationary} solutions.  More precisely, if $\Phi_0\colon M\to\bbR^{4}$ is a Riemannian embedding with vanishing mean curvature, then $\Phi\colon \bbR\times M\to\bbR^{1+4}$ given by $\Phi(t,x)=(t,\Phi_0(x))$ satisfies \eqref{eq:HVMC1} with $\Phi(0)=\Phi_0$ and $\partial_t\Phi(0)=(1,0)$.  

A natural question regarding the long time dynamics of solutions of \eqref{eq:HVMC1} is the stability of stationary solutions. The simplest case is when $\Phi_0$ is a linear embedding of a hyperplane in $\bbR^{4}$. It was proved in \cite{B1} that small perturbations of a hyperplane solution lead to global solutions which decay back to a hyperplane. A similar result in two dimensions was later proved in \cite{Lin1}. Analytically, the results in \cite{B1,Lin1} amount to proving global existence and decay estimates for solutions to a system of quasilinear wave equations with small initial data on Minkowski space. From this point of view, hyperplanes can be thought of as the zero solution to~\eqref{eq:HVMC1}.

The first stability result for a non-flat stationary solution of \eqref{eq:HVMC1} is contained in \cite{DonKrie1,DKSW} for the Lorentzian catenoid. The Riemannian catenoid is a VMC surface of revolution in $\bbR^{n+1}$ (see Section~\ref{subsec:Riemcat} for a more detailed description), and the Lorentzian catenoid is the corresponding stationary solution. The authors in \cite{DKSW} consider radial perturbations of the $(1+2)$ dimensional Lorentzian catenoid that satisfy an additional discrete symmetry\footnote{This is an important technical assumption, which avoids the resonances of the linearized operator in dimension two.}. Their main result asserts that if the initial data belong to a codimension one subset in an appropriate topology, then the corresponding solution can be extended globally and converges to a Lorentzian catenoid as $t\to\infty$. The codimension one restriction on the initial data is necessary and sharp (see the comments following Theorem~\ref{thm:main-0}). A similar result for radial perturbations of the Lorentzian helicoid was subsequently obtained in \cite{Marchali1}.  From a PDE point of view, \cite{DKSW,Marchali1} establish the codimension one (asymptotic) stability of a time-independent solution to a quasilinear wave equation on Minkowski space under radial symmetry. In \cite{LuOS1} the authors proved the codimension one stability of the $(1+n)$ dimensional Lorentzian catenoid in dimensions $n\geq 5$ without any symmetry restrictions on the perturbations. In the present work we extend the result of \cite{LuOS1} to $n=3$. For a more complete account of the literature we refer the reader to \cite[Section~1.10.1]{LuOS1}.

\subsection{The Riemannian catenoid $\ucalC$}\label{subsec:Riemcat}
Consider the  Euclidean space $(\bbR^{4}, \bfdlt)$, and let $\uX := (X', X^{4}) \in \bbR^{4}$ be the rectangular coordinates, where $X' = (X^{1}, X^2, X^{3}) \in \bbR^{3}$ denotes the first $3$ coordinates. The \emph{$3$-dimensional catenoid $\ucalC$} is defined as
\begin{equation} \label{eq:catenoid}
	\ucalC := \set{(X', X^{4}) \in \bbR^{4} : \abs{X'}^{2} = \frkf^{2}(X^{4})},
\end{equation}
where $\frkf$ is the unique solution to the ODE
\begin{equation} \label{eq:catenoid-f}
\frac{\frkf''}{(1+(\frkf')^{2})^{\frac{3}{2}}} - \frac{2}{\frkf(1+(\frkf')^{2})^{\frac{1}{2}}} = 0, \qquad \frkf(0) = 1, \quad \frkf'(0) = 0.
\end{equation}
By definition, $\ucalC$ is  \emph{invariant under rotations about the $X^{4}$-axis}, and is \emph{even with respect to the hyperplane $\set{X^{4} = 0}$}. Equation \eqref{eq:catenoid-f} is the condition that $\ucalC$ has \emph{vanishing mean curvature}, i.e., $\ucalC$ is a \emph{minimal hypersurface} (provided, of course, that $\ucalC$ is an embedded submanifold; see its verification below). Upon solving \eqref{eq:catenoid-f}, it follows that $\frkf(X^{4})$ is smooth, even, and strictly increasing from $1$ to $+\infty$ on the interval $X^{4} \in (0, S)$, where
\begin{equation}\label{eq:Sendpointdef1}
S = \int_{1}^{\infty} \frac{\ud \frkf}{(\frkf^{4} + 1)^{\frac{1}{2}}}<\infty.
\end{equation}
By the above properties of $\frkf$, it follows that $\ucalC$ is a smooth embedded $3$-dimensional hypersurface in $\bbR^{4}$ situated between, and asymptotic to, the two parallel hyperplanes
\begin{equation*}
\Pi_{\pm \infty} := \set{(X', X^{4}) \in \bbR^{4} : X^{4} = \pm S}.
\end{equation*}
The catenoid can be parameterized as
\begin{align}\label{eq:Fdef}
\begin{split}
F:I\times \bbS^2\to\bbR^4,\qquad F(\zed,\omega)=(f(\zed)\Theta(\omega),\zed),
\end{split}
\end{align}
where $\Theta:\bbS^2\to \bbR^3$ is the usual embedding of the unit sphere. We introduce the \emph{polar coordinates} $(\rho, \omg) \in \bbR \times \bbS^{2}$ on $\ucalC$, given by (the inverse of) the parametrization
\begin{equation}\label{eq:Zdef1}
	X' = \brk{\rho} \Tht(\omg), \quad X^{4} = \frac{\rho}{\abs{\rho}} \frkf^{-1}(\brk{\rho}),
\end{equation}
where $\frkf^{-1} : (1, \infty) \to (0, S)$ denotes the inverse of $\left. \frkf \right|_{(0, S)}$ and $\jap{\cdot}:=\sqrt{1+|\cdot|^2}$. These coordinates are smooth everywhere on $\ucalC$, including on the \emph{collar} (where $\rho = 0$, or equivalently, $\ucalC \cap \set{X^{4} = 0}$). The metric on $\ucalC$ in these coordinates takes the form
\begin{equation*}
	g_{\ucalC} = \frac{ \brk{\rho}^{2}}{\brk{\rho}^{2}+1 } \ud \rho \otimes \ud \rho
	+ \brk{\rho}^{2} \rsmet_{ab} \ud \omg^{a} \otimes \ud \omg^{b},
\end{equation*}
where $\rsmet$ is the metric of the unit round sphere $\bbS^{2}$. For the Lorentzian catenoid $\calC:=\bbR\times \barcalC$, and with $t$ denoting the time variable in $\bbR$, the Lorentzian metric becomes
\begin{align*}
\begin{split}
g\equiv g_\calC=-\ud t\otimes \ud t+ \frac{ \brk{\rho}^{2}}{\brk{\rho}^{2}+1 } \ud \rho \otimes \ud \rho
	+ \brk{\rho}^{2} \rsmet_{ab} \ud \omg^{a} \otimes \ud \omg^{b}.
\end{split}
\end{align*}

From the second variation of the area functional one can see that the stability, or  linearized, operator for the catenoid as a minimal surface is  $\HHbar:=\Delta_\barcalC+ |\secondff|^2$ (respectively, $\Box_\calC+|\secondff|^2$ in the Lorentzian case), where $\Delta_\barcalC$ denotes the Laplacian on $\barcalC$ (respectively, $\Box_\calC=-\partial_t^2+\Delta_\barcalC$ denotes the d'Alembertian on $\calC$). See for instance \cite{Tam-Zhou,F-CS}. It is known (cf. \cite{F-CS,Tam-Zhou}) that $\HHbar$ admits a unique positive eigenvalue, indicating the instability of the catenoid as a minimal surface:
\begin{align*}
\begin{split}
\HHbar \varphibar_\mu=\mu^2\varphibar_\mu.
\end{split}
\end{align*}
On the other hand, since every translation of $\barcalC$ in the ambient $\bbR^{4}$ is another minimal surface (another catenoid), by differentiating in the translation parameter one obtains $4$ zero modes of $\HHbar$. Explicitly, in the $(\rho,\omega)$ coordinates above, these are given by
\begin{align*}
\begin{split}
\fybar_j=\frac{\Theta^j(\omega)}{\jap{\rho}^{2}},\quad 1\leq j\leq 3,
\end{split}
\end{align*}
corresponding to translations in the direction $\frac{\partial}{\partial X^j}$ respectively. The zero mode $\fybar_{4}=\frac{\rho\sqrt{\jap{\rho}^{2}+1}}{\jap{\rho}^{2}}$ corresponding to translation in the direction of the axis of symmetry does not belong to $L^2$, which is why we have not included it in the list above. In the Lorentzian case, the Lorentz boosts of the ambient $\bbR^{1+4}$ give the additional zero modes $t\fybar_j$ of $\Box_\calC+|\secondff|^2$, which, for $1\leq j\leq 3$.\footnote{Ambient rotations about the axis of symmetry map $\barcalC$ to itself, so differentiation along the rotation parameter yields the trivial zero mode $\fybar=0$ for $\HHbar$. Similarly for translations along the time axis in the Lorentzian case. Scaling changes the value of $S$ in \eqref{eq:Sendpointdef1} and differentiation in the scaling parameter yields a zero solution which is neither an eigenfunction nor a resonance.} We note in passing that the normal $\nu$ to $\barcalC\subseteq \bbR^{4}$ satisfies (see \cite[sections~1.2 and~2.1.8]{LuOS1})
\begin{align*}
\begin{split}
\nu^i= \fybar_i,\qquad i= 1,2,3.
\end{split}
\end{align*}
For this reason we sometimes use $\nu^i$ to denote the kernel elements of $\Hbar$.

The implication of such a non-trivial kernel (arising from symmetries) for $\Box_\calC+|\secondff|^2$ for the stability problem is that a perturbation of the initial data for the catenoid may lead to a solution that approach a boosted and translated catenoid. We next give a more explicit description of the boosted and translated catenoids. For any $\ell_0\in \bbR^3\backslash\{0\}$ let $P_{\ell_0}$ denote the orthogonal projection in the direction of $\ell_0$ ($P_{\ell_0} v=(|\ell_0|^{-2}\ell_0\cdot v)\ell_0$), and $P_{\ell_0}^\perp$ the orthogonal projection to the complement. Here and below, by a slight abuse of notation, we  view $\bbR^3$ as a subset of $\bbR^{4}$ using the embedding  $X'\mapsto (X',0)^\intercal$. The corresponding ambient Lorentz boost $\Lambda_{\ell_0}$, with $0<|\ell_0|<1$, is defined by
\begin{align}\label{eq:Lambdadef1}
\begin{split}
\Lambda_{\ell_0}=\pmat{\gamma&-\gamma\ell_0^\intercal\\-\gamma \ell_0&A_{\ell_0}},\qquad A_{\ell_0}= \gamma P_{\ell_0}+P_{\ell_0}^\perp,\quad \gamma=\frac{1}{\sqrt{1-|\ell_0|^2}},
\end{split}
\end{align}
and the inverses of $\Lambda_{\ell_0}$ and $A_{\ell_0}$ are
\begin{align*}
\begin{split}
\Lambda_{\ell_0}^{-1}=\Lambda_{-\ell_0},\qquad A_{\ell_0}^{-1}=\gamma^{-1}P_{\ell_0}+P_{\ell_0}^\perp.
\end{split}
\end{align*}
In what follows, Lorentz boosts are always with respect to a direction vector of length strictly less than one. The Lorentzian catenoid boosted by $\ell_0\in \bbR^3$ and translated by $a_0\in \bbR^3$ is  the following HVMC submanifold of $\bbR^{1+4}$,
\begin{align*}
\begin{split}
\calC_{a_0,\ell_0}:=\{\Lambda_{-\ell_0} X~\vert~ X\in \calC\}+\pmat{0\\a_0}.
\end{split}
\end{align*}
\subsection{First Statement of the main theorem}\label{sec:introfirststatement}
We are now ready to give a first formulation of the main result of this paper. For a manifold $\calX$, we will use the notation $\calT_p\calX$ to denote the tangent space at $p\in \calX$. We also use the parameterization $F\colon I\times \bbS^{2}\to \bbR^{4}$ introduced in \eqref{eq:Fdef}. By a slight abuse of notation we often identify $\barcalC$ with $I\times \bbS^{2}$ and view functions on $I\times \bbS^{2}$ as functions on $\barcalC$.

\begin{theorem}\label{thm:main-0} Let $\Phi_0\colon I\times \bbS^{2} \to\{0\}\times\bbR^{4}$ be an embedding and~$\Phi_1\colon I\times \bbS^{2}\to \bbR^{1+4}$ be a family of future directed timelike vectors such that $\Phi_0=F$ and $\Phi_1=(1,0)$ outside of a compact set. Suppose $\Phi_0$ and $\Phi_1$ belong to an appropriate codimension-1 subset in a suitable topology, and are sufficiently close to $F$ and $(1,0)$, respectively, in this topology. Then there is a unique complete timelike VMC hypersurface $\calM$ in $\bbR^{1+4}$ such that $\calM\cap \{X^0=0\}=\Phi_0(\barcalC)$ and $\calT_{\Phi_0(p)}\calM$ is spanned by $\ud_p \Phi_0(\calT_p\barcalC)$ and $\Phi_1(p)$ for any $p\in I \times \bbS^{2}$. Moreover, there exist $a_0,\ell_0\in \bbR^3$ such that the ambient Euclidean distance between $\calM\cap\{X^0=t\}$ and $\calC_{a_0,\ell_0}\cap \{X^0=t\}$ tends to zero as $t\to\infty$.
\end{theorem}
A more precise version of this result will be stated in Theorem~\ref{thm:main} below. A few remarks are in order.
\begin{itemize}[leftmargin=*]
  \item[{\bf{1.}}] In view of the growing mode of the stability operator  (see Section~\ref{subsec:Riemcat}), the codimension-1 restriction on the data in Theorem~\ref{thm:main} is optimal. However, we do not pursue the question of uniqueness or regularity of the codimension-1 set in the initial data topology. See Remark~\ref{rem:codim1} for further discussion.
  \item [{\bf{2.}}] As part of the proof of Theorem~\ref{thm:main-0} we derive ODEs that track the translation and boost parameters. See Section~\ref{sec:eqnsofmotion}.
\item [{\bf{3.}}] The assumption $\Phi_0=F$ and $\Phi_1=(1,0)$ outside a compact set can be replaced by sufficient decay at spatial infinity. Indeed, outside an ambient cone $\scC$ with vertex at $(-R,0)$, with $R$ sufficiently large, the problem reduces to a quasilinear wave equation on Minkowski space. By finite speed of propagation, this problem can be analyzed separately in this region, for instance using the vectorfield $(t-r)^{p} \partial_{t}$. This will lead to suitably decaying and small data on the cone $\scC$ which can be taken as the starting point of the analysis in this paper. Note that in this region the distance between $\calC_{a_j,\ell_j}\cap\{X^0=t\}$ for any $a_j,\ell_j$, $j=1,2$, with $|\ell_j|<1$ decays to zero as $t\to\infty$ in view of the strong asymptotic decay of the catenoid metric to the flat metric.
\end{itemize}
\subsection{Overall scheme and main difficulties} \label{subsec:ideas-outline}
We begin with a description of the overall scheme of the proof. We then outline the main difficulties in carrying out this scheme. This discussion also serves as a brief summary of the main achievements in \cite{LuOS1} on which we rely in this paper. 

\subsubsection{Overall scheme}\label{subsubsec:introoverallscheme} The overall scheme of our proof of Theorem~\ref{thm:main-0} is as follows:
\begin{enumerate}
\item {\it Decomposition of solution.} The basic idea is to make the (formal) decomposition
\begin{equation} \label{eq:basic-decomp}
	\hbox{(Solution)} = \underbrace{\calQ}_{\hbox{profile}} \hbox{``}+\hbox{''} \underbrace{\psi}_{\hbox{perturbation}}
\end{equation}
and show that the perturbation $\psi$ decays to zero as $t \to \infty$ in a suitable sense. In the absence of any obstructions, the profile $\calQ$ would be the object that we wish to prove the asymptotic stability of -- the catenoid in our case. However, as discussed earlier, the linearized HVMC equation around the catenoid, $(-\partial_{t}^{2} + \underline{H}) \psi = 0$, admits non-decaying solutions. These come from the kernel, generated by the translation symmetry, and positive eigenvalue of $\uH$. To avoid these obstructions, we employ the ideas of \emph{modulation} and \emph{shooting}. 

\smallskip

\item {\it Modulation.} To ensure transversality to the $6$-dimensional family of non-decaying solutions corresponding to the generalized kernel if $-\partial_t^2+\uH$ we impose $6$ orthogonality conditions on $\psi$ at each time. To compensate for such a restriction, we allow the profile $\calQ$ to depend on $6$ time-dependent parameters: the \emph{position} vector $\xi(\sigma) = (\xi^{1}, \xi^2, \xi^{2})(\sigma)$, and the \emph{velocity} (or \emph{boost}) vector $\ell(\sigma) = (\ell^{1}, \ell^2, \ell^{3})(\sigma)$.  Here $\sigma$ is a foliation parameters whose leaves will represent an appropriate notion of time for our problem. The \emph{approximate solution} (or \emph{profile}) $\calQ = \calQ(\xi(\cdot), \ell(\cdot))$ to HVMC then represents ``a moving catenoid at position $\xi(\sigma)$ with velocity $\ell(\sigma)$ at each time $\sigma$''. Appropriate choices of the profile $\calQ$, the foliation $\sigma$, and the $6$ orthogonality conditions would lead, upon combination with the HVMC equation, to $6$ equations that dictate the evolution of $(\xi(\sigma), \ell(\sigma))$ in terms of $\psi$.

\smallskip

\item {\it Shooting argument.} To avoid the exponential growth stemming from the positive eigenvalue of $\uH$, we further decompose $\psi$ as
\begin{equation*}
	\psi = a_{+}(\sigma) Z_{+} + a_{-}(\sigma) Z_{-} + \phi,
\end{equation*}
where $Z_{+}$ and $Z_{-}$ are uniformly bounded functions, $a_{+}(\sigma)$ and $a_{-}(\sigma)$ obey ODEs in $\sigma$ with growing and damping linear parts, respectively, and $\phi$ obeys $8$ orthogonality conditions so as to be transversal (to a sufficient extent) to all possible linear obstructions to decay at each time. By analyzing the ODE for $a_{-}$, the modulation equations for $(\xi, \ell)$, and the wave equation for $\phi$, we will show that $\dot{\xi} - \ell$, $\dot{\ell}$ and $\psi$ decay as long as the unstable mode $a_{+}(\sigma)$ satisfies the so-called \emph{trapping assumption}, which roughly says that $a_{+}(\sigma)$ decays in time. We then employ a topological \emph{shooting argument} to select a family of initial data -- which is codimension $1$ in the sense described below in Theorem~\ref{thm:main} and Remark~\ref{rem:codim1} -- such that $a_{+}$ indeed continues to satisfy the trapping assumption for all times $\sigma \geq 0$.

\smallskip

\item {\it Integrated local energy decay and vectorfield method.} Finally, we study the quasilinear wave equation satisfied by $\phi$ which satisfies $8$ orthogonality conditions. Under suitable bootstrap assumptions (to handle nonlinear terms) and the trapping assumption for $a_{+}$, we prove the pointwise decay of $\phi$ via the following steps:
\begin{align*}
&\hbox{(transversality to linear obstructions)}  \\
&\Rightarrow \hbox{(uniform boundedness of energy and integrated local energy decay)} \\
& \Rightarrow 
\hbox{(pointwise decay)}
\end{align*}
Here, integrated local energy decay (ILED) estimates refer to, roughly speaking, bounds on integrals of the energy density on spacetime cylinders for finite energy solutions. They are a weak form of dispersive decay. These have the advantage of being $L^{2}$-based and hence being amenable to a wide range of techniques, such as the vectorfield method, Fourier transform, spectral theory, etc. A powerful philosophy, that has recently arisen in works \cite{DR1, Tat2, MTT, OlSt} related to the problem of black hole stability, is to view integrated local energy decay as a key intermediate step for obtaining stronger pointwise decay (see also \cite{Tat1, MeTa, RodSch1} for papers in the related context of global Strichartz estimates). Specifically, in our proof we adapt the $r^{p}$-method of Dafermos--Rodnianski \cite{DR1}, extended by Schlue \cite{Schlue1} and Moschidis \cite{Moschidis1}.

The scheme described at the level of generality here is the same as that in higher dimensions from \cite{LuOS1}. There are however several important differences in our construction of the modified profile and execution of the ILED and pointwise decay estimates. These are discussed in Sections~\ref{subsubsec:intromaindiff},~\ref{sec:refprofintro1},~\ref{subsec:ILEDintro}, and~\ref{subsec:tailintro} below.

\end{enumerate}

\subsubsection{Main difficulties}\label{subsubsec:intromaindiff} There are several significant challenges in implementing the above scheme that are also present in higher dimensions:
\begin{itemize}
\item {\it (Quasilinearity)} First and foremost, the hyperbolic vanishing mean curvature equation is \emph{quasilinear}. In particular, in the proof of ILED we face second order perturbations of $-\partial_t^2+\uH$. Furthermore, since the highest order term is nonlinear, at various places in the proof we need to be careful to avoid any derivative losses. 

\item {\it (Gauge choice)} Another basic point about HVMC is that it is an equation for a geometric object, namely a hypersurface in $\bbR^{1+4}$. Hence, we need to fix a way of describing the hypersurface by a function to perform any analysis.

\item {\it (Profile and foliation construction)} In order for the above scheme to work, it is crucial for the profile $\calQ$, representing a moving catenoid at position $\xi(\sigma)$ with velocity $\ell(\sigma)$ at each time $\sigma$, to solve the HVMC equation up to an adequately small error. Unfortunately, the obvious construction based on the standard $t = X^{0}$-foliation would lead to an inappropriately large error. The key issue is the inaccuracy of the construction in the far-away region, which is fatal due to the slow decay of the catenoid  towards the flat hyperplane. As we will see, we are led to consider a different foliation $\sigma$ consisting of \emph{moving asymptotically null leaves}.
\item {\it (Proof of integrated local energy decay)}  Proving ILED for the solution $\phi$ to the quasilinear wave equation satisfying our orthogonality conditions, however, is met with several difficulties, such as (i) quasilinearity, (ii) existence of a trapped null geodesic (traveling around the collar $\set{\rho = 0}$ in case of $\calC$), (iii) existence of zero and negative eigenvalues of $\underline{H}$ (what we referred to as linear obstructions to decay) and (iv) nonstationarity of the profile $\calQ$.

\item {\it (Modulation theory and vectorfield method)} Standard modulation theory \cite{Stuart1, Weinstein1} is based on the standard $t = X^{0}$-foliation, whose leaves are flat spacelike hypersurfaces; the method needs to be adapted to the foliation $\sigma$ used in our profile construction. Similarly, we need to adapt the $r^p$ vectorfield method to our foliation $\sigma$ of moving asymptotically null leaves. The presence of linear obstructions to decay also needs to be incorporated.
\end{itemize}

As mentioned earlier, these difficulties are already present in higher dimensions, considered in \cite{LuOS1}. We now describe the main new challenges that are specific to three dimensions. The heart of the problem is that our scheme requires having a twice integrable decay rate for the parameter derivatives $\dotell$ and $\dotxi-\ell$. First, such a strong decay rate is needed to guarantee the existence of the final translation and boost parameters $a_{0}, \ell_{0} \in \bbR^{3}$ (so that our solution tends to $\calC_{a_{0}, \ell_{0}}$ as $t \to \infty$), which is one of the key objectives of our main theorem. Moreover, this decay requirement also plays a crucial role in our analysis. More specifically, to prove ILED we need to use the orthogonality conditions satisfied by the perturbation $\phi$, and this relies on  freezing the coefficients of the linearized operator at the final bootstrap time. In this context, the twice integrable decay rate is used to control the resulting errors.\footnote{It is conceivable that if one does not aim to ensure the existence of the final translation and boost parameters $a_{0}, \ell_{0} \in \bbR^{3}$, then the global existence part of our main theorem may be proved with less stringent spatial decay assumptions on the initial perturbation, but with a weaker control on the parameter evolution and the late-time tail of the perturbation. The techniques developed in the recent work \cite{DaHoRoTa2} may be relevant for this direction.} 

To obtain a twice integrable decay rate for the parameter derivatives, our scheme in turn requires the remainder $\phi$ to decay at such a rate. Indeed, to obtain orthogonality conditions that are consistent with our foliation of the domain that is adapted to the moving center of the boosted and translated catenoids, we use truncated versions of the exact eigenfunctions of the linearized operator; see \cite[Sections~1.4--1.7]{LuOS1}  for a more detailed discussion of this point. A consequence of not using exact eigenfunctions is that in the resulting modulation equations, the parameter derivatives depend linearly on $\phi$, leading to the requirement that $\phi$ should decay at a twice-integrable rate.

When the spatial dimension is at least five, the $r^p$ method already gives a twice integrable decay rate for $\phi$ as desired. However, this scheme runs into a few difficulties in dimension three:
\begin{enumerate}
\item The $r^p$ method as in \cite{Moschidis1} at best gives the decay rate $\tau^{-\frac{3}{2}}$ for $\phi$, which is not twice integrable.
\item The eigenfuctions in the kernel of the linearized operator have spatial decay $|\rho|^{-2}$ and do not belong to the dual local energy space. This creates technical issues when: (i) trying to ``project away" from these modes in the proof of ILED, and (ii) gaining smallness from the decay of the eigenfunctions when using truncated versions of the eigenfunctions. 
\item Since the profile is not an exact solution, the equation for $\phi$ contains a source term, that is, a term that is independent of $\phi$. The slow, $|\rho|^{-3}$, spatial decay of this source term in dimension three affects the $\tau$ decay rate for $\phi$. In fact, even for the wave equation on $\bbR^{1+3}$, a source term with spatial decay $|\rho|^{-3}$ restricts the interior time decay to $\tau^{-2}$, which is not twice integrable. This can be seen for instance by integration along characteristics in the radial case. Therefore, even if we can improve the time decay beyond what is given by the $r^p$ method, we cannot expect a twice integrable rate for $\phi$. To resolve this issue, we further decompose $\phi$ as $\phi=P+\varepsilon$ such that $P$ ``solves away" the contribution of the slowly decaying part of the source term. This needs to be done in a way that: (i) $\varepsilon$ satisfies an equation for which we expect a twice integrable decay rate, and (ii) there is a cancellation in the modulation ODEs, in the sense that the slow time decay of $P$ does not contribute. We will refer to $P$ as the \emph{refined} or \emph{modified} profile.
\end{enumerate}
In the foregoing discussion we have used $\tau$ to denote time parameter and $\rho$ to denote the radial variable measured from the center of the moving catenoid. We will continue to use this notation in the remainder of the introduction. In Sections~\ref{subsec:mainprofile}--\ref{subsec:tailintro} below we discuss the resolutions to the challenges outlined in this section.
\subsection{Main profile: construction of the moving catenoid}\label{subsec:mainprofile}
We recall the construction of the main catenoid profile from \cite{LuOS1}, but with a slightly different normalization. The different normalization does not play a fundamental role but yields slightly nicer formulas in some computations (in particular, in the notation of \cite[Section~4.1]{LuOS1}, $\frac{\ud\sigma}{\ud\tau}=1$ instead of $\frac{\ud\sigma}{\ud\tau}=1-\gamma'R$ and there is no need to introduce $\eta$ as a modification of~$\xi$; see Section~\ref{subsubsec:2ndordereq} below). Let $\xi$ and $\ell$ be given curves with $|\dotxi|, |\ell|<1$, and let $R\gg1$ be a fixed constant. For future applications we also assume that $|\dotell|$, $|\dotxi|$, and $|\ell|$ are sufficiently small, although the smallness is not necessary for the profile construction. Recall that
\begin{align}\label{eq:calCsigmadef1}
\begin{split}
\calC_\sigma = \{\Lambda_{-\ell(\sigma)}X~\vert~X\in \calC\}+ \pmat{0\\\xi(\sigma)-\sigma\ell(\sigma)}.
\end{split}
\end{align}
We define the profile to be
\begin{align*}
\begin{split}
\calQ := \cup_\sigma \Sigma_\sigma,\qquad \Sigma_\sigma:=\calC_\sigma\cap\bsUpsigma_\sigma,
\end{split}
\end{align*} 
where $\cup_\sigma \bsUpsigma_\sigma$ is a foliation of the interior of the cone
\begin{align*}
\begin{split}
\scC_{-R}=\{X\in\bbR^{1+4}\vert X^0+R-2=|\Xbar|\},
\end{split}
\end{align*}
which we now describe. To  define $\bsUpsigma_\sigma$ let $\calH_\sigma$ denote the reference hyperboloids
\begin{align*}
\begin{split}
\calH_\sigma = \{y = (y^0,y',y^4)~\vert ~y^0-\gamma^{-1}(\sigma)\sigma=\sqrt{1+|y'|^2})\}.
\end{split}
\end{align*}
In general we denote the restriction to $\{X^{4}=S\}$ by an underline, so for instance
\begin{align*}
\begin{split}
\barcalH_\sigma = \calH_\sigma\cap\{y^{4}=S\}.
\end{split}
\end{align*}
The hyperboloids adapted to $\ell$ and $\xi$ are (here $\xi$, $\ell$, and $\gamma$ are evaluated at $\sigma$)
\begin{align*}
\begin{split}
\tilbsUpsigma_\sigma=\Lambda_{-\ell}\calH_\sigma+(-R,\xi-\sigma\ell)^\intercal=\{X~\vert~X^0-\sigma+R=\sqrt{1+|X'-\xi|^2}\}.
\end{split}
\end{align*}
It follows from the fact that $\xi$ is timelike that $\cup_\sigma\tilbsUpsigma$ provides a foliation of a region containing $\scC_{-R}$ (see \cite[Remark 1.3]{LuOS1}). To arrive at $\bsUpsigma_\sigma$ we modify the hyperboloids $\tilbsUpsigma_\sigma$ as follows so they have flat interiors. First let $\frakm:\bbR\to\bbR$ be a smoothed out function of the minimum function such that for some small $\delta_1>0$
\begin{align*}
\begin{split}
\frakm(x,y)=\min(x,y)\mand |x-y|>\delta_1.
\end{split}
\end{align*}
We assume that all derivatives of $\frakm$ are bounded. Then define $\sigma,\sigma_\temp:\{-S\leq X^4\leq S\}\to \bbR$ as
\begin{align*}
\begin{split}
&\sigma_\temp(X)=\sigma'\mif X\in \tilbsUpsigma_{\sigma'},\\
&\sigma(X)=\frakm(X^0,\sigma_\temp(X)).
\end{split}
\end{align*}
With these definitions we let
\begin{align}\label{eq:bsUpsugmasigmadef1}
\begin{split}
\bsUpsigma_{\sigma'}=\{X~\vert~\sigma(X)=\sigma'\},\quad \barbsUpsigma_{\sigma'}=\bsUpsigma_{\sigma'}\cap \{X^4=S\}.
\end{split}
\end{align}
If $\frakm$ is suitably chosen, $\cup_\sigma \bsUpsigma$ also provides a foliation of a region containing $\scC_{-R}$. The hyperboloidal, where $X^0\geq \sigma_\temp+\delta_1$, and flat, where $\sigma_\temp\geq X^0+\delta_1$, parts of of the leaves are denoted by $\bsUpsigma_\sigma^\hyp$, $\barbsUpsigma_\sigma^\hyp$ and $\bsUpsigma_\sigma^\flatt$, $\barbsUpsigma_\sigma^\flatt$ respectively. We will often informally refer to the region inside a large compact set, usually contained in the flat part of the foliation, as the interior, and to the complement of this region as the exterior. The region where $\sigma_\temp-\delta_1\leq X^0\leq \sigma_\temp+\delta_1$ is referred to as the transitioning region.

To describe the VMC hypersurface we need to fix a gauge. Let $\chi$ be a cutoff function that is equal to one on $\calC_\flatt:=\calC_\sigma\cap\bsUpsigma_\sigma^\flatt$ and equal to zero in $\calC_\hyp:=\calC_\sigma\cap\bsUpsigma_\sigma^\hyp$. If $p\in \Sigma_\sigma$ is given by $p = \Lambda_{-\ell(\sigma)}q+(0,\xi(\sigma)-\sigma\ell(\sigma))^\intercal$ for some $q\in \calC$, let $n_\wp(p)=\Lambda_{-\ell(\sigma)}n(q)$ where $n(q)$ is the normal to $\calC$ at $q$. With $\tilN_\Int$ denoting the normal to $\Sigma_\sigma$ viewed as a subspace of $\bsUpsigma_\sigma$, let  
\begin{align*}
\begin{split}
\tilN=\chi \tilN_\Int+(1-\chi)\frac{\partial}{\partial X^4},
\end{split}
\end{align*}
and defined $N$ to be the unique vector that is parallel to $\tilN$ and satisfies $\bfeta(n_\wp,N)=1$:
\begin{align}\label{eq:Ndef1}
\begin{split}
N\parallel \tilN\quad\mand\quad \bfeta(n_\wp,N)=1.
\end{split}
\end{align} The perturbation $\psi:\cup_\sigma\Sigma_\sigma\to\bbR$ is then defined by the requirement that
\begin{align}\label{eq:psidef1}
\begin{split}
p+\psi(p)N(p)\in \calM,\qquad \forall p \in \cup\Sigma_\sigma.
\end{split}
\end{align}
The fact that this requirement determines $\psi$ uniquely is proved in \cite[Lemma 2.2]{LuOS1}. Note that the \emph{almost normal} vector $N$ is chosen so that in the hyperboloidal part of the foliation $\calM$ is parameterized as a graph over the asymptotic hyperplanes of the catenoid. We refer the reader to \cite[Section~1.5]{LuOS1} for more motivation behind our gauge choice.
\subsection{Refined profile construction}\label{sec:refprofintro1}
Consider the model equation
\begin{align}\label{eq:modelprofile1}
\begin{split}
\calP\phi = g+\calN,
\end{split}
\end{align}
where $g$ denotes a source term with spatial decay $|\rho|^{-3}$ and twice integrable time decay $\tau^{-2-\gamma}$. As shown in Lemma~\ref{lem:calG01} this is a good approximation to our problem. $\calN$ contains the nonlinearity and source terms with twice integrable time decay and spatial decay $O(|\rho|^{-4})$. $\calP$ is the main linear operator which we take to be of the form
\begin{align*}
\begin{split}
\calP\phi=\frac{1}{\sqrt{|h|}}\partial_\mu (\sqrt{|h|}(h^{-1})^{\mu\nu}\partial_\nu\phi)+V\phi,
\end{split}
\end{align*}
with $h$ the expression for the catenoid metric in some (non-geometric) global coordinates. See Section~\ref{subsec:prelimvfs}. In our applications, the source term $g$ appears because the modulated catenoid is not an exact solution of the HVMC equation. The failure is because $\dotxi-\ell$ and $\dotell$ are not zero, although they decay at a twice integrable rate in our bootstrap assumptions. We will use $\dotwp$ to denote the parameter derivatives $\dotxi-\ell$ and $\dotell$ in the rest of this discussion. For simplicity of exposition, we assume that $\uH$ has a one dimensional kernel, spanned by $\fybar$, and no positive eigenvalues. We also assume that the coefficients of $h$ are independent of time. Our simplified discussion under these assumptions captures the main ideas of the general case which is carried out in Section~\ref{sec:modprof}. As already mentioned, the best time decay we can expect for $\phi$ in \eqref{eq:modelprofile1} is $\tau^{-2}$. Indeed, even in the Minkowski case
\begin{align*}
\begin{split}
\Box_m\phi=\jap{\tau}^{-2-\gamma}\jap{\rho}^{-3},
\end{split}
\end{align*}
it can be seen by integration along characteristics that the time decay of the solution is $\tau^{-2}$ even if $\gamma$ is very large. On the other hand, without further modifications to the scheme in \cite{LuOS1}, $\phi$ appears linearly in the ODEs for $\xi$ and $\ell$ so $\dotwp$ cannot decay faster than $\tau^{-2}$, and the bootstrap assumptions cannot be closed. To remedy this, we further decompose $\phi$ into a refined profile $P$ and a remainder $\varepsilon$, $\phi=P+\varepsilon$, as follows. Our goal is to define $P$ so that it solves away the contribution of $g$ but such that it does not appear in the parameter ODEs. For this we need the analogue of \eqref{eq:modelprofile1} in first order formulation which we write as follows
\begin{align}\label{eq:modelprofile2}
\begin{split}
(\partial_t-M)\vecphi=\vecF_0+\vecF_1.
\end{split}
\end{align}
Here $\vecphi=(\phi,\dotphi)^\intercal$ is the vector form of $\phi$ and $M$ denotes the spatial part of the main operator. $\vecF_0$ is the main source term, corresponding to $g$, and $\vecF_1$ corresponds to $N$. In the interior $\vecF_0$ has the form
\begin{align*}
\begin{split}
\vecF_0\approx \pmat{\fybar\\0}(\dotxi-\ell)+\pmat{0\\\fybar}\dotell.
\end{split}
\end{align*}
The parameter ODEs are derived by using \eqref{eq:modelprofile2} to compute $\bfOmega(\vecF_0,\vecZ_1)$ and $\bfOmega(\vecF_0,\vecZ_2)$ where $\vecZ_1=(\chi_{\leq R_\ctf}\fybar,0)^\intercal$, $\vecZ_1=(0,\chi_{\leq R_\ctf}\fybar)^\intercal$, and $\bfOmega((u,\dotu)^\intercal,(v,\dotv^\intercal))=\int(u\dotv-v\dotu)$. Because of the form of $F_0$ the coefficients of $\dotxi-\ell$ and $\dotell$ are strictly positive if $R_\ctf$ is large (almost $\|\fybar\|_{L^2}^2$), that is, $\bfOmega(\vecF_0,\vecZ_1)\approx \dotell$ and $\bfOmega(\vecF_0,\vecZ_2)\approx \dotxi-\ell$. This suggests that the refined profile should be defined such that $\calP P=g$ in the exterior, but $\calP P=0$ in the interior. The latter condition will lead to $(\partial_t-M)\vecP\approx 0$ in the interior, so the non-vanishing of the coefficients of $\dotxi-\ell$ and $\dotell$ is preserved. Here $\vecP$ denotes the vector form of the refined profile $P$. That is, as a first guess the profile can be defined as a solution to the wave equation $\calP P = g_\Ext$, where $g_\Ext$ denotes a truncation of $g$ to the exterior. In the second order formulation, this yields the equation
\begin{align}\label{eq:modelprofile3}
\begin{split}
\calP\varepsilon = g-g_\Ext + \calN,
\end{split}
\end{align}
where on the right-hand side every appearance of $\phi$ is replaced by $P+\varepsilon$. In the first oder formulation we get (with $\vecP\approx (P,\partial_\tau P)^\intercal$)
\begin{align}\label{eq:modelprofile4}
\begin{split}
(\partial_t-M)\vecP\approx 0,\qquad\mathrm{in~}\supp \vecZ_i.
\end{split}
\end{align}
Combining \eqref{eq:modelprofile2} and \eqref{eq:modelprofile4} and letting $\vecvarepsilon=\vecphi-\vecP$ we get 
\begin{align*}
\begin{split}
\bfOmega(\calF_0,\vecZ_i)\approx\bfOmega((\partial_t-M)\vecvarepsilon,\vecZ_i)-\bfOmega(\vecF_1,\vecZ_i)\approx \partial_t\bfOmega(\vecvarepsilon,\vecZ_i)-\bfOmega(M\vecvarepsilon,\vecZ_i)-\bfOmega(\vecF_1,\vecZ_i).
\end{split}
\end{align*}
Since $M$ has the structure that $\bfOmega(M\vecu,\vecv)=-\bfOmega(\vecu,M\vecv)$, $M\vecZ_1=o_{R_\ctf}(1)$, and $M\vecZ_2=\vecZ_1+o_{R_{\ctf}}(1)$, by imposing the orthogonality conditions $\bfOmega(\vecvarepsilon,\vecZ_i)\approx0$, we arrive at the schematic ODEs
\begin{align*}
\begin{split}
\dotxi-\ell\approx o_{R_\ctf}(\varepsilon)+\mathrm{nonlinear~terms},\quad \dotell\approx o_{R_\ctf}(\varepsilon)+\mathrm{nonlinear~terms}.
\end{split}
\end{align*}
Note that $P$, which cannot decay at a twice integrable rate as discussed above, does not appear linearly in these ODEs. Therefore, if we can show that the remainder $\varepsilon=\phi-P$ has twice integrable decay, we can still hope to close the bootstrap assumptions and obtain a twice integrable decay rate for $\dotwp$. The problem with this definition is that since $\uH$ has a non-trivial kernel, $P$ may not decay in time (in fact, it may grow). This yields non-decaying terms on the right-hand side of \eqref{eq:modelprofile3} so we cannot hope to prove favorable decay estimates for $\varepsilon$.  To rectify this, we have to guarantee that $P$ satisfies a suitable orthogonality condition of the form $\int P W\approx 0$ for all $\tau$, where $W$ is an approximation to $\fybar$, for instance $W=\chi_{\leq R_\ctf}\fybar$. To ensure this condition we modify the defining equation for $P$ as $\calP P=g_\Ext+\tilg$, where $\tilg$ is to be chosen so that the desired orthogonality condition is satisfied. Note, however, that the introduction of $\tilg$ results in errors in equations \eqref{eq:modelprofile3} and \eqref{eq:modelprofile4} and we need to make sure that these errors still allow for twice integrable decay rates for $\varepsilon$ and $\dotwp$. In particular, since $P$ itself can decay at best at the rate $\tau^{-2}$, a linear appearance of $P$ or its spatial derivatives in these equations would be a problematic. We now briefly explain how to choose $\tilg$ so that the only linear appearances of $P$ in \eqref{eq:modelprofile3} and \eqref{eq:modelprofile4} involve $\partial_\tau P$, which we can hope to have better decay. Let $\tilg=\txtg_1+\partial_\tau\txtg_2$ so that$P$ satisfies
\begin{align}\label{eq:modelprofile4alt}
\begin{split}
\calP P = g_\Ext+\txtg_1+\partial_\tau \txtg_2.
\end{split}
\end{align}
We also require that $P$ have vanishing initial data. Let $W$ and $Y$ be approximations to $\fybar$ to be specified. Then using the relation $\partial_\mu(\sqrt{|h|}(h^{-1})^{\mu\nu}(Y\partial_\nu P-P\partial_\nu Y))=\sqrt{|h|}(Y\calP P - P\calP Y)$ we get
\begin{equation}\label{eq:modelprofile5}
\int \big[ (h^{-1})^{0\nu}(Y\partial_\nu P- P\partial_\nu Y) -\txtg_2 Y\big]\sqrt{|h|}\ud x\Big\vert_{\tau_1}^{\tau_2}=\int_{\tau_1}^{\tau_2}\int (g_\Ext Y+\txtg_1 Y-P\calP Y)\sqrt{|h|}\ud x \ud \tau.
\end{equation}
Here we have used the notation $\sqrt{|h|}\ud x$ for the volume form on constant $\tau$ surfaces. Similarly, integrating $\partial_\tau (PW\sqrt{|h|})=(W\partial_\tau P+P\partial_\tau W)\sqrt{|h|}$ we get
\begin{align}\label{eq:modelprofile6}
\begin{split}
\int P W\sqrt{|h|}\ud x\Big\vert_{\tau_1}^{\tau_2}=\int_{\tau_1}^{\tau_2}\int (W\partial_\tau P+P\partial_\tau W)\sqrt{|h|}\ud x \ud\tau.
\end{split}
\end{align}
Note that except for the factor of $\txtg_2$, the integrand on the left-hand side of \eqref{eq:modelprofile5} is approximately (that is up to small but not decaying linear terms $P$) equal to the integrand on the right-hand side of \eqref{eq:modelprofile6}. The idea now is to choose $\txtg_1$ such that the right-hand side of \eqref{eq:modelprofile5} vanishes. Since $P$ has zero initial data, it follows that the left-hand side of \eqref{eq:modelprofile5} must vanish. But then we choose $\txtg_2$ so that the integral on the left-hand side of \eqref{eq:modelprofile5} is equal to the spatial integral on the right-hand side of \eqref{eq:modelprofile6}. This would imply that $\int PW \sqrt{|h|}\ud x$ must also vanish as desired. Note that by choosing $Y$ to be better approximations of the eigenfunctions than $W$, in the sense that it agree with $\fybar$ on $\{|\rho|\lesssim \tau\}$, we can ensure that $\txtg_1$ has a twice integrable decay rate. On the other hand $\txtg_2$ is only at the level of a small, but not decaying, multiple of $P$. But since only $\partial_\tau\txtg_2$, and not $\txtg_2$ itself, appears in \eqref{eq:modelprofile4alt} the only linear errors in $P$ are accompanied by a time derivative as we had hoped. See Section~\ref{sec:modprof} for a more detailed and accurate version of this argument.
\begin{remark}
Note that we have chosen $Y$ to have growing support in $\tau$, while $W$ and $\vecZ_i$ are assumed to have compact support in $\{|\rho|\lesssim R_\ctf\}$. This is because in our scheme of defining orthogonality conditions for the parameters and proving ILED it is important that the defining test functions have compact support. Indeed, the relation $R_\ctf\ll R\ll|\ell|$ is essential in our argument (see for instance Section~\ref{subsec:ILEDintro}). On the other hand, we are always free to multiply the equation with a function with non-compact support, which is why we can allow $Y$ to have growing support in $\tau$.
\end{remark}
\subsection{Uniform boundedness of energy, ILED, and vectorfield method}\label{subsec:ILEDintro}
As mentioned above one of the main challenges in the proof of ILED and vectorfield estimates is the moving center of the catenoid. Since this difficulty was already addressed in \cite{LuOS1} we refer the reader to \cite[Sections~1.7, 1.8]{LuOS1} for the related discussion. We also mention that for the uniform boundedness of energy and ILED we need the solution to satisfy suitable orthogonality conditions. These conditions are imposed by first passing to a first order formulation for the equation. This has the advantage of being a systematic approach to arriving at first order ODEs for the parameters. In this work we borrow the first order formulation developed in \cite{LuOS1}, and we refer the reader to \cite[Section~1.6]{LuOS1} for a discussion of the related challenges, including the smoothing procedure for the parameters. Here we address new issues that arise due to the slow spatial decay of the eigenfunctions. The main issue is already present in the proof of ILED for 
\begin{equation}\label{eq:ILEDintro1}
(-\partial_t^2+\uH)u=f
\end{equation}
on the product Lorentzian catenoid, so we focus on this case. Specifically, since the kernel elements decay as $|\rho|^{-2}$, they do not belong to the spatial part of the dual local energy norm (see~\eqref{eq:LEnormdef1}). Therefore, we cannot use the spectral projection to the continuous part of the spectrum of $\uH$ as in \cite{LuOS1} to deal with them in the proof of ILED. Instead we work with the new unknown $v=\Dar u$ which satisfies an equation of the form $(-\partial_t^2+\tilde{\uH})v=\tilf$, where $\tilde{\uH}$ has trivial kernel (more precisely, we pass to the transformed variable $v$ only as an intermediate step in the proof). Here $\Dar$ is the Darboux transform defined in~\eqref{eq:Darprod1}. The proof of ILED for $v$ then proceeds as in \cite{LuOS1}, which in the product case was inspired by \cite{MMT1,MST}. But, in transferring the resulting estimate to $u$ we incur large errors of the form $R_\ctf^m$ for some power $m$. See~\eqref{eq:ILEDprod2}. On the other hand, an important feature of our scheme is that the parameter derivatives enter the equation for $\varepsilon$ linearly (due to the failure of the profile to be an exact solution) and $\varepsilon$ enters the parameter ODEs linearly (because we use truncated eigenfunctions in imposing orthogonality conditions). The circularity is broken because the linear appearance of $\varepsilon$ in the parameter ODEs comes with a small factor in terms of inverse powers of $R_\ctf$.  This gain is not sufficient to compensate for the loss in inverting the Darboux transformation. The key observation is that the part $f$ in \eqref{eq:ILEDintro1} that carries the linear contribution of the parameters, let us call it $f_2$, appears in the right-hand side of the ILED estimate~\eqref{eq:ILEDprod2} (and using the notation there) as
\begin{equation}\label{eq:ILEDintro2}
\|f_2\|_{LE^\ast}+R_{\ctf}^{m}\|\Dar f_2\|_{LE^\ast}. 
\end{equation}
That is, the contribution that gets multiplied by a growing factor of $R_\ctf$ is $\Dar f_2$. Now from the discussion in Section~\ref{sec:refprofintro1} we know that $f_2$ is essentially of the form $\dotwp \fybar_i$, where $\dotwp$ denotes a parameter derivative, and the Darboux transform is designed so that $\Dar \fybar_i=0.$\footnote{This is analogous to the vanishing of $\bbP\fybar_i$ where $\bbP$ is the spectral projection to the continuous spectrum of $\uH$.} But, to be precise, as discussed in Section~\ref{sec:refprofintro1} (see also Lemma~\ref{lem:calG01}) this structure for $f_2$ is valid only in the flat part of the foliation where $|\rho|\leq R-C$, for some absolute constant $C$, while $f_2$ has stronger spatial decay in the hyperboloidal part of foliation defined in Section~\ref{subsec:mainprofile}. Recall that $R$ denotes the radius at which the foliation transitions from flat to hyperboloidal. Going back to \eqref{eq:ILEDintro2} we conclude that this contribution is bounded by
\begin{align*}
\begin{split}
\|\dotwp\|_{L^2_\uptau}(R^{a}+R_{\ctf}^m R^{-b})
\end{split}
\end{align*}
for some positive exponents $a$, $b$, and $m$. Now the point is that $R\gg1$ can be chosen much larger than $R_\ctf\gg1$, and that the exponent of $R^a$ can be chosen to be arbitrarily small (in fact, if we work with scale invariant local energy spaces, we would have $\log R$ instead of $R^a$). On the other hand, from the ODEs for the parameters we can prove that $\|\dotwp\|_{L^2_\tau}$ is bounded by $R_\ctf^{-c}$ times the local energy norm of the solution. Therefore, by choosing the right balance between  $R$ and $R_\ctf$ we are able to close the energy boundedness and ILED estimates. The right balance between $R$ and $R_\ctf$ for us turns out to be such that $R^{\frac{\alpha}{2}}R_\ctf^{-1+\frac{\alpha}{2}},R_\ctf^{2+\frac{\alpha}{2}}R^{-1+\frac{\alpha}{2}}\ll1$, where $0<\alpha\ll1$ is the defining exponent in the local energy norm in \eqref{eq:LEnormdef1}.
\subsection{Improved time decay for $\varepsilon$}\label{subsec:tailintro}
In the Section~\ref{sec:refprofintro1} we explained that $P$ cannot decay faster than $\tau^{-2}$. Here we describe how we obtain a twice integrable decay rate for $\varepsilon$. Consider the equation
\begin{align}\label{eq:taildecayintro0}
\begin{split}
\calP \varepsilon=f
\end{split}
\end{align}
where $f$ is assumed to have sufficient decay in $|\rho|$ and $\tau$. We first describe the general scheme which is essentially from \cite{MTT} (see also \cite{LuOh}). We assume that we have already established some suboptimal decay rates for $\varepsilon$, for instance using the $r^p$ vectorfield method. For the sake of the current discussion let us suppose that we already have $|\varepsilon|+|\tau\partial_\tau\varepsilon|\lesssim \tau^{-\frac{5}{4}}$. We then first consider the exterior where we view the difference between the Minkowski wave operator $\Box_m$ and $\calP$ perturbatively and conclude from \eqref{eq:taildecayintro0} that
\begin{align*}
\begin{split}
\Box_m\chi_{\geq \rho_\ast} \varepsilon = \chi_{\geq\rho_\ast}f+(\Box_m-\calP)(\chi_{\geq \rho_\ast}\varepsilon)+[\calP,\chi_{\geq\rho^\ast}]\varepsilon.
\end{split}
\end{align*}
Here $\rho_\ast$ is a large constant and we are restricting attention to one asymptotic end. Using this equation and the known decay estimates on $\varepsilon$ we can prove that
\begin{align}\label{eq:taildecayintro1}
\begin{split}
|\chi_{\geq\rho_\ast}\varepsilon|+|\chi_{\geq\rho_\ast}\tau\partial_\tau\varepsilon|\lesssim \rho^{-1}\tau^{-\frac{5}{4}}.
\end{split}
\end{align}
In the \cite{MTT} scheme this would be achieved by using the positivity of the fundamental solution to $\Box_m$ in spatial dimension three, and Sobolev estimates on $\bbS^2$, to reduce the problem to the radial case and then using integration along characteristics. Here, we use Lemma~\ref{eq:Boxmhugens1} instead (which is inspired by the intermediate zone analysis in \cite{LuOh}), which in particular has the advantage that it easily applies to our scenario where spatial decay is measured from the moving center $\xi(\tau)$. In both cases, a key role is played by the sharp Huygens principle and it is important that $f$ has faster than $|\rho|^{-3}$ spatial decay. Note that \eqref{eq:taildecayintro1} already gives us a twice integrable decay rate in the region $\{|\rho|\gtrsim\tau\}$. For the region $\{|\rho|\lesssim \tau\}$ we use the elliptic estimates from Proposition~\ref{prop:uH}, specifically equation~\eqref{eq:uH-wSob}, and the Sobolev estimate $\|u\|_{L^\infty}\lesssim \|u\|_{\ell^\infty \calH^{s,-\frac{3}{2}}}$, $s>\frac{3}{2}$ (see Definition~\ref{def:w-Sob}). We then need to estimate the right-hand side of equation~\eqref{eq:uH-wSob}, with $\phi$ replaced by $\chi_{\lesssim \tau}\varepsilon$. Here to apply Proposition~\ref{prop:uH} we move the terms with time derivatives in $\calP (\chi_{\lesssim \tau}\varepsilon)$ to the right-hand side of the equation. The contribution of $f$ in \eqref{eq:taildecayintro0} then gives the desired estimate as long as $f$ has sufficient $\tau$ and $\rho$ decay. The commutator $[\calP,\chi_{\lesssim \tau}]$ and the terms in  $\calP (\chi_{\lesssim \tau}\varepsilon)$ which come with extra time derivatives can be bounded using the estimates for $\chi_{\gtrsim \rho_\ast}\varepsilon$ established in the first step and the improved time decay of $\partial_\tau\varepsilon$ which we had assumed from the $r^p$ vectorfield method.

Besides the extra care needed to carry out this scheme in the presence of a moving center, there are a few more points that deserve extra attention. First, in view of the discussion in Section~\ref{sec:refprofintro1}, in the region $\{|\rho|\leq R+C\}$ the part of $f$ in \eqref{eq:taildecayintro0} containing the source term is approximately of the form $\dotwp \fybar_i$. In the notation of Proposition~\ref{prop:uH} we treat this part of $f$ as $f_1$ in estimate~\eqref{eq:uH-wSob} and use the bootstrap assumptions on $\dotwp$. The circularity is broken by the same reason as for the ILED estimate, where we use $\Dar\calS_1\fybar_i=0$ and the balance $R^{\frac{\alpha}{2}}R_\ctf^{-1+\frac{\alpha}{2}},R_\ctf^{2+\frac{\alpha}{2}}R^{-1+\frac{\alpha}{2}}\ll1$. Similarly the contribution of the last term on the right-hand side of \eqref{eq:uH-wSob} is estimated by a small multiple of $\|\varepsilon\|_{L^\infty}$ using our orthogonality conditions. Finally, $\partial_\tau P$ appears to linear order, with no extra smallness, through $\partial_\tau \txtg_2$ (see Section~\ref{sec:refprofintro1}) on the right-hand side of \eqref{eq:taildecayintro0}. Here we crucially use the fact that we have already closed the vectorfield estimates for $\varepsilon$ and $P$. In particular $\partial_\tau P$ already has twice integrable decay rate coming from the $r^p$ vectorfield estimates.
\begin{remark}
In \cite{Moschidis1} the $r^p$ vectorfield method already gives the decay rate $\tau^{-\frac{3}{2}}$ in spatial dimension~$3$. In principle we can obtain this decay rate for $P$ and $\varepsilon$ using the $r^p$ vectorfield arguments in this paper. However, to avoid technical complications we have only proved the weaker decay rate $\tau^{-\frac{5}{4}+\frac{\kappa}{2}}$ where $\kappa$ is an arbitrary fixed small constant. See \cite[Section~1.8]{LuOS1} for an explanation of this restriction.
\end{remark}
\subsection{Second statement of the main theorem}
To state the main theorem we consider two functions
\begin{align}\label{eq:initialdata1}
\begin{split}
&\psi_0,\psi_1\in C^\infty_0(\barcalC),\qquad \supp~ \psi_0,\psi_1\subseteq \barcalC\cap\{|\Xbar|<R/2\},\\
&\sum_{j=0}^M(\|\jap{\rho}^{1+j}\partial_\Sigma^{1+j}\psi_0\circ F\|_{L^2(\sqrt{|g_\barcalC|}\ud\omega\ud\rho)}+\|\jap{\rho}^{1+j}\partial_\Sigma^j\psi_1\circ F\|_{L^2(\sqrt{|g_\barcalC|}\ud\omega\ud\rho)}\leq 1.
\end{split}
\end{align} 
Here $M$ is a fixed large constant (corresponding to the number of derivatives we commute), and $\partial_\Sigma$ denotes either $\partial_\rho$ or $\jap{\rho}^{-1}\partial_\omega$, with $\partial_\omega$ a unit size derivative on $\bbS^2$. Using the notation introduced in Sections~\ref{subsec:Riemcat} and~\ref{subsec:mainprofile}, consider
\begin{align}\label{eq:initialdata2}
\begin{split}
\Phi_0[\psi_0]= F+(\psi_0\circ F) N\circ F,\qquad \Phi_1[\psi_1]=(1,0)+(\psi_1\circ F)N\circ F.
\end{split}
\end{align}
We also let $\tilvarphi_\mu:=\chi \varphibar_\mu$ where the cutoff function $\chi$ is equal to one on $\barcalC\cap\{|\Xbar|<R/3\}$ and supported on $\barcalC\cap\{|\Xbar|<R/2\}$. The codimension one condition in Theorem~\ref{thm:main} below can be interpreted as follows. First we consider functions in the ball of radius $\epsilon$ in the norm \eqref{eq:initialdata1} that satisfy a codimension one condition. This condition is the vanishing of a certain functional given in \eqref{eq:codim1}. As the exact form of this functional is a bit complicated to state at this point, we defer this until Section~\ref{sec:bootstrap}. The initial data for Theorem~\ref{thm:main} are then parameterized, by a function $b_0$, as a graph over this codimension one set. If $b_0$ is a $C^1$ function then the resulting set can be viewed as a codimension one manifold in the topology of \eqref{eq:initialdata1}. But, we do not investigate the regularity of the function $b_0$ in this work. 
\begin{theorem}\label{thm:main}
Consider $\Phi_0$, $\Phi_1$ as in \eqref{eq:initialdata1}, \eqref{eq:initialdata2}, and assume that $(\psi_0,\psi_1)$ satisfy \eqref{eq:codim1}. If $\epsilon\geq0$ is sufficiently small, then there exist $b_0\in\bbR$ with $|b_0|\lesssim 1$ and  $\Phi:\bbR\times I \times \bbS^{2}\to \bbR^{1+4}$ satisfying \eqref{eq:HVMC1}, such that $\Phi\vert_{\{t=0\}}=\Phi_0[\epsilon(\psi_0+b_0\tilvarphi_\mu)]$ and $\partial_t\Phi\vert_{\{t=0\}}=\Phi_1[\epsilon(\psi_1-\mu b_0\tilvarphi_\mu)]$. Moreover, there exist $\ell,\xi:\bbR\to \bbR^{3}$ satisfying  $|\ell|,|\dotxi|\lesssim \epsilon$ and
\begin{align*}
\begin{split}
|\dotell(\sigma)|, |\dotxi(\sigma)-\ell(\sigma)|\to0\qquad \mathrm{as}~\sigma\to\infty,
\end{split}
\end{align*}
such that the image of $\Phi$ can be parameterized as
\begin{align*}
\begin{split}
\cup_\sigma\Sigma_\sigma\ni p\mapsto p+\psi(p)N(p),
\end{split}
\end{align*}
with $\|\psi\|_{L^\infty(\Sigma_\sigma)}\to0$ as $\sigma\to\infty$. More precisely, there exists a positive $\kappa\ll1$ such that
\begin{align*}
\begin{split}
|\dotell(\sigma)|,|\dotxi(\sigma)-\ell(\sigma)|\lesssim \epsilon \sigma^{-\frac{9}{4}+\frac{\kappa}{2}},\qquad \|\psi\|_{L^\infty(\Sigma_\sigma)}\lesssim \epsilon \sigma^{-\frac{5}{4}+\frac{\kappa}{2}},\qquad \mathrm{as}~\sigma\to\infty.
\end{split}
\end{align*}
\end{theorem}
More precise decay estimates on $\psi$ and the parameters can be found in Propositions~\ref{prop:bootstrappar1},~\ref{prop:bootstrapphi1}, and~\ref{prop:boostrapvareptail}. We now make a few remarks about Theorem~\ref{thm:main}.
\begin{remark}\label{rem:parametersintro1}
It follows from the decay rate of $\dotell$ and $\dotxi-\ell$ that there exist $\abar,\ellbar\in\bbR^n$ such that $\ell(\tau)\to\ellbar$ and $\xi(\tau)\to \abar+ \ellbar \tau$ as $\tau\to\infty$. In this sense our theorem implies that the solution approaches a fixed, boosted and translated Lorentzian catenoid. The differential equations governing the evolution of the parameters are derived in Section~\ref{sec:eqnsofmotion}.
\end{remark}
\begin{remark}\label{rem:codim1}
As discussed earlier the codimension one condition of the data in Theorem~\ref{thm:main} is optimal, but we do not pursue the question of regularity of $b_0$.  As a result, we cannot infer that the set of initial data, considered in Theorem~\ref{thm:main} form a codimension one \emph{submanifold} in any topology. See also \cite[Remark~1.7]{LuOS1} for further discussion.
\end{remark}
\begin{remark}
In our proof we decompose $\psi$ further into a \emph{refined} profile $P$ and a remainder as $\varepsilon$ as $\psi\approx P+\varepsilon$. The refined profile contains the leading contribution of $\psi$ and we prove the decay rate $\|P\|_{L^\infty(\Sigma_\sigma)}\lesssim \epsilon \sigma^{-\frac{5}{4}+\frac{\kappa}{2}}$. See Proposition~\ref{prop:bootstrapphi1}. For the remainder we prove the improved decay $\|\varepsilon\|_{L^\infty(\Sigma_\sigma)}\lesssim \epsilon \sigma^{-\frac{9}{4}+\frac{\kappa}{2}}$. See Proposition~\ref{prop:boostrapvareptail}. We expect that by first improving the decay of higher order time derivatives, our methods can give the almost sharp rates $\|P\|_{L^\infty(\Sigma_\sigma)}\lesssim \epsilon \sigma^{-2+\cdot}$ and $\|\varepsilon\|_{L^\infty(\Sigma_\sigma)}\lesssim \epsilon \sigma^{-3+\cdot}$.
\end{remark}

\subsection{Further discussions}
Here we briefly discuss some related works and refer the reader to \cite[Section~1.10]{LuOS1} for a more in depth account. Concerning the HVMC equation, \cite{AC2,Ettinger, AIT21,Wong1} address questions related to the local well-posedness, and \cite{B1, Lin1, Stefanov11, Wong17} consider the nonlinear stability of hyperplanes under the HVMC evolution. As discussed earlier, within radial symmetry, the nonlinear stability of the Lorentzian catenoid was studied in \cite{KL1, DKSW} and that of the Lorentzian helicoid in \cite{Marchali1}.
The nonlinear stability of simple planar traveling wave solutions was established in \cite{AW20}. For the study of singularity formation, we refer to \cite{NT13, JNO15, Wong18, BMP1}. 
For further discussion on the physical significance of the HVMC equation, see \cite{AC1, AC2, Hoppe13}.
The global stability problems for Lorentzian constant positive mean-curvature flow has also been explored in \cite{Wong2}.

While the underlying PDEs are significantly different, our main result can be formally compared to recent remarkable works \cite{DHRT1,KlSz1,KlSz2, KlSz-GCM1, KlSz-GCM2, GiKlSz1, GiKlSz2, Sh} on the nonlinear asymptotic stability of Kerr and Schwarzschild black holes. These black holes are stationary solutions to the vacuum Einstein equation, a (3+1)-dimensional quasilinear wave equation. Our problem is simpler in several ways, including the gauge choice (compare our choice described in Section~\ref{subsec:mainprofile} with \cite{DHRT1, KlSz1, KlSz2}) and the analysis of the linearized problem (compare the discussion in Section~\ref{subsec:Riemcat} with \cite{HHV1,DHR1,ABBM1,HKW1}). Nevertheless, in this paper and \cite{LuOS1} we satisfactorily resolve a key issue that is shared by many soliton stability problems, but not with the black hole stability problem -- this is the issue of \emph{modulation of the translation and boost parameters}. In our problem, the stationary solution is defined on a natural ambient spacetime, and it is an important problem to track the evolution of these parameters in relation to this spacetime. In general relativity, there is no such ambient spacetime, and the analogous issue in the black hole stability problem is addressed through gauge choice. Given the prevalence of this issue in soliton stability problems, we hope that our ideas can be applied to other related problems.

Finally, there is an extensive literature on the stability of solitons for \emph{semilinear} dispersive equations. For those interested, we recommend the survey articles by Kowalczyk--Martel--Mu\~{n}oz  \cite{KMM17} and Tao \cite{Tao09}. Our focus in this paper and \cite{LuOS1} is on the challenges posed by the \emph{quasilinearity} of the equation. This focus, in turn, is motivated by the conjectured asymptotic stability of certain celebrated topological solitons, such as the Skyrmion for the Skyrme model \cite{MantonSutcliffe}.

\subsection{Outline of the Paper}\label{subsec:introremainingoutline}
The remainder of this paper is organized as follows. Section~\ref{sec:prelim} contains the notation and some preliminary results including a notation list in Section~\ref{sec:notation}, a discussion of the local theory in Section~\ref{sec:LWP}, and the proof of the ILED for the product catenoid and various elliptic estimates in Section~\ref{sec:ILEDproduct}. In Section~\ref{sec:profile1} we record the form of the equation in the first and second order formulations. We also introduce the different coordinate systems used in the rest of the paper and write out the form of the main operators in these coordinates. In Section~\ref{sec:modprof} we introduce the modified profile $P$. Section~\ref{sec:finaldecomp} contains the final decomposition of the solution in terms of the main profile, the modified profile, and the remainder. Here we also impose the orthogonality conditions on the remainder and derive the modulation equations for the parameters. The bootstrap assumptions are stated in Section~\ref{sec:bootstrap}, where, in Propositions~\ref{prop:bootstrappar1},~\ref{prop:bootstrapphi1}, and~\ref{prop:boostrapvareptail} we give more precise decay estimates than the ones given in Theorem~\ref{thm:main}. In this section we also show how Theorem~\ref{thm:main} follows from the bootstrap propositions. The analysis of the modulation ODEs leading to the proof of Proposition~\ref{prop:bootstrappar1} are contained in Section~\ref{sec:ODEanalysis}. The ILED estimate is proved in Section~\ref{sec:ILED} and the $r^p$ vectofield estimates and the proof of Proposition~\ref{prop:bootstrapphi1} are contained in Section~\ref{sec:VF}. Finally, the improved late time tail estimates of Proposition~\ref{prop:boostrapvareptail} are proved in Section~\ref{sec:tails}. In Appendix~\ref{sec:Appendix} we record various relevant computations for the Minkowski metric $m$. With the exception of Section~\ref{sec:Appendixtail},  which contains a key ingredient of the proof of Proposition~\ref{prop:boostrapvareptail}, the results in Appendix~\ref{sec:Appendix} are borrowed from \cite{LuOS1}.

\section{Preliminaries}\label{sec:prelim}

\subsection{Notation and conventions}\label{sec:notation}
Here we collect some of the notation and conventions that are used repeatedly in this work. This is meant as a reference for the reader, and some of the precise definitions will appear only later in the paper. Some of the notation and conventions which are used more locally in various parts of the paper do not appear in this list.
 \subsubsection{\underline{The profile and the main variables}} $\barcalC$ denotes the Riemannian catenoid with its standard embedding in $\bbR^{4}$, and $\calC = \bbR \times \underline{\calC}$ the product Lorentzian catenoid in $\bbR^{1+4}$. $\xi$ and $\ell$ denote the translation and boost parameters, respectively, and in our applications always satisfy $|\ell|,|\dotxi|\ll1$. The main profile for the boosted and translated catenoid is $\calQ:=\cup_\sigma\Sigma_\sigma$, $\Sigma_\sigma=\calC_\sigma\cap\Sigma_\sigma$, where  $\calC_\sigma$ and $\bsUpsigma_\sigma$ are defined in \eqref{eq:calCsigmadef1} and \eqref{eq:bsUpsugmasigmadef1}.  We will denote the flat and hyperboloidal parts of the profile by $\calC_{\flatt}:=\{X^0\geq \sigma_\temp(X)+\delta_1\}$ and $\calC_{\hyp}:=\{X^0\leq \sigma_\temp(X)-\delta_1\}$ respectively. Here $\delta_1$ is a fixed constant and $\sigma_\temp$ is defined in Section~\ref{subsec:mainprofile}. We often refer to the region inside a large compact set  as the interior and to the complement as the exterior. The perturbation $\psi:\cup_\sigma\Sigma_\sigma\to \bbR$ is defined in \eqref{eq:psidef1}, where $N:\cup_\sigma\Sigma_\sigma\to \bbR^{1+4}$ denotes the almost normal vector to the profile and is defined in \eqref{eq:Ndef1}. In the first order formulation introduced in Section~\ref{subsubsec:firstorder}, the vector form of the perturbation is denoted by $\vecpsi=(\psi,\dotpsi)^\intercal$, where $\dotpsi$ is the momentum variable and roughly corresponds to the time derivative of $\psi$. The linearized operator for the vanishing mean curvature equation about the catenoid solution is denoted by $\uH:=\Delta_\barcalC+|\secondff|^2$ (see Section~\ref{subsec:Riemcat}) where $\Delta_\barcalC$ is the Laplacian on $\calC$ and $\secondff$, which we sometimes also denote by $V$,  is the second fundamental form. $\uH$ has a three dimensional kernel and one positive eigenvalue $\mu^2$. The corresponding eigenfunctions are denoted by $\fybar_i$, $i=1,2,3$ and $\fybar_\mu$. In the first order formulation there are two projection coefficients associated with $\mu$: $a_{+}$ denotes the unstable (growing mode) coefficient, and $a_{-}$ the stable (decaying mode) coefficient. The remainder, after subtracting the contribution of the corresponding eigenfunctions from $\vecpsi$ is denoted by $\vecphi$ at the vector level (in the first order formulation) and by $\phi$ at the scalar level (see Section~\ref{subsubsec:firstorder}). The refined profile $P$, with the corresponding vector form $\vecP$, are defined in Section~\ref{sec:modprof} and capture the most slowly decaying (in time) part of the perturbation. Accordingly we write $\varepsilon=\phi-P$.

 \subsubsection{\underline{Parameter derivatives}} We will use $\dotwp$ to denote the parameter derivatives $\dotell$ and $\dotxi-\ell$. When used as a vector, $\dotwp=(\dotell,\dotxi-\ell)^\intercal$ in that order. When used schematically, for instance in estimates or to denote dependence on parameter derivatives, the order will not be important, so for example $O(\dotwp)$ denotes terms that are bounded by $|\dotell|$ or $|\dotxi-\ell|$. The distinction will be clear from the context. More generally, $\dotwp^{(k)}$ denotes a total of $k$ derivatives of the parameters, so for instance $\dotwp^{(2)}$ could be any of $\ddotell$, $|\dotxi-\ell|^2$, $\ddotxi-\dotell$, etc. $\dot{\wp}^{(\leq k)}$ denotes a total of up to $k$, but at least one, parameter derivatives. $\wp^{(\leq k)}$ denotes a total of up to $k$ parameter derivatives, but possibly also an undifferentiated $\ell$. We sometimes also use the notation $\wp$ for $\ell$. Note that $\xi$ itself cannot be written as $\wp$ ($\xi$ is expected to grow linearly in time), but $\dotxi$ can be written as $\dotxi=\dotxi-\ell+\ell$, which is a sum of terms of the form $\dotwp$ and $\wp$. A similar notation is used for $\dota_{\pm}^{(k)}$, $a_{\pm}^{(\leq k)}$, etc. 
 \subsubsection{\underline{Constants}} $\epsilon$ is the smallness parameter for the size of the initial perturbation.  $\kappa$ is a small positive absolute constant which arises in the decay rates in the bootstrap argument; see Section~\ref{sec:bootstrap}. In our estimates, the parameters and the perturbation $\varepsilon$ depend linearly on each other . What breaks the circularity is that the linear appearance of the $\varepsilon$ in the estimates for the parameter derivatives is always accompanied by a small (but not decaying) constant. This small constant is denoted by $\delta_\wp$ in the bootstrap assumptions of Section~\ref{sec:bootstrap}. The final time of the bootstrap interval is denoted by $\tau_f$. There are also a few large radii that appear in our arguments. $R\gg1$ is  a large constant such that the initial data are supported in $\barcalC\cap\{|\Xbar|<R/2\}$; see \eqref{eq:initialdata1}. More importantly, the transition region from the flat  to hyperboloidal parts of the foliation happens in the region $R-C\leq |\Xbar| \leq R+C$ for some absolute constant $C$; see Section~\ref{subsec:mainprofile}. The constant $R_\ctf\gg1$ is a constant such that the support of the test functions, which we use as proxies for the eigenfunctions $\fybar_i$ of $\uH$ is contained in $\{|\rho|\leq R_\ctf\}$. We will choose $R_\ctf$ such that $R^{\frac{\alpha}{2}}R_\ctf^{-1+\frac{\alpha}{2}},R_\ctf^{2+\frac{\alpha}{2}}R^{-1+\frac{\alpha}{2}}\ll1$. Here the constant $0<\alpha\ll1$ is used in the definition of the local energy norm. See equation~\eqref{eq:LEnormdef1}. The size of the data, $\epsilon$, is considered small relative to any inverse power of $\Reigenfunctioncutoffscale$. In particular, since in view of the bootstrap assumptions in Section~\ref{sec:bootstrap} we have $|\ell|\lesssim \epsilon$, quantities such as $(\Reigenfunctioncutoffscale)^{m_1}R^{m_2}|\ell|$ are considered small, for any powers $0<m_i<100$. The smallness of the constant $\delta_\wp$ above is in terms of $\ell$ and inverse powers of $\Reigenfunctioncutoffscale$ and $R$. 
 
 \subsubsection{\underline{Row and columns vectors}} When there is no risk of confusion, we identify row and column vectors. For instance, we use both $(\psi,\dotpsi)$ and $(\psi,\dotpsi)^\intercal$ for $\vecpsi$.
 
 \subsubsection{\underline{Coordinates, derivatives, and vectorfields}}\label{subsec:prelimvfs} We use four sets of coordinate systems: $(t,\rho,\omega)$ are used in the interior. $(\tau,r,\theta)$ are used in the exterior. $(\uptau,\uprho,\uptheta)$ are referred to as the non-geometric global polar coordinates and are global coordinates that agree with $(t,\rho,\omega)$ and $(\tau,r,\theta)$ in the respective regions. $(\tiluptau,\tiluprho,\tiluptheta)$ are referred to as the geometric global coordinates, and are define with respect to the values of the parameters at a fixed time. Their purpose is that the main part of the linear operator with coefficients defined with respect to the fixed-time parameters takes the product form $-\partial_\uptau^2+\uH$ in these coordinates. The relevant properties of these coordinates are recorded in Section~\ref{subsubsec:2ndordereq}. $\RbfT$ denotes the global almost stationary vectorfield, which in terms of the global non-geometric coordinates is given by $\partial_{\uptau}$. In general $\partial$ denotes arbitrary derivatives that have size of order one, and $\partial_\Sigma$ the subset of these derivatives that are tangential to the leaves of the foliations. In the exterior region, $\tilpartial_\Sigma$ denotes derivatives which can be written as a linear combination of $\partial_\Sigma$ and $r^{-1}\RbfT$, with coefficients of size of order one. In general we denote the number of derivatives by a superscript. For instance $\partial_\Sigma^{\leq k}$ means up to $k$ tangential derivatives. There are also a few commutator and multiplier vectorfields which are used in the exterior in Section~\ref{sec:VF} in the context of proving decay estimates for the perturbation. The precise definitions are given in Appendix~\ref{subsec:appMink}, but we give a brief description here: $L$ and $\Lbar$ are the outgoing and incoming almost null vectorfields. $\Omega$ is the rotation vectorfield. $T$, which is comparable and almost colinear with $\RbfT$, is defined by $T=\frac{1}{2}(L+\Lbar)$. In the exterior region where these vectorfields are defined we use $\tilde{r}L$, $\Omega$, and $T$ as commutators, and use  $X^k$ (when $k=1$ we simply write $X$) to denote an arbitrary string of $k$ such vectorfields. Here $\tilr$ is a geometric radial variable introduced in Appendix~\ref{subsec:appMink}. 
 
 \subsubsection{\underline{Volume forms}} In general we use $\ud V$ to denote the induced volume form from the ambient space $\bbR^{1+4}$. If there is any risk of confusion we use a subscript to denote the subset on which the volume form is induced (for instance $\ud V_S$ for the subset $S$). When working in a fixed set of coordinates we sometimes write out the volume form explicitly. In the exterior region, it is sometimes more convenient to use the coordinate volume form for the Minkowski metric rather than the geometric induced one. It will be clear from the bootstrap assumptions that these two volume forms are comparable and therefore various norms defined with respect to them are equivalent. The volume form on the standard unit sphere will be denoted by $\ud\theta$ or $\ud S$ interchangeably (or $\ud \omega$, $\ud \uptheta$, etc, depending on the coordinate system we are using). Sometimes we also use $\ud x$ or $\ud y$ for the volume form on $\Sigma_\tau$.
\subsubsection{\underline{Cutoffs}} We use the notation $\chi$ for smooth cutoff functions defined on $\cup_\sigma\Sigma_\sigma$ and taking values in $[0,1]$. We may denote the set on which $\chi$ is equal to one by a subscript. For instance $\chi_S$ is equal to one on $S$ and equal to zero outside of a neighborhood of $S$ (we will make the support more precise when needed). For a positive number $c$, $\chi_{\geq c}$ denotes a cutoff which is one in the region $\{|\uprho|\geq c\}$ and equal to zero outside of $\{|\uprho|\geq \frac{c}{2}\}$. Here $\uprho$ is the radial coordinate from the global non-geometric coordinates in Section~\ref{subsubsec:2ndordereq}. $\chi_{<c}$, $\chi_{a \leq \cdot \leq b}$, etc., are defined similarly.
\subsubsection{\underline{The normal and decay of eigenfunctions}}\label{subsubsec:normal} For the standard Riemannian catenoid as described in Section~\ref{subsec:Riemcat}, and with the notation used there, the normal vector is given by 
\begin{align*}
\begin{split}
\nu\equiv \nu(z,\omega)=(\frac{\Theta(\omega)}{\jap{z}^{2}},\sqrt{1-\jap{z}^{-4}}).
\end{split}
\end{align*}
As mentioned in Section~\ref{subsec:Riemcat}, the first three components, $\nu^i=\frac{\Theta^i}{\jap{z}^{2}}$, $i=1,2,3$,
appear as eigenfunctions of the main linearized operator \Hbar. It is useful to keep in mind that these have decay $\jap{z}^{-2}$ and satisfy
\begin{align*}
\begin{split}
\int \nu^i\nu^j \sqrt{|g_\barcalC|}\ud \omega\ud z = C\delta^{ij},
\end{split}
\end{align*}
where $\delta^{ij}$ is the Kronecker delta and $C$ is a constant of order one. We also remark that since the metric is asymptotically flat, the eigenfunction $\varphibar_\mu$ from Section~\ref{subsec:Riemcat} is exponentially decaying.
\subsubsection{\underline{Two asymptotic ends}} Many of the estimates and identities in this work are derived only near one of the asymptotic ends of the solution. In all cases, the other asymptotic end can be treated in exactly the same way, possibly with a change of overall sign. This remark applies in particular to many of the vectorfied identities and estimates, for instance in Section~\ref{sec:VF}. 
 \subsubsection{\underline{The Darboux transform}} As mentioned earlier the elements in the kernel of $\uH$ do not belong to the spatial part of the dual local energy norm. To deal with this we use the Darboux transform to pass to an operator with trivial kernel. The Darboux transformation is denoted by $\Dar$ and is defined in \eqref{eq:Darprod1} for the product catenoid and in \eqref{eq:Dardef1} in the general case.
 \subsubsection{\underline{Exterior parametrization over a hyperplane}}\label{subsubsec:calOnotation} Outside a large compact set, we can parameterize each asymptotic end of the solution as a graph over a hyperplane (for instance the hyperplanes $\{X^{n+1}=\pm S\}$). The function giving this parameterization for the Riemannian catenoid is denoted by $Q$. We use $Q_\wp$ to denote the corresponding function when taking into account the boost and translation parameters, although when there is no risk of confusion we drop $\wp$ from the notation and simply write $Q$. See Section~\ref{subsubsec:2ndordereq}. In this region, the catenoid and Minkowski metrics are close. The latter is denoted by $m=-\ud x^0\otimes\ud x^0+\sum_{i=1}^3\ud x^i\otimes \ud x^i$. Expressions for $m$ in various coordinates are recorded in Appendix~\ref{subsec:appMink}.

\subsubsection{\underline{The $O$ and $\calO$ notation}} The notation $f=O(g)$ is used as usual to mean $|f|\leq C|g|$ for some constant $C$. The notation $f=o_{\alpha}(g)$ is also used in the usual way to mean that $|f|/|g|$ goes to zero as the parameter $\alpha$ approaches a limiting value which will be clear from the context (usually zero or infinity). We will also use the notation $\calO$ in the following way. An error term of the form $\calO(f)$ is still bounded by $|f|$ after any number of differentiations by $\tilr L$ or $\Omega$ in the exterior (see Section~\ref{subsec:prelimvfs}), and by $\sum_{j\leq k}|\dotwp^{(j)}\RbfT^{k-j}f|$ after $k$ differentiations by $\RbfT^k$ globally. For instance, an error that is denoted by $\calO(\dotwp)$ will still be bounded by $\calO(\dotwp)$ after applications of $\tilr L$ and $\Omega$ in the exterior, and by $\calO(\dotwp^{(k+1)})$ after $k$ applications of $\RbfT$ globally. Even though we start using this notation already in Section~\ref{sec:profile1}, the corresponding properties of these error terms follow only after the bootstrap estimates are stated in Section~\ref{sec:bootstrap}. In general, the dependence on parametersis denoted by a subscript in the $O$ and $\calO$ notation, for instance as $O_{R_\ctf}(|\ell|)$.

\subsection{Local existence}\label{sec:LWP}
We recall the local theory for the HVMC from \cite{LuOS1}, which in turn was based on \cite{AC1,AC2,Luk1}. To state the needed results we introduce some notation. Recall the foliation $\cup_\sigma(\calC_\sigma\cap\bsUpsigma_\sigma)$ from Section~\ref{subsec:mainprofile}. Given $|\ell_0|<1$ and $\xi_0$ let $\xi(\tau)=\xi_0+\ell_0\tau$ and denote the corresponding submanifolds of $\bbR^{1+4}$ by $\mathring\calC_\sigma(\ell_0,\xi_0)$ and $\mathring\bsUpsigma_\sigma(\ell_0,\xi_0)$. Similarly, let $\ringN_{\ell_0,\xi_0}$ denote the corresponding almost normal vector from Section~\ref{subsec:mainprofile}. We equip each leaf of $\ringSigma_\sigma(\ell_0,\xi_0):=\mathring\calC_\sigma(\ell_0,\xi_0)\cap\mathring\bsUpsigma_\sigma(\ell_0,\xi_0)$ with the Riemannian metric induced from the ambient space and let $\calD_0^\tau:=\cup_{\sigma\in[0,\tau)}\ringSigma_\sigma(\ell_0,\sigma_0)$. The coordinate $\rho$ is used to denote the distance along $\ringSigma_\sigma(\ell_0,\xi_0)$ to $\Xbar=\xi_0+\sigma\ell_0$ with respect to the induced metric. Size one derivatives tangential to $\ringSigma_\sigma(\ell_0,\xi_0)$ are denote by $\ringpartial_\Sigma$ and the restriction of $\frac{\partial}{\partial X^0}+\ell_0$ to $\calD_0^\tau(\ell_0,\xi_0)$ is denoted by $T$. The following  result follows the standard theory. 
\begin{proposition}\label{prop:LWP}
Let  $\ringPhi_0(p)=p+\ringepsilon\ringpsi_0 \ringN_{\ell_0,\xi_0}$, $\ringPhi_1= (1,\ell_0)+\ringepsilon\ringpsi_1 \ringN_{\ell_0,\xi_0}$, where $\ringpsi_j$ are smooth functions on $\ringSigma_0(\ell_0,\xi_0)$ with $\|\jap{\rho}^{k-1}\ringpartial_\Sigma^k\ringpsi_0\|_{L^2(\Sigma_0(\ell_0,\xi_0))}$ and $\|\jap{\rho}^{k-2}\ringpartial_\Sigma^{k-1}\ringpsi_1\|_{L^2(\Sigma_0(\ell_0,\xi_0))}$ finite for~$k\leq M$, $M$ sufficiently large. If $|\ell_0|$ and $\ringepsilon>0$ are sufficiently small, then there exists $\tau\gtrsim1$ and a unique smooth function $\ringpsi:\calD_0^{\tau}(\ell_0,\xi_0)\to\bbR$, such that $\ringPhi:\calD_0^{\tau}(\ell_0,\xi_0)\to\bbR^{1+4}$ defined by
\begin{align}\label{eq:lwppar}
\begin{split}
\ringPhi(p)= p+\ringpsi(p)\ringN_{\ell_0,\xi_0}(p)
\end{split}
\end{align} 
satisfies \eqref{eq:HVMC1}, and $\ringpsi(0)=\ringepsilon\ringpsi_0$, $T\ringpsi(0)=\ringepsilon\ringpsi_1$. Moreover, given $\ell$ and $\xi$ with $|\ell|$, $|\dotxi|$, and $|\xi(0)-\xi_0|$ sufficiently small, the condition \eqref{eq:psidef1} for $\sigma\in [0,\tau_0]$ determines $\psi$ uniquely and gives another parameterization of $\calM$.
\end{proposition}
Note that the last statement in the proposition follows from \cite[Lemma~2.2]{LuOS1}.
\subsection{ILED for the product catenoid and elliptic estimates}\label{sec:ILEDproduct}
In this section we study ILED and uniform boundedness of the energy for the product catenoid. Specifically, we are interested in proving ILED and energy estimates for the equation
\begin{align}\label{eq:ILEDprod1}
\begin{split}
(\Box+V)\psi=g+f.
\end{split}
\end{align}
The reason for writing the right-hand side as a sum is that we will estimate the two terms in different norms. See Remark~\ref{rem:LEDproductrhsdivision}. Since we are only interested in the product catenoid in this section, the notation we use is independent of the rest of the paper. 
The only difference with the  high dimensional case considered in \cite{LuOS1} is that the eigenfunctions corresponding to the zero mode of the linear operator do not belong to the spatial part of the $LE^\ast$ norm. Therefore, in the ILED estimate we cannot simply consider the $L^2$ orthogonal projection to the complement of the kernel of the linear operator. Instead, we use a Darboux  transformation to circumvent the zero eigenfunctions. 

Consider the operator $\Box+V$ where $\Box=-\partial_t^2+\Delta$ is the wave operator associated to the product catenoid with metric
\begin{align*}
\begin{split}
g=-\ud t\otimes \ud t +\frac{\jap{\rho}^2}{\jap{\rho}^2+1}\ud\rho\otimes\ud\rho+\jap{\rho}^2\ringsg_{ab}\ud\omega^a\otimes \ud\omega^b,
\end{split}
\end{align*}
and $V$ is the second fundamental form of the standard embedding of the Riemannian catenoid. $V$ is a smooth, time independent, radial potential satisfying $|V|\lesssim \jap{\rho}^{-6}$. Recall from Section~\ref{subsec:Riemcat} that the operator $\Hbar:=\Delta+V$ has one simple positive eigenvalue with a radial exponentially decaying eigenfunction $\fybar_\mu$, $\Hbar\fybar_\mu=\mu^2\fybar_\mu$, and a three dimensional $L^2$ (with respect to $\sqrt{|g|}\ud\rho\ud\omega$) kernel spanned by the eigenfunctions $\nu^i=\frac{\Theta^i(\omega)}{\jap{\rho}^2}$, $i=1,2,3$. We also claim that that $\Hbar$ has no \emph{threshold resonances}. Recall from \cite{LuOS1} that a threshold resonance refers to a solution $\fy$ of $\Hbar\fy=0$ which is not in $L^2$ but belongs to the local energy space $\calL\calE_{\sharp}^1$ defined in the proof of Proposition~3.5 in \cite{LuOS1}.
\begin{lemma}
The operator $\Hbar$ does not have a threshold resonance.
\end{lemma}
\begin{proof}
By decomposing into spherical harmonics, we see that if $\Hbar$ has a threshold resonance, then the radial part of $\Hbar$ must also have a threshold resonance. Indeed, as in the higher dimensional case studied in \cite{LuOS1}, in view of the fast decay of $\Hbar$ to $\Delta_{\bbR^3}$ for large $\rho$,  and the repulsive potential $\ell(\ell+1)r^{-2}$ from the angular Laplacian, the projection onto the $\ell\neq0$ harmonics would already belong to $L^2$ by a perturbative analysis for large $\rho$. For the radial problem $\ell=0$ we demonstrate the existence of two asymptotically constant radial solutions $\fy_{\mathrm{odd}}$ and $\fy_{\mathrm{even}}$ to $\Hbar\fy=0$, such that $\fy_{\mathrm{odd}}$ is odd and $\fy_{\mathrm{even}}$ is even. Since any other radial solution $\fy$ of $\Hbar\fy=0$ must be a linear combination of these two solutions, we see that $\fy$ must be asymptotically constant at least at one of the ends $\rho\to\pm\infty$. Since asymptotically constant functions do not belong to the local energy space $\calL\calE_{\sharp}^1$  (see also Definition 2.8 and the remarks following it in \cite{MST}), this completes the proof. The functions $\fy_{\mathrm{even}}$ and $\fy_{\mathrm{odd}}$ correspond to scaling in the ambient space and translation symmetry along the axis of rotation, respectively. Indeed, $\fy_{\mathrm{odd}}$ was already calculated as $\fybar_4$ in Section~\ref{subsec:Riemcat} to be $\fy_{\mathrm{odd}}=\frac{\rho\sqrt{\jap{\rho}^2+1}}{\jap{\rho}^2}$. To compute the zero mode corresponding to scaling in the ambient space, we write $R$ for the radial function in the $X'=(X^1,X^2,X^3)$ coordinates. The catenoid solution $\barcalC$, is parameterized as $(R=\jap{\rho},X=Z(\rho))$ with $Z(\rho)=\sgn(\rho)\frakf^{-1}(\jap{\rho})$ (see \eqref{eq:Zdef1}). This solution is a member, corresponding to $\lambda=1$, of the larger family $\barcalC_\lambda$ given by $(R=\lambda\jap{\rho},X=\lambda Z(\rho))$. We want to parameterize $\barcalC_\lambda$ as a graph over $\barcalC$, differentiate in $\lambda$, and set $\lambda=1$. Since the normal direction to $\barcalC$ is $(\frac{1}{\jap{\rho}^2},-\frac{\rho\sqrt{\jap{\rho}^2+1}}{\jap{\rho}^2})$ we set $(\lambda\jap{\rho},\lambda Z(\rho))=(\jap{\tilrho}+\frac{\phi}{\jap{\tilrho}^2},Z(\tilrho)-\frac{\tilrho\sqrt{\jap{\tilrho}^2+1}}{\jap{\tilrho}^2}\phi)$, where we view $\tilrho$ and $\phi$ as functions of $\rho$ and $\lambda$, with $\tilrho\vert_{\lambda=1}=\rho$ and $\phi\vert_{\lambda=1}=0$. Differentiating both components of the equality and solving for $\fy_{\mathrm{even}}:=\frac{\ud}{\ud\lambda}\vert_{\lambda=1}\phi$ we obtain $$\fy_{\mathrm{even}}=\frac{\jap{\rho}Z'-\frac{\rho}{\jap{\rho}}Z }{\frac{Z'}{\jap{\rho}^2}+\frac{\rho^2\sqrt{\jap{\rho}^2+1}}{\jap{\rho}^3}}.$$
Recalling that $Z$ is odd and asymptotically constant and $Z'=\frac{1}{ \jap{\rho}\sqrt{\jap{\rho}^2+1} }$ (see \cite[Section~1.2]{LuOS1}), we see that $\fy_{\mathrm{even}}$ is even and asymptotically constant as desired.
\end{proof}

 We use the notation $\bbP_c\psi:=\psi-\angles{\psi}{\fybar_\mu}\fybar_\mu$. Note that this is not the projection to the continuous spectrum, as we have not projected away from the kernel of $\Hbar$. As a proxy for $\nu^i$ we use $Z_i:=\chi_{R_\ctf}\nu^i$, where $\chi_{R_\ctf}$ is a radial cutoff such that $\chi_{R_\ctf}\equiv 1$ for $|\uprho|\leq R_\ctf$ and $\chi_{R_\ctf}\equiv0$ for $|\uprho|\geq 2R_\ctf$. Here $R_\ctf\gg1$ is a large constant to be fixed. We use $L^p_x$ to denote the $L^p$ norm on the constant $t$ hypersurfaces with respect to the volume form induced by $g$. The corresponding $L^2_x$ pairing is denoted by $\angles{\cdot}{\cdot}$. We will use the notation $L_t^pL_x^q[t_1,t_2]$ to indicate that the $L_t^p$ norm is calculated over the time interval $[t_1,t_2]$. More generally, for a domain $\calD$ in space or time, we write $L_t^pL_x^q(\calD)$ to denote the restriction of the integrations to that domain. We use $\ringsDelta$ to denote the Laplacian on the round sphere of radius one with its standard metric and write $\sDelta=\jap{\rho}^{-2}\ringsDelta$, and similarly for $\ringsnabla$ and $\snabla$. The relevant norms are defined as
\begin{align}\label{eq:LEnormdef1}
\begin{split}
&\|\phi\|_{LE(I)}^2:=\|\jap{\rho}^{-\frac{1}{2}(1+\alpha)}\partial_\rho\phi\|_{L_{t,x}^2(I)}^2+\|\rho\jap{\rho}^{-\frac{1}{2}(3+\alpha)}\partial_t\phi\|_{L_{t,x}^2(I)}^2\\
&\phantom{\|\phi\|_{LE(I)}^2:=}+\|\rho\jap{\rho}^{-\frac{3}{2}}\snabla\phi\|_{L_{t,x}^2(I)}^2+\|\rho\jap{\rho}^{-\frac{1}{2}(5+\alpha)}\phi\|_{L_{t,x}^2(I)}^2,\\
&\|f\|_{LE^\ast(I)}^2:=\|\jap{\rho}^{\frac{1}{2}(1+\alpha)}f\|_{L_{t,x}^2(I)}^2,
\end{split}
\end{align}
where $0<\alpha\ll1$ is an arbitrary fixed constant, and $I$ is a finite or infinite time interval. The spatial parts of the $LE$ and $LE^\ast$ norms will be denoted by $\|\cdot\|_{LE_x}$ and $\|\cdot\|_{LE^\ast_x}$ respectively. For the energy norm we use the notation
\begin{align*}
\begin{split}
\|\phi\|_E^2= \|\partial_t\phi\|_{L^2_x}^2+\|\partial_x\phi\|_{L^2_x}^2,
\end{split}
\end{align*}
We will also use the notation $\|\jap{\partial_t}g\|_{LE^\ast}^2=\|g\|_{LE^\ast}^2+\|\partial_tg\|_{LE^\ast}^2$. 

Before stating the main result we pause to introduce the Darboux transform. Let $\calS_j$ denote the $L^2$ projection (with respect to the $L^2$ norm on the standard unit sphere) onto the space $\calY_j$ of the eigenfuctions of the $\bbS^2$ Laplacian $\ringsDelta$  with eigenvalue $-j(j+1)$. Note that $\nu^j(\rho,\omega)= \nu_0(\rho)\Theta^j(\omega)$ where $\nu_0=\jap{\rho}^{-2}$ and $\Theta^j\in\calY_{1,j}$ and $\calY_1=\oplus_{j=1}^3\calY_{1,j}$ as usual. Given a function $f(\rho,\omega)$ we define $\Dar f$ by
\begin{align}\label{eq:Darprod1}
\begin{split}
\calS_j\Dar f=\begin{cases}\calS_jf,\quad &j\neq 1\\\calS_1\Dar f= \nu_0\partial_\rho(\nu_0^{-1}\calS_1 f) \quad&j=1\end{cases}.
\end{split}
\end{align}

We can now state the main result of this section.
\begin{proposition}\label{prop:ILEDprod1}
Suppose $\psi$ satisfies \eqref{eq:ILEDprod1} for $t\in[t_1,t_2]$. Then
\begin{align}\label{eq:ILEDprod2}
\begin{split}
&\|\bbP_c\psi\|_{LE[t_1,t_2]}+\sup_{t_1\leq t\leq t_2}\|\bbP_c\psi(t)\|_E\\
&\lesssim R_{\ctf}^{2+\frac{\alpha}{2}}\big( \|\bbP_c\psi(t_1)\|_E+\|\bbP_cg\|_{L_t^1L_x^2[t_1,t_2]}+\|\Dar \calS_1 f\|_{LE^\ast[t_1,t_2]}\big)\\
&\quad +\|\jap{\partial_t}\bbP_cf\|_{LE^\ast[t_1,t_2]}+\sum_{j=1}^3\|\angles{\psi}{Z_j}\|_{L^2_t[t_1,t_2]}.
\end{split}
\end{align} 
\end{proposition}
\begin{remark}\label{rem:LEDproductrhsdivision}
We have written the right-hand side in \eqref{eq:ILEDprod1} as a sum of two terms for the following reason. To estimate the nonlinearity we will be able to use the norm $L_t^1L_x^2$ in our applications. However, for the source terms which involve a linear contribution from the parameters, we need to use the $LE^\ast$ norm. For this purpose, since the multiplier for the energy estimate is $\partial_t$ and in view of the degeneracy at $\rho=0$ in the $LE$ norm, we need to perform an integration by parts in time to move $\partial_t$ from $\phi$ to the source term. This is how the term $\|\partial_t\bbP_cf\|_{LE^\ast[t_{1}, t_{2}]}$ in the second line of \eqref{eq:ILEDprod2} arises. See the proof of \cite[Proposition~2.3]{LuOS1}, as well as the proof of Lemma~\ref{lem:ILEDprodj}, for the details of this argument.
\end{remark}
We write
\begin{align*}
\begin{split}
\psi= \sum_{j=0}^\infty\psi_j,\qquad \psi_j=\calS_j\psi.
\end{split}
\end{align*}
Since $\calS_j$ commutes with $\Box+V$, with $g_j$ and $f_j$ 
defined similarly as above,
\begin{align*}
\begin{split}
(-\partial_t^2+H_j)\psi_j=g_j+f_j,
\end{split}
\end{align*}
where
\begin{align*}
\begin{split}
H_j=\Delta_\rho-\frac{j(j+1)}{\jap{\rho}^2}+V, \qquad  \Delta_\rho= \frac{1}{\sqrt{|g|}}\partial_\rho\big(\sqrt{|g|}(g^{-1})^{\rho\rho}\partial_\rho\big).
\end{split}
\end{align*}
In particular, $H_1=\Delta_\rho-\frac{2}{\jap{\rho}^2}+V$. To prove Proposition~\ref{prop:ILEDprod1} we prove \eqref{eq:ILEDprod2} with $\psi$, $g$ replaced by $\psi_{\neq1}$, $g_{\neq1}$ and by $\psi_1$, $g_1$. Here $\psi_{\neq1}=\sum_{j\neq1}\calS_j\psi_j$ and similarly for $g_{\neq1}$.  The difficulty with $\psi_1$ is the presence of slowly decaying eigenfunctions for $H_1$. We start with the easier case of $\psi_{\neq1}$. Notice that the operators $H_j$, $j\neq1$, have trivial kernel (and in fact only $H_0$ has a positive eigenfunction) so the last term on the right-hand side of \eqref{eq:ILEDprod2} is unnecessary for $j\neq1$.
\begin{lemma}\label{lem:ILEDprodj}
$\psi_{\neq1}$ satisfies
\begin{align}\label{eq:ILEDprodj}
\begin{split}
\|\bbP_c\psi_{\neq1}\|_{LE[t_1,t_2]}+\sup_{t_1\leq t\leq t_2}\|\bbP_c\psi_{\neq1}(t)\|_E&\lesssim \|\bbP_c\psi_{\neq1}(t_1)\|_E+\|\bbP_cg_{\neq1}\|_{L_t^1L_x^2[t_1,t_2]}\\
&\quad+\|\jap{\partial_t}\bbP_cf_{\neq1}\|_{LE^\ast[t_1,t_2]}.
\end{split}
\end{align} 
\end{lemma}
\begin{proof}
Note that $\angles{\psi_{\neq1}}{\fybar_i}=0$ for $i=1,2,3$. Moreover, $\fybar_\mu$ decays exponentially at spatial infinity so $\|\jap{\rho}^{\frac{1+\alpha}{2}}\fybar_\mu\|_{L^2_x}$ is finite (this is the spatial part of the $LE^\ast$ norm). With these observations, the proof is the same as that \cite[Proposition~2.3]{LuOS1}. Indeed, we obtain the following bound by repeating the proof of the last two estimates in the statement of \cite[Proposition~2.3]{LuOS1}:
\begin{align*}
\|\bbP_c\psi_{\neq1}\|_{LE[t_1,t_2]}+\sup_{t_1\leq t\leq t_2}\|\bbP_c\psi_{\neq1}(t)\|_E
&\lesssim \|\bbP_c\psi_{\neq1}(t_1)\|_E+\|\bbP_cg_{\neq1}\|_{L_t^1L_x^2[t_1,t_2]} \\
&\quad+\nrm{\bbP_{c} f_{\neq 1}}_{LE^{\ast}[t_{1}, t_{2}]} + \abs*{\int_{t_{1}}^{t_{2}} \int \chi(\rho) \bbP_{c} f_{\neq 1} \rd_{t} \bbP_{c} \psi_{\neq 1} \, \ud x},
\end{align*}
where $\chi(\rho)$ is a smooth cutoff supported in $\set{\abs{\rho} \leq 1}$. For the last term, we claim that
\begin{align*}
\abs*{\int_{t_{1}}^{t_{2}} \int \chi(\rho) \bbP_{c} f_{\neq 1} \rd_{t} \bbP_{c} \psi_{\neq 1} \, \ud x}
&\leq \dlt \left(\nrm{\bbP_{c} \psi_{\neq 1}}_{LE[t_{1}, t_{2}]} + \sup_{t_{1} \leq t \leq t_{2}} \nrm{\bbP_{c} \psi_{\neq 1}(t)}_{E}\right) \\
&\quad + C_{\dlt} \left(\nrm{\chi(\rho) \bbP_{c} f}_{LE^{\ast}[t_{1}, t_{2}]} + \nrm{\chi(\rho) \rd_{t} \bbP_{c} f}_{LE^{\ast}[t_{1}, t_{2}]} \right),
\end{align*}
for every $\dlt > 0$. With such an estimate, we may immediately complete the proof by taking $\dlt > 0$ sufficiently small to absorb the contribution of $\bbP_{c} \psi_{\neq 1}$ into the left-hand side.

To prove the claim, we divide into two cases. First, when $\abs{t_{2} - t_{1}} \leq 1$, we simply observe that
\begin{align*}
\abs*{\int_{t_{1}}^{t_{2}} \int \chi(\rho) \bbP_{c} f_{\neq 1} \rd_{t} \bbP_{c} \psi_{\neq 1} \, \ud x} 
&\leq \nrm{\bbP_{c} f_{\neq 1}(t)}_{L^{1}_{t} L^{2}_{x}[t_{1}, t_{2}]} \sup_{t_{1} \leq t \leq t_{2}} \nrm{\rd_{t} \bbP_{c} \psi_{\neq 1}(t)}_{L^{2}} \\
&\aleq \abs{t_{2} - t_{1}} \nrm{\bbP_{c} f_{\neq 1}(t)}_{LE^{\ast}[t_{1}, t_{2}]} \sup_{t_{1} \leq t \leq t_{2}} \nrm{\bbP_{c} \psi_{\neq 1}(t)}_{E}.
\end{align*}
Since $\abs{t_{2} - t_{1}} \leq 1$, the desired estimate follows from Cauchy-Schwarz. 

Consider now the second case, when $\abs{t_{2}-t_{1}} \geq 1$. We integrate $\rd_{t}$ by parts and estimate
\begin{align*}
\abs*{\int_{t_{1}}^{t_{2}} \int \chi(\rho) \bbP_{c} f_{\neq 1} \rd_{t} \bbP_{c} \psi_{\neq 1} \, \ud x}
&\leq \sup_{t_{1} \leq t \leq t_{2}} \nrm{\bbP_{c} f_{\neq 1}(t)}_{L^{2}_{x}(\set{\abs{\rho} \leq 1})} \nrm{\bbP_{c} \psi_{\neq 1}(t)}_{L^{2}_{x}(\set{\abs{\rho} \leq 1})} \\
&\quad + \abs*{\int_{t_{1}}^{t_{2}} \int \chi(\rho) \rd_{t} \bbP_{c} f_{\neq 1} \bbP_{c} \psi_{\neq 1} \, \ud x}.
\end{align*}
The last term can be estimated as in the proof of \cite[Proposition~2.3]{LuOS1} by 
\begin{equation*}
\abs*{\int_{t_{1}}^{t_{2}} \int \chi(\rho) \rd_{t} \bbP_{c} f_{\neq 1} \bbP_{c} \psi_{\neq 1} \, \ud x}
\leq \dlt \nrm{\bbP_{c} \psi_{\neq 1}}_{LE[t_{1}, t_{2}]}^{2} + C_{\dlt} \nrm{\rd_{t} f}_{LE^{\ast}[t_{1}, t_{2}]}^{2},
\end{equation*}
which is acceptable. It remains to estimate the first term on the right-hand side. We begin with the following simple calculus identity on $[t_{1}, t_{2}]$:
\begin{equation*}
	\sup_{t_{1} \leq t \leq t_{2}} \abs{\bbP_{c} f_{\neq 1}(t)}^{2}
	\leq \nrm{\bbP_{c} f}_{L^{2}[t_{1}, t_{2}]} \nrm{\rd_{t} \bbP_{c} f}_{L^{2}[t_{1}, t_{2}]} + \frac{1}{t_{2} - t_{1}} \int_{t_{1}}^{t_{2}} \abs{\bbP_{c} f_{\neq 1}(t)}^{2} \, \ud t.
\end{equation*}
Multiplying by $\chi(\rho)$ and integrating in $x$, we obtain
\begin{equation*}
\sup_{t_{1} \leq t \leq t_{2}} \nrm{\bbP_{c} f_{\neq 1}(t)}_{L^{2}_{x}(\set{\abs{\rho} \leq 1})}^{2} \aleq \nrm{\bbP_{c} f_{\neq 1}}_{LE^{\ast}[t_{1}, t_{2}]} \nrm{\rd_{t} \bbP_{c} f_{\neq 1}}_{LE^{\ast}[t_{1}, t_{2}]} + \frac{1}{t_{2}-t_{1}} \nrm{\bbP_{c} f_{\neq 1}}_{LE^{\ast}[t_{1}, t_{2}]}^{2}.
\end{equation*}
Since $\abs{t_{2} - t_{1}} \geq 1$, the desired estimate follows from Cauchy-Schwarz.
 \qedhere
\end{proof}
Recalling that $\calY_1$ is three dimensional, with $\nu^i$ belonging to orthogonal subspaces of $\calY_1$, we perform a further decomposition $\psi_1=\psi_{1,1}+\psi_{1,2}+\psi_{1,3}$, with $\psi_{1,j}=\angles{\psi_1}{\Theta^j}_{L^2(\bbS^2)}\Theta^j$. Using a similar decomposition for $g_1$ and $f_1$, we see that $\psi_{1,j}$ satisfy the radial equations
\begin{align}\label{eq:shreduced1}
\begin{split}
(-\partial_t^2+\Delta_\rho-\frac{2}{\jap{\rho}^2}+V)\psi_{1,j}=g_{1,j}+f_{1,j},\qquad j=1,2,3.
\end{split}
\end{align}
Our task has therefore reduced to studying ILED for the radial wave equation \eqref{eq:shreduced1}, where we note that the operator $H_1:=\Delta_\rho-\frac{2}{\jap{\rho}^2}+V$ has no negative eigenvalues and a simple zero eigenvalue with eigenfunction $\nu_0:=\jap{\rho}^{-2}$. The following lemma is the key ingredient of the proof.
\begin{lemma}\label{lem:shred1}
Suppose the radially symmetric function $\Psi$ satisfies $-\partial_t^2\Psi+H_1\Psi=G$ for $t\in [t_1,t_2]$. Let $Z_0=\chi_{R_\ctf}\nu_0$ where  $\chi_{R_\ctf}\equiv 1$ for $|\uprho|\leq R_\ctf$ and $\chi_{R_\ctf}\equiv0$ for $|\uprho|\geq 2R_\ctf$. Then
\begin{align}\label{eq:shILEDtemp1}
\begin{split}
&\|\jap{\rho}^{-\frac{5+\alpha}{2}}\Psi\|_{L^2_{t,x}[t_1,t_2]}\\
&\lesssim R_\ctf^{2+\frac{\alpha}{2}} \big(\|\Dar \Psi (t_1)\|_E+\|\Dar G\|_{LE^\ast[t_1,t_2]}\big)+\|\angles{\Psi}{Z_0}\|_{L^2[t_1,t_2]},
\end{split}
\end{align}
where $\Dar G := \nu_0\partial_\uprho(\nu_0^{-1}G)$.
\end{lemma}
\begin{proof}
The idea is to perform a Darboux transformation to arrive at a new operator with no eigenvalues. We begin by conjugating the equation to the line. Let $\Phi=|g|^{\frac{1}{4}}\Psi$. Then $\Phi$ satisfies 
\begin{align*}
\begin{split}
(-\partial_t^2+L_1)\Phi=H,\qquad L_1=\partial_\rho((g^{-1})^{\rho\rho}\partial_\rho)-\frac{2}{\jap{\rho}^2}+V-\frac{\partial_\rho((g^{-1})^{\rho\rho}\partial_\rho |g|^{\frac{1}{4}})}{|g|^{\frac{1}{4}}},
\end{split}
\end{align*}
where $H=|g|^{\frac{1}{4}} G$. Recalling that $\nu_0=\jap{\rho}^{-2}$ is the unique eigenfunction of $H_1$, we see that $y_0=|g|^{\frac{1}{4}}\nu_0$ is the unique eigenfunction of $L_1$. Moreover, $L_1$ admits the decomposition
\begin{align*}
L_1=-D^\ast D,\qquad D:=\sqrt{(g^{-1})^{\rho\rho}}y_0\partial_\rho y_0^{-1},\quad D^\ast := -y_0^{-1}\partial_\rho \sqrt{(g^{-1})^{\rho\rho}}y_0.
\end{align*}
Letting $L_2 = -D D^\ast$ we see that $w=D\Phi$ satisfies
\begin{align*}
\begin{split}
(-\partial_t^2+L_2)w=DH.
\end{split}
\end{align*}
Note that $L_2$ has no eigenvalues, because if $L_2\varphi=\lambda\fy$ then $L_1 (D^\ast \fy)=\lambda D^\ast \fy$, which implies that $\lambda=0$ and $D^\ast \fy=cy_0$. But then
\begin{align*}
\begin{split}
c\|y_0\|_{L^2_{\ud \rho}}^2=\angles{D^\ast \fy}{y_0}_{L^2_{\ud \rho}}=\angles{\fy}{Dy_0}_{L^2_{\ud \rho}}=0,
\end{split}
\end{align*}
which is a contradiction. We now conjugate back to three dimensions by letting $u=|g|^{-\frac{1}{4}}w$, so that with $h=|g|^{-\frac{1}{4}}DH$,
\begin{align}\label{eq:DarPsi1}
\begin{split}
(-\partial_t^2+\Delta_\rho +\tilV)u= h.
\end{split}
\end{align}
Here $\Delta_\rho=|g|^{-\frac{1}{2}}\partial_\rho((g^{-1})^{\rho\rho}|g|^{\frac{1}{2}}\partial_\rho\cdot)$, and $\tilV$ is a smooth radial potential with $|\tilV|\lesssim \jap{\rho}^{-4}$. Moreover, the operator $\Delta_\rho+\tilV$ now has no eigenvalues (otherwise, by conjugation, $L_2$ would have an eigenvalue). As in Lemma~\ref{lem:ILEDprodj} we conclude that
\begin{align}\label{eq:Darlemtemp1}
\begin{split}
\|u\|_{LE[t_1,t_2]}+\sup_{t_1\leq t\leq t_2}\|u(t)\|_{E}\lesssim \|u(t_1)\|_{E}+ \|h\|_{LE^\ast[t_1,t_2]}.
\end{split}
\end{align}
Note that here on the right-hand side we see $\|h\|_{LE^\ast}$ instead of $\|\jap{\partial_t}h\|_{LE^\ast}$ because we are dealing with a radial problem so there is no degeneracy at $\rho=0$ in the $LE$ norm. Indeed, the $\partial_\rho$ derivative is already controlled without a degeneracy in the $LE$ norm and to obtain a non-degenerate bound on $\partial_t$ in the interior, we can simply multiply the equation by $\chi_{\lesssim 1}u$, with $\chi_{\lesssim1}$ radial and supported in $\{|\rho|\lesssim 1\}$, to obtain
\begin{align*}
\begin{split}
\|\partial_t u\|_{L^2_{t,x}([t_1,t_2]\times\{|\rho|\lesssim1\})}\lesssim \|\partial_\rho u\|_{L^2([t_1,t_2]\times\{|\rho|\lesssim1\})}+\|h\|_{LE^\ast[t_1,t_2]}+\|u(t_1)\|_E.
\end{split}
\end{align*}
It remains to relate \eqref{eq:Darlemtemp1} estimate to $\Psi$. Tracking back we have $u=|g|^{-\frac{1}{4}}D(|g|^{\frac{1}{4}}\Psi)=\sqrt{(g^{-1})^{\rho\rho}}\Dar \Psi$, which implies that for some time dependent constant $c_0$,
\begin{align}\label{eq:Psinu0temp1}
\begin{split}
\Psi(\rho)=c_0\nu_0(\rho)+\nu_0(\rho)\int_0^\rho \nu_0^{-1}(\rho')u(\rho')\sqrt{g_{\rho\rho}(\rho')}\ud\rho'.
\end{split}
\end{align}
From Schur's test, and since $h=|g|^{-\frac{1}{4}}D(|g|^{\frac{1}{4}}G)=\sqrt{(g^{-1})^{\rho\rho}}\nu_0\partial_\rho(\nu_0^{-1}G)$,
\begin{align*}
\begin{split}
\|\jap{\rho}^{-\frac{5+\alpha}{2}}\Psi\|_{L^2_{t,x}[t_1,t_2]}&\lesssim \|u\|_{LE[t_1,t_2]}+\|c_0\|_{L^2_t[t_1,t_2]}\\
&\lesssim \|\Dar \Psi(t_1) \|_E+\|\Dar G\|_{LE^\ast[t_1,t_2]}+\|c_0\|_{L^2_t[t_1,t_2]}.
\end{split}
\end{align*}
To estimate $c_0$ we can pair \eqref{eq:Psinu0temp1} with $Z_0$ to conclude that
\begin{align*}
\begin{split}
\|c_0\|_{L^2[t_1,t_2]}\lesssim \|\angles{\Psi}{Z_0}\|_{L^2[t_1,t_2]}+R_\ctf^{2+\frac{\alpha}{2}}\|u\|_{LE[t_1,t_2]},
\end{split}
\end{align*}
which combined with the previous estimate gives \eqref{eq:shILEDtemp1}.
\end{proof}
We are now ready to prove Proposition~\ref{prop:ILEDprod1}.
\begin{proof}[Proof of Proposition~\ref{prop:ILEDprod1}]
We extend $\psi$ to be a solution of the homogeneous wave equation for $t\notin[t_1,t_2]$, and drop the restriction to $[t_1,t_2]$ in all norms. By Lemma~\ref{lem:ILEDprodj}, it suffices to consider $\psi_1$. For this we perform a near-far decomposition as follows. Let $\psi_{1,\far}$ be the solution to
\begin{align}\label{eq:ILEDpsi1far1}
\begin{split}
(-\partial_t^2+H_1+V_\far)\psi_{1,\far}=g_1,\qquad (\psi_{1,\far}(t_1),\partial_t\psi_{1,\far}(t_1))=(\psi_{1}(t_1),\partial_t\psi_{1}(t_1)),
\end{split}
\end{align}
where the spherically symmetric potential $V_\far$ is chosen so that $|V_\far|\lesssim \jap{\rho}^{-6}$ and $V-2\jap{\rho}^{-2}+V_\far$ is repulsive potential in the sense that it is negative and increasing. It follows as in Lemma~\ref{lem:ILEDprodj} that 
\begin{align*}
\begin{split}
\|\psi_{1,\far}\|_{LE}+\sup_{t}\|\psi_{1,\far}(t)\|_E\lesssim \|\psi_1(t_1)\|_{E} +\|g_1\|_{L_t^1L_x^2}.
\end{split}
\end{align*}
Next, observe that $\psi_{1,\near}=\psi_1-\psi_{1,\far}$ satisfies
\begin{align}\label{eq:Darpsi1neartemp1}
\begin{split}
(-\partial_t^2+H_1)\psi_{1,\near}=V_\far \psi_{1,\far}+f_1,\qquad (\psi_{1,\near}(t_1),\partial_t\psi_{1,\near}(t_1))=(0,0).
\end{split}
\end{align}
For $\tilR$ is sufficiently large, we claim that
\begin{equation}\label{eq:psi1prodtemp1}
\|\psi_{1,\near}\|_{LE}+\sup_{t}\|\psi_{1,\near}\|_E\lesssim \|\psi_1(t_1)\|_{E}+\|f_1\|_{LE^\ast}+\|g_1\|_{L_t^1L_x^2}+\|\psi_{1,\near}\|_{L^2_{t,x}(\{|\rho|\leq \tilR\})}.
\end{equation}
 
Indeed this follows from a multiplier argument as in the proof of \eqref{eq:Darlemtemp1} for $\psi_1$ and noting that $\psi_{1,\near}=\psi_1-\psi_{1,\far}$. To derive the estimate for $\psi_1$ we treat $V\psi_1$ as an error in the right-hand side for both the energy and ILED estimates. The decay of $V$ shows that the corresponding error in the ILED estimate can be bonded by $\|\psi_1\|_{L^2_{t,x}(\{|\rho|\leq \tilR\})}$, which in view of the estimate for $\psi_{1,\far}$ is bounded by the right-hand side of \eqref{eq:psi1prodtemp1}. For the contribution of $\int\chi_{\{|\rho|\leq\tilR\}}V\psi_{1}^2\sqrt{|g|}\ud\omega\ud\rho$ in the energy estimate we have noted that (see \cite[Lemma~7.7]{LuOS1} for a similar argument in a different context)
\begin{align*}
\begin{split}
\int\chi_{\{|\rho|\leq\tilR\}}V\psi_{1}^2\sqrt{|g|}\ud\omega\ud\rho=\int\chi_{\{|\rho|\leq\tilR\}}\partial_\rho(\rho)V\psi_{1}^2\sqrt{|g|}\ud\omega\ud\rho\leq \tilde{\delta} \|\psi_{1}\|_{E}^2+C_{\tilde\delta}\|\rho\psi_1\|_{L^2_x(\{|\rho|\leq \tilR\})}^2,
\end{split}
\end{align*}
where $\tilde{\delta}$ is a small constant. The first term can be absorbed by the energy norm while for the second term we have
\begin{align*}
\begin{split}
\|\rho\psi_1(t_2)\|_{L^2_x(\{|\rho|\leq \tilR\})}^2&\lesssim \|\psi_1(t_1)\|_{E}^2+\int_{t_1}^{t_2}\int\chi_{\{|\rho|\leq \tilR\}}\rho^2\psi_{1}\partial_t\psi_{1}\sqrt{|g|}\ud\omega\ud\rho\ud t\\
&\leq  \|\psi_1(t_1)\|_{E}^2+ \tilde{\epsilon}\|\psi_{1}\|_{LE}^2+C_{\tilde{\epsilon}}\|\psi_{1}\|_{L^2_{t,x}(\{|\rho|\leq \tilR\})}^2,
\end{split}
\end{align*}
where $\tilde{\epsilon}$ is a small constant, so that the first term can be absorbed by the $LE$ norm. Returning to \eqref{eq:psi1prodtemp1}, we need to deal with the last term on the right-hand side. 
By \eqref{eq:Darpsi1neartemp1} and Lemma~\ref{lem:shred1}, and in view of the established bounds on $\psi_{1,\far}$ and the fast decay of $V_\far$, we conclude that
\begin{align}\label{eq:DarPN0psinear1temp1}
\begin{split}
\|\psi_{1,\near}\|_{L^2_{t,x}(\{|\rho|\leq \tilR\})}&\lesssim R_\ctf^{2+\frac{\alpha}{2}}(\|\psi_1(t_1)\|_{E}+\|\Dar f_1\|_{LE^\ast}+\|g_1\|_{L_t^1L_x^2})+\sum_{j=1}^3\|\angles{\psi_{1,\near}}{Z_j}\|_{L^2_t}.
\end{split}
\end{align}
Finally, to estimate $\|\angles{\psi_{1,\near}}{Z_j}\|_{L^2_t}$ we observe that this term is bounded by
\begin{align*}
\begin{split}
\|\angles{\psi_{1,\near}}{Z_j}\|_{L^2_t}=\|\angles{\psi_1-\psi_{1,\far}}{Z_j}\|_{L^2_t}\lesssim \|\angles{\psi_{1}}{Z_j}\|_{L^2_t}+R_\ctf^{1+\frac{\alpha}{2}}\|\psi_{1,\far}\|_{
LE}.
\end{split}
\end{align*}
Since $Z_j=\calS_1Z_j$, this gives the desired estimate.
\end{proof}

We next state some elliptic and coercivity estimates for $\Hbar$. These will be used in the context of deriving decay for the solution from the decay of its time derivatives. We start with a definition, where we recall that $\partial_\Sigma$ denotes derivatives of size one.
\begin{definition} \label{def:w-Sob}
For $1 \leq p \leq \infty$, $s \in \bbZ_{\geq 0}$ and $\gmm \in \bbR$, define
\begin{align*}
	\nrm{\psi}_{\ell^{p}_{r} \calH^{s, \gmm}}
	= \sum_{\abs{\alp} \leq s} \nrm{2^{k\gmm} \nrm{2^{k\abs{\alp}} \rd_\Sigma^{\alp} \psi}_{L^{2}(A_{k})}}_{\ell^{p}(k\in\bbZ_{\geq0})},
\end{align*}
where $A_{k} = \set{(\rho, \omg) \in \ucalC : 2^{k} \leq \abs{\rho} \leq 2^{k+1}}$ for $k \geq 1$, and $A_{0} = \set{(\rho, \omg) \in \ucalC : \abs{\rho} \leq 2}$. 
\end{definition}
The elliptic and coercivity estimates are given in the next proposition.
\begin{proposition} \label{prop:uH}
Suppose $\uH \phi=f_1+f_2$. Let $\uZ_i=\chi_{R_\ctf}\fybar_i$, $i=\mu,1,2,3$ where~$\chi_{R_\ctf}(\cdot)=\chi(\cdot/R_\ctf)$ and $\chi$ is a radial cutoff to the region $\{|\rho|\leq1\}$.
Then for $s \in \bbZ_{\geq 0}$, we have
\begin{equation} \label{eq:uH-wSob}
	\nrm{\phi}_{\ell^{\infty} \calH^{s+2, -\frac{3}{2}}} \aleq \|f_1\|_{\ell^{1} \calH^{s, \frac{1}{2}}} +R_\ctf^2\|\Dar \calS_1 f_1\|_{\ell^{1} \calH^{s, \frac{1}{2}}} + R_\ctf^2\nrm{f_2}_{\ell^{1} \calH^{s, \frac{1}{2}}} + \sum_{i=\mu,1,2,3} \abs{\brk{\phi, \uZ_{i}}}.	
\end{equation}
For $s\in(1,\frac{3}{2})$, we have
\begin{equation}\label{eq:shardy1}
\|\jap{\uprho}^{-s}\phi\|_{L^2}\lesssim R_\ctf^{s-\frac{1}{2}} \|\Hbar\phi\|_{L^2}^{s-1}\|\partial_\Sigma\phi\|_{L^2}^{2-s}+R_\ctf^{2s-1} \|\Hbar\phi\|_{L^2}+R_\ctf^{\frac{1}{2}}|\angles{\phi}{\uZ_\mu}|+\sum_{i=1}^3 \abs{\brk{\phi, \uZ_{i}}}.
\end{equation}
Moreover,
\begin{align}\label{eq:energycoer1}
\begin{split}
R_\ctf^\frac{1}{2}|\angles{-\Hbar\phi}{\phi}|^{\frac{1}{2}}+R_\ctf^{\frac{1}{2}}|\angles{\phi}{\uZ_\mu}|+\sum_{i=1}^3 \abs{\brk{\phi, \uZ_{i}}}\gtrsim \|\partial_\Sigma\phi\|_{L^2}.
\end{split}
\end{align}
\end{proposition}
\begin{proof}
Starting with \eqref{eq:energycoer1} note that $\phi^\perp=\phi-\sum_{i=\mu,1,2,3}\angles{\phi}{\fybar_i}\fybar_i$ satisfies (see for instance \cite[Theorem~1.3]{StuartHiggs} for a proof of such an estimate)
\begin{align*}
\begin{split}
\mu^2\angles{\phi}{\fybar_\mu}^2+\angles{-\Hbar \phi}{\phi}= \angles{-\Hbar \phi^\perp}{\phi^\perp}\gtrsim \|\partial_\Sigma\phi^\perp\|_{L^2}^2.
\end{split}
\end{align*}
On the other hand 
\begin{align*}
\begin{split}
\|\partial_\Sigma\phi\|_{L^2}\lesssim \|\partial_\Sigma\phi^\perp\|_{L^2}+\sum_{i=\mu,\dots,3}|\angles{\phi}{\fybar_i}|,
\end{split}
\end{align*}
and by a Hardy type estimate and in view of the definition of $\uZ_i$,
\begin{align*}
\begin{split}
\sum_{i=\mu,\dots,3}|\angles{\phi}{\fybar_i}|\lesssim \sum_{i=\mu,\dots,3}|\angles{\phi}{\uZ_{i}}|+\sum_{i=\mu,\dots,3}|\angles{\phi^\perp}{\uZ_{i}}|\lesssim \sum_{i=\mu,\dots,3}|\angles{\phi}{\uZ_{i}}|+R_\ctf^{\frac{1}{2}}\|\partial_\Sigma\phi^\perp\|_{L^2}.
\end{split}
\end{align*}
A similar argument, and the fact that $|\angles{\phi^\perp}{\uZ_\mu}|\leq e^{-\tilde{\mu}R_\ctf}\|\partial_\Sigma\phi^\perp\|_{L^2}$ for a suitable $\tilde{\mu}>0$, shows that
\begin{align*}
\begin{split}
|\angles{\phi}{\fybar_\mu}|\lesssim e^{-\tilde{\mu}R_\ctf}\sum_{j=1}^3|\angles{\phi}{\fybar_j}|+e^{-\tilde{\mu}R_\ctf}|\angles{-H\phi}{\phi}|^{\frac{1}{2}}+|\angles{\phi}{\uZ_\mu}|.
\end{split}
\end{align*}
Putting these estimates together completes the proof of \eqref{eq:energycoer1}. For \eqref{eq:shardy1} we choose $V_\far$ such that it is zero inside a large compact set and agrees with $V$ outside a larger compact set. Moreover, since $\Delta_\barcalC$ has purely absolutely continuous spectrum, by choosing the compact set on which $V_\far$ vanishes large enough, we can ensure that $\Hbar_\far=\Delta_\barcalC+V_\far$ has no eigenvalues. Then a unique solution $\phi_\far$ to $\Hbar_\far\phi_\far=\Hbar\phi$ exists by the Fredholm alternative, and arguing as in the proof of equations (8.34) and (8.35) in \cite[Lemma~8.12]{LuOS1},
\begin{equation}\label{eq:shardytemp1}
\|\jap{\uprho}^{-s}\phi_\far\|_{L^2}\lesssim \|\Hbar\phi\|_{L^2}^{s-1}\|\partial_\Sigma\phi_\far\|_{L^2}^{2-s}\lesssim\|\Hbar\phi\|_{L^2}^{s-1}\|\partial_\Sigma\phi\|_{L^2}^{2-s}+\delta^{-1}\|\Hbar\phi\|_{L^2}+\delta\|\partial_\Sigma\phi_\near\|_{L^2},
\end{equation}
where $\phi_\near = \phi-\phi_\far$, and $\delta\ll1$ is to be chosen. Since $\Hbar\phi_\near = (V-V_\far)\phi_\far$, by Hardy's inequality 
\begin{align*}
\begin{split}
\|\jap{\uprho}^{-s}\phi_\near\|_{L^2}&\lesssim \|\jap{\uprho}^{-1}\phi_\near\|_{L^2}+ \|\partial_\Sigma\phi_\near\|_{L^2}.
\end{split}
\end{align*}
On the other hand, by \eqref{eq:energycoer1}, and with $\delta'\ll1$,
\begin{align*}
\begin{split}
\|\partial_\Sigma\phi_\near\|_{L^2}&\lesssim C_{\delta'}R_\ctf^{\frac{1}{2}}\|\jap{\rho}^{-s}\phi_\far\|_{L^2}+R_\ctf^{\frac{1}{2}}|\angles{\phi-\phi_\far}{\uZ_\mu}|+\sum_{i=1}^3 |\angles{\phi-\phi_\far}{\uZ_i}|
+\delta'\|\jap{\rho}^{-1}\phi_\near\|_{L^2}\\
&\lesssim C_{\delta'}R_\ctf^{s-\frac{1}{2}}\|\jap{\rho}^{-s}\phi_\far\|_{L^2}+R_\ctf^{\frac{1}{2}}|\angles{\phi}{\uZ_\mu}|+\sum_{i=1}^3 |\angles{\phi}{\uZ_i}|+\delta'\|\jap{\rho}^{-1}\phi_\near\|_{L^2}.
\end{split}
\end{align*}
Combining with \eqref{eq:shardytemp1} and choosing $\delta=\delta_0R_\ctf^{-s+\frac{1}{2}}$ with $\delta_0\ll1$ gives \eqref{eq:shardy1}. The proof of \eqref{eq:uH-wSob} follows a similar outline. Define $V_\far$, $\phi_\far$, and $\phi_\near$ as before, but with $\uH_\far\phi_\far= f_2$. We claim that
\begin{align}\label{eq:uH-wSob-temp1}
\begin{split}
\|\phi_\far\|_{\ell^\infty\calH^{s+2,-\frac{3}{2}}}\lesssim \|f_2\|_{\ell^1\calH^{s,\frac{1}{2}}}.
\end{split}
\end{align}
Indeed, as in the proof of \cite[Lemma~8.12]{LuOS1} using the coordinates $(\rho,\omega)$ we first consider
\begin{align*}
\begin{split}
\Hbar_\euc \phi_\euc=\chi f_2
\end{split}
\end{align*}
where $\chi$ is a cutoff to the large $\rho$ region (with a similar construction at the other asymptotic end). Here $\Hbar_\euc$ is an operator defined using the $(\rho,\omega)$ coordinates which agrees with the Euclidean Laplacian  $\Delta_\euc$ inside a large ball $\{0\leq \rho\leq R_{\mathrm{large}}\}$ and agrees with $\Hbar_\far$ for $\rho\geq2R_{\mathrm{large}}$. By the existing theory for $\Delta_\euc$ (see \cite[Section~7]{LuOh}) $\phi_\euc$ satisfies 
\begin{align*}
\begin{split}
\|\phi_\euc\|_{\ell^\infty\calH_\euc^{s+2,-\frac{3}{2}}}\lesssim \|\chi f_2\|_{\ell^1\calH_\euc^{s,\frac{1}{2}}},
\end{split}
\end{align*}
where the spaces $\calH_\euc^{s,\gamma}$ are defined as before but with respect to the Euclidean measure and derivatives. Here we have treated $\Hbar_\euc-\Delta_\euc$ perturbatively as $\|(\Hbar_\euc-\Delta_\euc)\phi_\euc\|_{\ell^1\calH_\euc^{s,\frac{1}{2}}}\ll \|\phi_\euc\|_{\ell^\infty\calH_\euc^{s+2,-\frac{3}{2}}}$.  Let $\psi_\euc:=\tilchi \phi_\euc$, where $\tilchi$ is another cutoff to the large $\rho$ region such that $\tilchi\chi=\chi$. Since $\psi_\euc$ is supported in the large $\rho$ region, using the coordinates $(\rho,\omega)$ we can again view it as a function on $\barcalC$. It follows from the previous estimate on $\phi_\euc$ that
\begin{align*}
\begin{split}
\|\psi_\euc\|_{\ell^\infty\calH^{s+2,-\frac{3}{2}}}\lesssim  \| f_2\|_{\ell^1\calH^{s,\frac{1}{2}}}.
\end{split}
\end{align*}
Now $\psi_\cat:=\phi_\far-\psi_\euc$ satisfies $\Hbar_\far\psi_\cat = (1-\chi)f_2-[\Hbar_\far,\tilchi]\phi_\euc-\tilchi(\Hbar_\far-\Hbar_\euc)\phi_\euc$, where the right-hand side is compactly supported. Moreover, by the same computation as for \eqref{eq:energycoer1},
\begin{align}\label{eq:psicattemp1}
\begin{split}
\|\partial_\Sigma u\|_{L^2}+\|\jap{\rho}^{-1}u\|_{L^2}\lesssim \|\jap{\rho}\Hbar_\far u\|_{L^2}
\end{split}
\end{align}
for any function $u$. Applying this with $u=\psi_\cat$ and using the earlier estimate for $\phi_\euc$ gives
\begin{align*}
\begin{split}
\sum_{|\alpha|\leq1}\nrm{2^{-\frac{3}{2}k} \nrm{2^{|\alpha|k}\partial^\alpha_\Sigma u}_{L^{2}(A_{k})}}_{\ell^{\infty}(k\in\bbZ_{\geq0})}&\lesssim \|\jap{\uprho}^{-1}u\|_{L^2}+ \|\partial_\Sigma u\|_{L^2}
\lesssim\| f_2\|_{\ell^1\calH^{s,\frac{1}{2}}}.
\end{split}
\end{align*} 
Starting from this estimate, we can prove the higher order estimates for $\psi_\cat$ inductively by writing the equation for $\psi_\cat$ as 
\begin{align*}
\begin{split}
\Delta_\barcalC \psi_\cat = (1-\chi)f_2-[\Hbar_\far,\tilchi]\phi_\euc-\tilchi(\Hbar_\far-\Hbar_\euc)\phi_\euc-V \psi_\cat,
\end{split}
\end{align*}
commuting $\chi_{\lesssim 1}\partial_\Sigma^{k-1}$ and $\chi_{\gtrsim1}\uprho^{k-\frac{3}{2}}\partial_\Sigma^{k-1}$, and using elliptic estimates. We conclude that $$\|\psi_\cat\|_{\ell^\infty\calH^{s+2,-\frac{3}{2}}}\lesssim  \| f_2\|_{\ell^1\calH^{s,\frac{1}{2}}}.$$
Combining this and the earlier estimate on $\psi_\euc$ gives \eqref{eq:uH-wSob-temp1}. Returning to $\phi_\near=\phi-\phi_\far$ note that $\Hbar \phi_\near = (V-V_\far)\phi_\far+f_1$. It follows by the same proof as for \eqref{eq:uH-wSob-temp1} that
\begin{align*}
\begin{split}
\|\phi_\near\|_{\ell^\infty\calH^{s+2,-\frac{3}{2}}}\lesssim   \|f_1\|_{\ell^1\calH^{s,\frac{1}{2}}}+\|f_2\|_{\ell^1\calH^{s,\frac{1}{2}}}+\|V\phi_\near\|_{\ell^1\calH^{s,\frac{1}{2}}}.
\end{split}
\end{align*}
To estimate $\|V\phi_\near\|_{\ell^1\calH^{s,\frac{1}{2}}}$ we decompose into spherical harmonics, and apply a Darboux transformation on the space of first harmonics where $\uZ_i$ belong. An analogous process is carried out in the proof of Proposition~\ref{prop:ILEDprod1}, specifically Lemma~\ref{lem:shred1}, so we will be brief here.  We start by decomposing 
\begin{align*}
\begin{split}
\phi_\near = \calS_1\phi_\near+\calS_{\neq1}\phi_\near,
\end{split}
\end{align*}
where $\calS_{\neq1}\phi_\near =\sum_{j\neq1}\calS_j \phi_\near$. As in the proof of \eqref{eq:energycoer1}, for any function $u$ (note that the kernel of $\Hbar$ belongs entirely to the space of first spherical harmonics, while the positive eigenvalue is radial)
\begin{align}\label{eq:psicattemp2}
\begin{split}
\|\partial_\Sigma\calS_{\neq1}u\|_{L^2}+\|\jap{\rho}^{-1}\calS_{\neq1}u\|_{L^2}\lesssim |\angles{u}{\uZ_\mu}|+\|\jap{\rho}\calS_{\neq1}\Hbar u\|_{L^2}.
\end{split}
\end{align}
 Using a near-far decomposition similar to what was done for $\psi_\cat$ and $\psi_\euc$ above, and using \eqref{eq:psicattemp2} in place of \eqref{eq:psicattemp1}, we can then show that
\begin{align*}
\begin{split}
\|V\calS_{\neq 1}\phi_\near\|_{\ell^1\calH^{s,\frac{1}{2}}}\lesssim \|\calS_{\neq1}\phi_\near\|_{\ell^\infty\calH^{s+2,-\frac{3}{2}}}\lesssim \| \calS_{\neq1}f_1\|_{\ell^1\calH^{s,\frac{1}{2}}}+ \| f_2\|_{\ell^1\calH^{s,\frac{1}{2}}}+|\angles{\phi}{\uZ_\mu}|.
\end{split}
\end{align*}
For $\calS_1\phi_\near$ observe that $(\uH_{\mathrm{rad}}-2\jap{\uprho}^{-2})\calS_1\phi_\near= \calS_1\uH\phi_\near$, where $\uH_{\mathrm{rad}}$ denotes the radial part of $\uH$.  We further decompose $\calS_1\phi_\near=\sum_{j=1}^3\calS_{1,j}\phi_\near$, with $\calS_{1,j}$ denoting the projection on $\calY_{1,j}$. Since the argument is the same for $j=1,2,3$, by an abuse of notation we suppress the $j$ dependence and denote the radial part of $\calS_{1,j}\phi_\near$ by $\calS\phi_{\near}$. After applying the Darboux transformation, we find that (see the proof of Lemma~\ref{lem:shred1} for the details of this computation)
\begin{align*}
\begin{split}
\tilH \Dar\calS\phi_\near = \calS\Dar \uH\phi_\near,
\end{split}
\end{align*}
where $\tilH$ is asymptotically Euclidean and has no discrete spectrum. It follows from a similar computation as for \eqref{eq:energycoer1} that  $\tilH$ satisfies the coercivity estimate 
\begin{align}\label{eq:psicattemp3}
\begin{split}
\|\partial_\uprho u\|_{L^2}+\|\jap{\uprho}^{-1}u\|_{L^2}\lesssim \|\jap{\uprho}\tilH u\|_{L^2}
\end{split}
\end{align}
for any function $u$. Arguing as before but using \eqref{eq:psicattemp3} in place of \eqref{eq:psicattemp1} we get 
\begin{align}\label{eq:fancySobolevDartemp0}
\begin{split}
\|\Dar \calS \phi_\near\|_{\ell^\infty \calH^{s+2,-\frac{3}{2}}}\lesssim \|\calS\Dar \uH\phi_\near\|_{\ell^2\calH^{s,\frac{1}{2}}}.
\end{split}
\end{align}
On the other hand
\begin{align}\label{eq:fancySobolevDartemp1}
\begin{split}
\calS\phi_\near(\uprho)=c_0\nu_0(\uprho)+\nu_0(\uprho)\int_0^\uprho \nu_0^{-1}(y)\Dar \calS\phi_\near(y)\ud y,
\end{split}
\end{align}
for some constant $c_0$. It follows from the decay of $V$ and \eqref{eq:fancySobolevDartemp0} that 
\begin{align*}
\begin{split}
\|V\calS\phi_\near\|_{\ell^1\calH^{s,\frac{1}{2}}}\lesssim |c_0|+\|\calS\Dar \uH\phi_\near\|_{\ell^2\calH^{s,\frac{1}{2}}}.
\end{split}
\end{align*}
Finally, to estimate $c_0$, we integrate \eqref{eq:fancySobolevDartemp1} against $\uZ_i$ (or more precisely $\calS\calZ_i$) to get
\begin{align*}
\begin{split}
|c_0|&\lesssim \sum_i |\angles{\phi_\near}{\uZ_i}|+R_\ctf^2\|\calS\Dar \uH\phi_\near\|_{\ell^2\calH^{s,\frac{1}{2}}}\\
&\lesssim \sum_i |\angles{\phi}{\uZ_i}|+R_\ctf\|\phi_\far\|_{\ell^\infty \calH^{s+2,-\frac{3}{2}}}+ R_\ctf^2\|\calS\Dar \uH\phi_\near\|_{\ell^2\calH^{s,\frac{1}{2}}}.
\end{split}
\end{align*}
Combining with the earlier estimate on $\phi_\far$ this completes the proof of \eqref{eq:uH-wSob}.
\end{proof}

\section{The main profile, the source term, and the equations of motion}\label{sec:profile1}
\subsection{Equations of motion}\label{sec:eqnsofmotion}
Here we recall from \cite{LuOS1} the main conclusions of the first order formulation of the HVMC equation, and various local and global coordinates and the form of the equation in these coordinates. The corresponding derivations from \cite{LuOS1} are independent of the dimension so we will not reproduce the proofs. In the remainder of this section we will assume that the curves $\xi$ and $\ell$ are given and satisfy $|\ell|,|\dotxi|, |\dotell|\ll1$. In the next section, after describing the modified profile, we will set up an implicit function theorem that guarantees the existence of such parameters.
\subsubsection{The first order formulation}\label{subsubsec:firstorder} Recalling the setup and notation from sections~\ref{subsec:Riemcat} and~\ref{subsec:mainprofile}, we start with the parameterization
\begin{align}\label{eq:Psiwpdef1}
\begin{split}
\Psi_\wp(t,\rho,\omega)=(t,\xi+\gamma^{-1}P_\ell F(\rho,\omega)+P_\ell^\perp F(\rho,\omega))
\end{split}
\end{align}
of the profile in the flat region $\calC_\flatt:=\{\sigma_\temp\geq X^0+\delta_1\}$. The almost normal vector $N$ in \eqref{eq:Ndef1} can be calculated to be $N=(0,|A_{-\ell}\nu|^{-2}A_{-\ell}\nu)^\intercal$. With this definition we write
\begin{align*}
\begin{split}
\Phi=\Psi_\wp+\psi N.
\end{split}
\end{align*}
We also introduce the notation
\begin{align*}
\begin{split}
h_{\mu\nu}:=\bfeta(\partial_\mu\Psi_\wp,\partial_\nu\Psi_\wp)\Big|_{\substack{\dot{\ell} = 0 \\\dot{\xi} = \ell}},\qquad 0\leq \mu,\nu\leq 3,
\end{split}
\end{align*}
by which we mean that every appearance of $\dotell$, respectively $\dotxi$, is replaced by $0$, respectively $\ell$,  in calculating $\partial\Psi_\wp$.  Similarly, for $0\leq \mu,\nu\leq 3$, let
\begin{align*}
\begin{split}
g_{\mu\nu}:=\bfeta(\partial_\mu \Phi,\partial_\nu\Phi),\qquad k_{\mu\nu}=\bfeta(\partial_\mu\Psi_\wp,\partial_\nu\Psi_\wp).
\end{split}
\end{align*}
We define the momentum variable $\dotpsi$ as
\begin{align}\label{eq:psidotdef1}
\begin{split}
\dotpsi:=\bfeta(\sqrt{|g|}(g^{-1})^{0\nu}\partial_\nu\Phi,N)-\bfeta(\sqrt{|h|}(h^{-1})^{00}(1,\ell),N)-B\psi,
\end{split}
\end{align}
where 
\begin{align*}
\begin{split}
B\psi:=\bfeta(\frac{\partial}{\partial\psi}\sqrt{|g|}(g^{-1})^{0\nu}\partial_\nu\Phi\Big|_{\substack{\dot{\ell} = 0 \\\dot{\xi} = \ell}},N).
\end{split}
\end{align*}
See \cite[Section~3.2]{LuOS1} for how we arrive at this definition. Let
\begin{align*}
\begin{split}
M:=\pmat{-\frac{ (h^{-1})^{0j} }{ (h^{-1})^{00} } \partial_j&\frac{1}{\sqrt{|h|}(h^{-1})^{00}}\\-\sqrt{|h|}L&-\partial_j\frac{(h^{-1})^{j0}}{(h^{-1})^{00}}},
\end{split}
\end{align*}
where
\begin{align*}
\begin{split}
L:=\frac{1}{\sqrt{|h|}}\partial_j(\sqrt{|h|}(\hbar^{-1})^{jk}\partial_k)+|\secondff|^2, \qquad (\hbar^{-1})^{jk}:=(h^{-1})^{jk}-\frac{(h^{-1})^{j0}(h^{-1})^{0k}}{(h^{-1})^{00}},
\end{split}
\end{align*}
and let
\begin{align}\label{eq:K1storder1}
\begin{split}
\vecK=\pmat{K\\\dotK}=\pmat{-\bfeta((\dotell\cdot\nabla_\ell+(\dotxi-\ell)\cdot\nabla_\xi)\Psi_\wp,N)\\-\bfeta(\sqrt{|h|}(h^{-1})^{00}(0,\dotell),N)}.
\end{split}
\end{align}
It then follows from the analysis in \cite[Sections 3.1--3.4]{LuOS1} (specifically equation (3.24) and (3.45) with slightly different notation) that if $\vecpsi:=(\psi,\dot\psi)^\intercal$ and $\wp$ are sufficiently small, then
\begin{align}\label{eq:firstorder1}
\begin{split}
(\partial_t-M)\vecpsi=\vecK+\vecf=\vecF_1,
\end{split}
\end{align}
where $\vecf=(f,\dotf)^{\intercal}$ satisfies
\begin{align}\label{eq:vecfbound1}
\begin{split}
\vecf=\calO(\dotell\wp, \dotwp^2,(\partial_\Sigma^{\leq 2}\vecpsi)^2).
\end{split}
\end{align}
Equation \eqref{eq:firstorder1} is the first order formulation of the equations of motion in the interior. For the purpose of the implicit function theorem determining the parameters and the modulation equations, it is useful to also obtain the equation that results from pairing \eqref{eq:firstorder1} with the truncated generalized eigenfunctions of $M$. To this end, recall from \cite[Section 3.5]{LuOS1} that the (generalized) zero eigenfunctions of $M$ have the forms
\begin{align*}
\begin{split}
\vecfy_i=\pmat{\fy_i\\\dotfy_i},\quad \vecfy_{3+i}=\pmat{\fy_{3+i}\\\dotfy_{3+i}},\quad i=1,2,3,
\end{split}
\end{align*}
where
\begin{align*}
\begin{split}
\fy_i=\nu^i+\calO(|\ell|^2), \quad\dotfy_i=\calO(|\ell|),\quad \fy_{3+i}=\calO(|\ell|),\quad \dotfy_{3+i}=\sqrt{|h|}(h^{-1})^{00}\nu^i+\calO(|\ell|).
\end{split}
\end{align*}
These satisfy
\begin{align*}
\begin{split}
M\vecfy_i=0,\qquad M\vecfy_{3+i}=\vecfy_i,\quad i=1,2,3.
\end{split}
\end{align*}
Correspondingly we have the truncated eigenfunctions $\vecZ_j:=\chi \vecfy_j$, $j=1,\dots,6$, where $\chi$ is a smooth compactly supported cut-off function that satisfies $\chi(\rho)=1$ for $|\rho|\leq R_\ctf$ and $\chi(\rho)=0$ for $|\rho|\geq 2R_\ctf$. Here $R_\ctf$ is a large constant that we use to denote the radius at which we cutoff approximate eigenfuntions. To define the orthogonality conditions later on, we define the \emph{symplectic pairing}
\begin{align*}
\begin{split}
\vecbfOmega\equiv \vecbfOmega(\vecpsi)=(\bfOmega(\vecpsi,\vecZ_1),\dots,\bfOmega(\vecpsi,\vecZ_{6})),
\end{split}
\end{align*}
where for $\vecu_i=(u_i,\dotu_i)$, $i=1,2$,
\begin{align}\label{eq:bfOmegadef1}
\begin{split}
\bfOmega(\vecu_1,\vecu_2)=\int (u_1\dotu_2-\dotu_1 u_2)\ud \omega\,\ud\rho.
\end{split}
\end{align}
Note that in general $\bfOmega(M\vecu_1,\vecu_2)=-\bfOmega(\vecu_1,M \vecu_2)$. From \eqref{eq:firstorder1} and the analysis of \cite[Section 3.6]{LuOS1} (specifically equation (3.30), but with slightly different notation), it follows that if $\vecpsi$ and $\wp$ are sufficiently small, then
\begin{align}\label{eq:bfOmega1}
\begin{split}
\partial_t\vecbfOmega+\vecN = \pmat{D+R&R\\R&D+R}\dotwp+\vecH=:\vecF(\partial_\Sigma^{\leq 2}\vecpsi,\ell,\dotwp),
\end{split}
\end{align}
where $R=\vecO(\partial_\Sigma^{\leq2}\vecpsi,\ell,\dotwp)$, $\vecH=\calO((\partial_\Sigma^{\leq2}\vecpsi)^2)$, $\vecN=(\bfOmega(\vecpsi,M\vecZ_1),\dots,\bfOmega(\vecpsi,M\vecZ_{6}))$, and $D$ is an invertible matrix with entries 
\begin{align*}
\begin{split}
d_{ij}=\int\chi\nu^i\nu^j\sqrt{|h|}(h^{-1})^{00}\ud\rho\,\ud\omega\simeq\begin{cases}1,\quad&i=j\\o_{R_\ctf,\ell}(1),\quad&i\neq j\end{cases}.
\end{split}
\end{align*}
By the (calculus) implicit function theorem there is a function $\vecG$ such that for small $x$, $y$, $w$, and $q$, we have $q=\vecF(x,y,w)$ if and only if $w=\vecG(x,y,q)$. Moreover, $\vecG$ satisfies the estimate
\begin{align}\label{eq:vecGbound1}
\begin{split}
|\vecG(x,y,q)|\lesssim |q|+|x|^2.
\end{split}
\end{align}
It follow from \eqref{eq:bfOmega1} that
\begin{align}\label{eq:bfOmega2}
\begin{split}
\dotwp=\vecG(\partial_\Sigma^{\leq 2}\vecpsi,\ell,\partial_t\vecbfOmega+\vecN).
\end{split}
\end{align}
Equation \eqref{eq:bfOmega2} will be our starting point for imposing the orthogonality conditions and deriving the modulation equations. 

To deal with the contribution of the growing mode, we need to further decompose $\vecpsi$. Recalling that $\uH\varphibar_\mu=\mu^2\varphibar_\mu$, let 
\begin{align*}
\begin{split}
\vecZ_\pm:= c_\pm(\chi\varphibar_\mu,\mp \mu\sqrt{|h|}\vert_{\ell=0}\chi\varphibar_\mu)^\intercal,
\end{split}
\end{align*}
where $\chi$ is as above and the normalization constants $c_\pm$ are chosen so that
\begin{align*}
\begin{split}
\bfOmega(\vecZ_{+},\vecZ_{-})=-\bfOmega(\vecZ_{-},\vecZ_{+})=1.
\end{split}
\end{align*}
The vectors $\vecZ_\pm$ satisfy
\begin{align*}
\begin{split}
M\vecZ_\pm=\pm\mu\vecZ_\pm+\calE_\pm,
\end{split}
\end{align*}
where the errors $\calE_\pm$ consist of terms that are supported around $|\rho|\simeq R_\ctf$ and are exponentially decaying in $R_\ctf$, or have additional smallness in terms of the parameter $\ell$. Given functions $a_\pm(t)$, we then decompose $\vecpsi$ as
\begin{align}\label{eq:psiphi1}
\begin{split}
\vecpsi = \vecphi+a_{+}\vecZ_{+}+a_{-}\vecZ_{-}.
\end{split}
\end{align}
As with $\wp$, the coefficients $a_\pm$ will be determined later by imposing suitable orthogonality conditions. For now we mention that in view of \eqref{eq:firstorder1} and the decomposition above,
\begin{equation}\label{eq:apm1}
\frac{\ud}{\ud t}(e^{-\mu t}a_{+})=-\frac{\ud}{\ud t}\big(e^{-\mu t}\bfOmega(\vecphi,\vecZ_{-})\big)-e^{-\mu t}F_{+},\qquad\frac{\ud}{\ud t} (e^{\mu t}a_{-})=\frac{\ud}{\ud t}\big(e^{\mu t}\bfOmega(\vecphi,\vecZ_{+})\big)+e^{\mu t}F_{-},
\end{equation}
where
\begin{equation}\label{eq:F+1}
-F_{+}=\bfOmega(\vecF_1,\vecZ_{-})-\bfOmega(\vecphi,M\vecZ_{-}+\mu\vecZ_{-})+a_{+}\bfOmega(M\vecZ_{+}-\mu\vecZ_{+},\vecZ_{+})+a_{-}\bfOmega(M\vecZ_{-}+\mu\vecZ_{-},\vecZ_{-})
\end{equation}
and
\begin{equation}\label{eq:F-1}
-F_{-}=\bfOmega(\vecF_1,\vecZ_{+})-\bfOmega(\vecphi,M\vecZ_{+}-\mu\vecZ_{+})+a_{+}\bfOmega(M\vecZ_{+}-\mu\vecZ_{+},\vecZ_{+})+a_{-}\bfOmega(M\vecZ_{-}+\mu\vecZ_{-},\vecZ_{+}).
\end{equation}
\subsubsection{The second order equation}\label{subsubsec:2ndordereq} We introduce a number of coordinates systems and write the corresponding second order equation satisfied by $\phi$. We will use the words ``interior" and ``exterior" to refer to the interior and complement of a large compact set, respectively.

{\bf{Interior and exterior coordinates.}} In the interior we will work with the coordinates $(t,\rho,\omega)$ used above. In the exterior we want to parameterize the VMC surface as a graph over the asymptotic hyperplanes of the catenoid, so we start by parameterizing the profile as such a graph. Let $x^0(\sigma,x')$ be defined by the requirement that $(x^0(\sigma,x'),x')\in \barbsUpsigma_\sigma$. Then, in $\calC_\hyp$
\begin{align*}
\begin{split}
x^0=\sigma-R+\sqrt{|x'-\xi|^2+1}
\end{split}
\end{align*}
and in $\calC_\flatt$
\begin{align*}
\begin{split}
x^0=\sigma.
\end{split}
\end{align*}
The expression for $x^0$ in the transition region $\{|X^0-\sigma_\temp(X)|<\delta_1\}$ depends on the choice of the smoothed out minimum function $\frakm$ (see Section~\ref{subsec:mainprofile}). With this definition of $x^0$, we define $\scQ$ by the requirement that $(x^0(\sigma,x'),x',\scQ(x^0(\sigma,x'),x'))\in\Sigma_\sigma$, that is, so that the map 
\begin{align*}
\begin{split}
(\sigma,x')\mapsto (x^0(\sigma,x'),x',\scQ(x^0(\sigma,x'),x'))
\end{split}
\end{align*}
is a parameterization of the profile in the exterior region $\{|x'|\gg1\}$.
To derive a more explicit expression for $\scQ$ we define the non-geometric polar coordinates $(\tau,r,\theta)$ by 
\begin{align*}
\begin{split}
\tau=\sigma\quad\mand \quad x'=\xi(\sigma)+r\Theta(\theta).
\end{split}
\end{align*}
It follows that
\begin{align}\label{eq:exttaurcoords1}
\begin{split}
\begin{cases}x^0=\tau-R+\jap{r}\\x'=\xi(\tau)+r\Theta(\theta)\end{cases}\mathrm{~in~} \calC_\hyp\qquad
\end{split}
\end{align}
and
\begin{align*}
\begin{split}
\begin{cases}x^0=\tau\\x'=\xi(\tau)+r\Theta(\theta)\end{cases}\mathrm{~in~}\calC_\flatt.
\end{split}
\end{align*}
A small computation then shows that in $\calC_\hyp$ our parameterization becomes (see \cite[Section~4.1]{LuOS1}, specifically equation (4.4) where we use a slightly different notation)
\begin{align*}
\begin{split}
(\tau,r,\theta)\mapsto (\tau-R+\jap{r},\xi(\tau)+r\Theta(\theta),Q(rA_\ell\Theta(\theta)-\gamma \jap{r}\ell)),
\end{split}
\end{align*}
where $Q(y)=Q(|y|)$ is a parameterization of the Riemannian catenoid satisfying the ODE
\begin{align*}
\begin{split}
Q''(\tilr)+\frac{2}{\tilr}Q'(\tilr)-(1+(Q'(\tilr))^2)^{-1}(Q'(\tilr))^2Q''(\tilr)=0.
\end{split}
\end{align*}

{\bf{Non-geometric global coordinates.}} As shown in \cite[Section~4]{LuOS1} one can also define \emph{non-geometric global coordinates} $(\uptau,\uprho,\upomega)$, with $\uptau=t=\tau$, which agree with $(\tau,r,\theta)$ in $\calC_\hyp$ and with $(t,\rho,\omega)$ in $\calC_\flatt$. If $\Psi_\glbl$ denotes the parameterization of the profile in these global coordinates, and $\Psi_{\glbl,\alpha}:=\partial_\alpha \Psi_{\glbl} \Big|_{ \substack{\dot{\ell} = 0 \\\dot{\xi} = \ell} }$, we define
\begin{align*}
\begin{split}
h_{\alpha\beta}=\bfeta(\Psi_{\glbl,\alpha},\Psi_{\glbl,\beta}).
\end{split}
\end{align*}
In the interior this definition agrees with the one in Section~\ref{subsubsec:firstorder} and in the exterior, and with $m$ denoting the Minkowski metric, $Q_{\alpha}:=\partial_\alpha Q \Big|_{ \substack{\dot{\ell} = 0 \\\dot{\xi} = \ell} }$, and $Q^\alpha=(m^{-1})^{\alpha\beta}Q_\beta$, we have
\begin{align*}
\begin{split}
h_{\alpha\beta}= m_{\alpha\beta}+Q_\alpha Q_\beta,\quad (h^{-1})^{\alpha\beta}=(m^{-1})^{\alpha\beta}-\frac{Q^\alpha Q^\beta}{1+Q^\gamma Q_\gamma},\quad |h|=|m|(1+Q^\alpha Q_\alpha).
\end{split}
\end{align*}
Then as in \cite[Section~4]{LuOS1}, the equation satisfied by $\phi$ can be written as (recall that $\phi$ is obtained from $\psi$ by subtracting the contribution of the growing mode; cf. \eqref{eq:psidef1} and \eqref{eq:psiphi1})
\begin{align}\label{eq:phicalG1}
\begin{split}
\calP \phi = \sum_{i=0}^3\calG_i, 
\end{split}
\end{align}
where $\calG_i$ denotes term of order $j$ in $\phi$ (with $\calG_1\equiv0$), and the linear operator $\calP$ satisfies
\begin{align}\label{eq:calPdef1}
\begin{split}
\calP = \calP_h+\calO(\jap{\uprho}^{-2}\dotwp^{\leq 2})\partial^{2} +\calO(\jap{\uprho}^{-2}\dotwp^{\leq 2})\partial +\calO(\jap{\uprho}^{-6}\dotwp),
\end{split}
\end{align}
where
\begin{equation}\label{eq:calPhdef1}
\calP_h=\frac{1}{\sqrt{|h|}}\partial_\mu \big(\sqrt{|h|} (h^{-1})^{\mu\nu}\partial_\nu\big)+ V.
\end{equation}
The term $\calG_1$ in \eqref{eq:phicalG1} is zero because the linear terms are included in $\calP$. However, the \emph{source term} $\calG_0$ is nonzero because in general $\Psi_\glbl$ is not an exact solution of the HVMC equation. This source term can be decomposed as
\begin{align*}
\begin{split}
\calG_0= g_\main+ g_\pert,
\end{split}
\end{align*}
where
\begin{align}\label{eq:calG01}
\begin{split}
g_\main= \calO(\dotwp\jap{\uprho}^{-3})\quad\mand\quad g_\pert=\calO(\dotwp^{\leq 3}\jap{\uprho}^{-4}).
\end{split}
\end{align}
For future reference we also introduce the notation $\calP_h^\stat$ for the spatial part of $\calP_h$:
\begin{align}\label{eq:calPhstat1}
\begin{split}
\calP_h^\stat:=\frac{1}{\sqrt{|h|}}\partial_i \big(\sqrt{|h|} (h^{-1})^{ij}\partial_j\big)+ V,
\end{split}
\end{align}
where $i,j$ range over the spatial variables $(\uprho,\upomega)$. The eigenfunctions of $\calP_h^\stat$ will be denoted by $\fy_i^\stat$, $i=\mu,1,2,3$. These can be obtained by passing to the global geometric coordinates introduced below, for the parameters fixed at any given time, and using the expression for $\fybar_i$.
\begin{remark}\label{rem:calG01} The proof of \eqref{eq:calG01} is contained in \cite[Lemma~8.5]{LuOS1}. The only fact that needs to be pointed out is that in this proof, the main contribution comes from $(\partial_\uprho+\frac{1}{\uprho})\partial_\uptau Q$. The leading order part of $\partial_\uptau Q$ is~$\calO(\dotwp\jap{\uprho}^{-1})$, but since $(\partial_\uprho+\frac{1}{\uprho})\uprho^{-1}=0$, we do not see any $\calO(\jap{\uprho}^{-2})$ terms in $g_\main$. On the other hand, terms of the order $\calO(\dotwp \jap{\uprho}^{-3})$ already appear in the expression for $\Box_m$ in the $(\uptau,\uprho,\upomega)$ coordinates in the exterior. 
\end{remark} 
In the proof of ILED in Section~\ref{sec:ILED} we will need to work with $\calP$ with coefficients fixed at a given time. The following properties of the fixed time operator, which are discussed in \cite[Section~4.2.3]{LuOS1} are relevant for this purpose. Supose $|\dotwp^{\leq 2}|\lesssim \epsilon \uptau^{-\gamma-1}$ for some $\gamma>1$ (this will be the case in our bootstrap argument), and write the operator $\calP$ as
\begin{align*}
\begin{split}
\calP\psi=\uppi_q^{\mu\nu}\partial^2_{\mu\nu}\psi+\uppi_l^\mu\partial_\mu\psi+\uppi_c\psi.
\end{split}
\end{align*}
For a given time value $t_2$, let
\begin{align*}
\begin{split}
&\pibar_q^{\mu\nu}=(h^{-1})^{\mu\nu}\vert_{\uptau=t_2}, \quad \pibar_l^\nu=|h|^{-\frac{1}{2}}\partial_\mu(|h|^{\frac{1}{2}}(h^{-1})^{\mu\nu})\vert_{\uptau=t_2},\quad \pibar_c=V\vert_{\substack{\dot{\ell} = 0 \\ \dot{\xi} = \ell\\ \uptau=t_2}},\\
& \ringpi_q=\uppi_q-\pibar,\quad \ringpi_l=\uppi_l-\pibar,\quad \ringpi_c=\uppi_c-\pibar,
\end{split}
\end{align*}
and correspondingly decompose $\calP$ as $\calP=\calP_0+\calP_\pert$ with
\begin{align}\label{eq:calPP0Ppertdecomp1}
\begin{split}
\calP_0=\pibar_q^{\mu\nu}\partial^2_{\mu\nu}+\pibar_l^\mu\partial_\mu+\pibar_c, \qquad \calP_\pert=\ringpi_q^{\mu\nu}\partial^2_{\mu\nu}+\ringpi_l^\mu\partial_\mu+\ringpi_c.
\end{split}
\end{align}
Then $|\ringpi_q|, |\ringpi_l|, |\ringpi_c|\lesssim \jap{\uptau}^{-\gamma}$, and, with $y$ denoting rectangular coordinates with respect to the spatial variables~$(\uprho,\uptheta)$,
\begin{align}\label{eq:Ppertcoeffbound1}
\begin{split}
\sup_y(\jap{\uprho}^{2}|\ringpi_q^{\uptau\uptau}|+|\ringpi_q^{\uptau y}|+|\ringpi_q^{yy}|)\lesssim \epsilon \jap{\uptau}^{-\gamma}. 
\end{split}
\end{align}
Moreover, in the hyperboloidal part of the foliation, $\calP_\pert$ has the more precise structure
\begin{align*}
\begin{split}
\ringa\partial_\uptau(\partial_\uprho+\frac{1}{\uprho})+\ringa^{\mu\nu}\partial_{\mu\nu}+\ringb^\mu \partial_\mu +\ringc,
\end{split}
\end{align*}
where,
\begin{align}\label{eq:calPP0Ppertdecomp2}
\begin{split}
&|\ringa|, |\ringa^{yy}|\lesssim \epsilon \jap{\uptau}^{-\gamma}, \quad  |\partial_y\ringa^{yy}|, |\ringa^{\uptau y}|, |\ringb^y|\lesssim \epsilon\jap{\uptau}^{-\gamma}\uprho^{-1},\quad |\ringa^{\uptau\uptau}|, |\ringb^\uptau|\lesssim \epsilon \jap{\uptau}^{-\gamma}\uprho^{-2},\\
&|\ringc|\lesssim \epsilon\jap{\uptau}^{-\gamma} \uprho^{-4}.
\end{split}
\end{align} 

{\bf{Renormalized graph formulation.}} In the exterior region where the almost normal vector $N$ is parallel to $\frac{\partial}{\partial X^{4}}$, we will often use a renormalized parameterization of the solution as a graph over the limiting hyperplanes. To wit, in the notation introduced above, the vector $N$ in this region is given by 
\begin{align}\label{eq:sgraphdef1}
\begin{split}
N=s\frac{\partial}{\partial X^4},\qquad s = (1+(m^{-1})^{\mu\nu}Q_\mu Q_\nu)^{\frac{1}{2}},
\end{split}
\end{align}
so introducing $\fy=s\phi$ we see that $u=Q+\fy$ satisfies the graph formulation of the mean curvature equation
\begin{align*}
\begin{split}
\frac{1}{\sqrt{|m|}}\partial_\mu\Big(\frac{\sqrt{|m|}(m^{-1})^{\mu\nu}\partial_\nu u}{\sqrt{1+(m^{-1})^{\alpha\beta}\partial_\alpha u\partial_\beta u}}\Big)=0.
\end{split}
\end{align*}
The advantage of this formulation is that the nonlinear terms are easy to calculate. Moreover, in view of the decay of derivatives of $Q$, estimates on either of $\fy$ and $\phi$ can be easily translated into estimates on the other one. In terms of $\fy$ the equation for $u$ above can be written as (here the subscript $g$ stands for \emph{graph})
\begin{align}\label{eq:calPgcalP1}
\begin{split}
\calP_g\fy=s\calP\phi=\sum_{i=0}^3\calG_{g,i},\qquad \calP=\calP_g+s^{-1}[\calP_g,s],
\end{split}
\end{align}
where $\calG_{g,i}$ is of order $i$ in $\fy$ (note that $\calG_{g,0}=s\calG_0$ and $\calG_{g,1}=0$) and
\begin{align}\label{eq:callP1}
\begin{split}
\callP_g&=\Box_m-(1+\nabla^\alpha Q \nabla_\alpha Q)^{-1}(\nabla^\mu Q)(\nabla^\nu Q)\nabla_{\mu}\nabla_\nu-2(1+\nabla^\alpha Q \nabla_\alpha Q)^{-1}(\nabla^{\mu}\nabla^\nu Q)(\nabla_\mu Q)\nabla_\nu\\
&\quad+2(1+\nabla^\alpha Q\nabla_\alpha Q)^{-2}(\nabla^\nu Q)(\nabla^\mu Q)(\nabla^\lambda Q)(\nabla_\lambda \nabla_\mu Q)\nabla_\nu.
\end{split}
\end{align}
In \eqref{eq:callP1} the notation $\nabla$ is used for the covariant derivative with respect to $m$. It follows that
\begin{align}\label{eq:callP2}
\begin{split}
\callP_g\psi&=\Box_m\psi+\Err_{\callP_g}(\psi),
\end{split}
\end{align}
where for some symmetric bounded  $p_2$ and bounded $p_1$,
\begin{align}\label{eq:ErrcallP1}
\begin{split}
\Err_{\callP_g}\psi&= \jap{r}^{-4}p_2^{\mu\nu}\partial^2_{\mu\nu}\psi+\jap{r}^{-5}p_1^\mu \partial_\mu\psi.
\end{split}
\end{align}

{\bf{Geometric global coordinates.}} The last set of coordinates we use are the \emph{geometric global coordinates} $(\tiluptau,\tiluprho,\tiluptheta)$, which are defined with respect to fixed values of the parameters $\ell\equiv\ellbar$, $\xi\equiv\xibar+\tau\ellbar$. These fixed values are usually taken to be the value of the parameters at a fixed $\tau$. The advantage of these coordinates is that in them the operator $\calP_0$ introduced in \eqref{eq:calPP0Ppertdecomp1} is of product form. Specifically, in these coordinates $\calP_0$ takes the form (see \cite[Section 4.2.4]{LuOS1})
\begin{align}\label{eq:tilDelta1}
\begin{split}
\calP_0=-\partial_{\tiluptau}^2+\tilDelta+V(\tilrho),
\end{split}
\end{align}
where, with $\ringsDelta$ denoting the Laplacian on the round sphere $\bbS^{2}$,
\begin{align}\label{eq:tilDelta2}
\begin{split}
\tilDelta=\frac{\sqrt{1+\jap{\tiluprho}^2}}{\jap{\tiluprho}^{3}}\partial_\tiluprho(\jap{\tiluprho}\sqrt{1+\jap{\tiluprho}^2}\partial_\tiluprho)+\frac{1}{\jap{\tiluprho}^2}\ringsDelta.
\end{split}
\end{align}

{\bf{The Darboux transform.}} We introduce the following definition for the Darboux transform for reference in the proof of ILED in Section~\ref{sec:ILED}. Consider the geometric global coordinates $(\tiluptau,\tiluprho,\tiluptheta)$ defined with respect to a fixed time $t_2$. Let $\calS_j$ denote the $L^2$ projection (with respect to the standard metric on $\bbS^{2}$) onto the space $\calY_j$ of the eigenfuctions of $\ringsDelta$ with eigenvalue $-j(j+1)$. Note that the zero  eigenfunctions $\nu^j(\tiluprho,\tiluptheta)$ of $\tilDelta+V(\tiluprho)$ (see \eqref{eq:tilDelta1} and \eqref{eq:tilDelta2}) satisfy $\nu^j(\tiluprho,\tiluptheta)= \nu_0(\tiluprho)\tilde{\nu}^j(\tiluptheta)$ where $\nu_0=\jap{\tiluprho}^{-2}$ and $\tilde{\nu}^j\in\calY_{1,j}$ and $\calY_1=\oplus_{j=1}^3\calY_{1,j}$ as usual. Given a function $f(\tiluptau,\tiluprho,\tiluptheta)$ we define $\Dar f$ by
\begin{align}\label{eq:Dardef1}
\calS_j\Dar f (\tiluptau,\tiluprho,\tiluptheta)=\begin{cases}
\calS_jf (\tiluptau,\tiluprho,\tiluptheta),\qquad &j\neq 1\\
\nu_0(\tiluprho)\partial_\tiluprho(\nu_0^{-1}(\tiluprho)\calS_1f(\tiluptau,\tiluprho,\tiluptheta)),\qquad &j=1
\end{cases}.
\end{align}
Now suppose $f\equiv f(\uprho,\uptheta)$ is a function given in the non-geometric global coordinates and supported in the flat region $\{|\tiluprho|< R\}$. As shown in \cite[equation~(4.4)]{LuOS1}, in this region the geometric and non-geometric global coordinates are related as
\begin{align*}
\begin{split}
\tiluptau=\gammabar^{-1}\uptau-\ellbar\cdot F(\uprho,\uptheta),\quad \tiluprho=\uprho,\qquad \tiluptheta =\uptheta.
\end{split}
\end{align*}
It follows that the definition $\Dar f$ is independent of the time $t_2$ with respect to which the coordinates $(\tiluptau,\tiluprho,\tiluptheta)$ are defined, and is given by
\begin{align*}
\calS_j\Dar f (\uprho,\uptheta)=\begin{cases}
f (\uprho,\uptheta),\qquad &j\neq 1\\
\nu_0(\uprho)\partial_\uprho(\nu_0^{-1}(\uprho)\calS_1f(\uprho,\uptheta)),\qquad &j=1
\end{cases},
\end{align*}
where the $L^2(\bbS^{n-1})$ projections $\calS_j$ are defined with respect to the coordinates $(\uprho,\uptheta)$. In our applications, the Darboux transform will be applied to the main part of $\calG_0$ from the source term in the interior. This can be written (see Lemma~\ref{lem:calG01}) as $\dotwp^{(\leq2)}(\uptau)\fybar_i(\uprho,\uptheta)$, $i=1,2,3$. It follows from the foregoing discussion that $\Dar(\dotwp^{(\leq2)}\fybar_i)=\dotwp^{(\leq2)}\Dar\fybar_i+(|\ell|\dotwp^{(\leq3)})\fybar_i$, where the second term comes from the commutator $[\Dar,\dotwp^{(\leq2)}]$, and $\Dar\fybar_i$ vanishes in the flat region of the foliation.
\subsection{Structure of the source term}
Here we provide a more detailed description of the structure of the source term $\calG_0$ in \eqref{eq:calPdef1}. As we already discussed in equation \eqref{eq:calG01} and Remark~\ref{rem:calG01}, the source term is bounded by $|\dotwp||\uprho|^{-3}$ for large $|\uprho|$. More precisely (see Remark~\ref{eq:calPdef1}), this holds for $|\uprho|> R$, that is, outside the transition region, where the foliation is asymptotically null. However, since we will use the inverse powers of $R_\ctf$ to close some of our estimates, it is important to have a size estimate for the source term in the region $\{|\uprho|\leq R\}$ as well. The structure we need for $\calG_0$ is summarized in the following lemma.
\begin{lemma}\label{lem:calG01}
The source term $\calG_0$ in \eqref{eq:calG01} satisfies the following properties, where the implicit constants in the $\calO$ terms are independent of $R$ and $R_\ctf$ unless indicated by a subscript:
\begin{enumerate}
\item In the region $\{|\uprho|\leq R+C\}$,
\begin{align*}
\begin{split}
\calG_0= \calO(\jap{\uprho}^{-2}\dotwp^{(\geq 1)})+\calO_R(|\ell|\dotwp^{(\geq 1)}).
\end{split}
\end{align*}
\item For $|\uprho|\geq R$, 
\begin{align*}
\begin{split}
\calG_0=\calO( |\dotwp^{(\geq 1)}||\uprho|^{-3}).
\end{split}
\end{align*}
\item For $|\uprho|\leq R+C$,
\begin{align*}
\begin{split}
\Dar \calG_0= \calO_R(|\ell|\dotwp^{(\geq 1)})+\chi_{\{R-C\leq\cdot\leq R+C\}}\calO(|\uprho|^{-2}\dotwp^{(\geq 1)}).
\end{split}
\end{align*}
\end{enumerate}
\end{lemma}
\begin{proof}
The exterior estimate in $\{|\uprho|> R\}$ was already discussed above. See Remark~\ref{rem:calG01}. For the interior, the main observations are that except for terms that come with a factor of $\ell$ or decay exponentially, the source term is simply a linear combination of $\partial_\uptau(\dotxi\fybar_i)$, $i=1,2,3$ (see \eqref{eq:K1storder1} and \eqref{eq:firstorder1}), and by design $\Dar\fybar_i=0$ for $|\uprho|\leq R$. See also the discussion following \eqref{eq:Dardef1} above. The reason the spatial decay is only $|\uprho|^{-2}$ (instead of $|\uprho|^{-3}$) in the transition region $\{R-C\leq|\uprho|\leq R+C\}$ is that in this region $\Dar$ contains $\partial_\uptau$ derivatives as well, and these do not improve the spatial decay.
\end{proof}
\section{Modified profile construction}\label{sec:modprof}
In this section we introduce a correction to the main profile from Section~\ref{subsec:mainprofile}. Its purpose is to remove the slowly decaying (in $\uprho$) error $g_\main$ from \eqref{eq:calG01}. To see the need for removing this error, observe (for instance by integration along characteristics) that a solution to the flat wave equation
\begin{align*}
\begin{split}
\Box_m \phi = \jap{\uptau}^{-\alpha}\jap{\uprho}^{-3},
\end{split}
\end{align*}
with $\alpha$ sufficiently large, will have decay $\jap{\uptau}^{-2}$ which is not twice integrable. Therefore, since $\phi$ appears linearly in the equation for $\dotwp$, and since we need a twice integrable decay rate for $\dotwp$, we need to make sure that the contribution of $g_\main$ does not enter the modulation equations. 

Let
\begin{align}\label{eq:gextdef1}
\begin{split}
g_\Ext=\chi_1\chi_2g_\main
\end{split}
\end{align}
where $\chi_1(\uprho)$ is a cutoff function supported in the region $\{|\uprho|\geq \uprho_0\}$, $\uprho_0\gg1$, and $\chi_2(\uptau)$ is a smooth cutoff which is equal to one for $\uptau\geq -2$ and equal to zero for $\uptau\leq -3$.
The profile correction $P$ is constructed as the solution to (recall the definition of $\calP_h$ from \eqref{eq:calPhdef1})
\begin{align*}
\begin{split}
\calP_h P = g_\Ext+\mathrm{acceptable~error}.
\end{split}
\end{align*}
The cutoffs in the definition of $g_\Ext$ are introduced so that the profile does not enter the modulation equations which are determined in the interior. The reason we cannot define $P$ to be an exact solution of $\calP_hP=g_\Ext$ is that in this case we have no way to ensure that $P$ satisfies suitable orthogonality conditions that lead to decay of $P$. Therefore, the main task is to modify the right-hand side of the equation $\calP_hP=g_\Ext$ in a way that leads to decay for $P$ and that the resulting errors in the equation for $$\varepsilon:=\phi-P$$ allow for a twice integrable decay rate for $\varepsilon$.

Turning to the details, let $P$ be the solution, with vanishing initial data, to the equation (recall that $(\uptau,\uprho,\upomega)$ denote the non-geometric global coordinates)
\begin{align}\label{eq:Pdef1}
\begin{split}
\calP_h P = g_\Ext+ \txtg_1+\partial_\uptau\txtg_2,
\end{split}
\end{align}
where $\txtg_1$ and $\txtg_2$ are to be specified. 
The remainder $\varepsilon=\phi-P$ satisfies 
\begin{align}\label{eq:varepeq1}
\begin{split}
\calP\varepsilon=(\calP\phi-g_\Ext)-\txtg_1-\partial_{\uptau}\txtg_2-(\calP-\calP_h)P.
\end{split}
\end{align}
To determine $\txtg_1$ and $\txtg_2$ let $W_i$ and $Y_i$, $i\in\{\mu,1,2,3\}$, be suitable proxies, to be specified, for the eigenfunctions $\fy_i^\stat$ for $\calPstato$ (recall \eqref{eq:calPhstat1}). Testing \eqref{eq:Pdef1} against $Y_i$ we get (recall that $\Sigma_c$ denotes the slice $\{\uptau=c\}$)
\begin{align}\label{eq:txtgtemp1}
\begin{split}
&\int_{\Sigma_\uptau}\Big((h^{-1})^{0\nu}(Y_i\partial_\nu P -P\partial_\nu Y_i )-\txtg_2 Y_i\Big)\sqrt{|h|}\ud x\Big\vert_{\uptau_1}^{\uptau_2}\\
&=\int_{\uptau_1}^{\uptau_2}\int_{\Sigma_{\uptau}}\Big((g_\Ext+\txtg_1)Y_i-P\calP_hY_i-\frac{\partial_\uptau(\sqrt{|h|}Y_i)}{\sqrt{|h|}}\txtg_2\Big)\sqrt{|h|}\ud x \ud\uptau,
\end{split}
\end{align}
where $\ud x=\ud\upomega\ud\uprho$. Similarly, integrating the identity
\begin{align*}
\begin{split}
\frac{1}{\sqrt{|h|}}\partial_\nu\big(\sqrt{|h|}\frac{(h^{-1})^{0\nu}}{(h^{-1})^{00}}PW_i\big)=\frac{1}{\sqrt{|h|}}PW_i\partial_\nu\big(\sqrt{|h|}\frac{(h^{-1})^{0\nu}}{(h^{-1})^{00}}\big)+\frac{(h^{-1})^{0\nu}}{(h^{-1})^{00}}(W_i\partial_\nu P+P\partial_\nu W_i)
\end{split}
\end{align*}
we get
\begin{align}\label{eq:txtgtemp2}
\begin{split}
&\int_{\Sigma_{\uptau}}PW_i\sqrt{|h|}\ud x\Big\vert_{\uptau_1}^{\uptau_2}\\
&=\int_{\uptau_1}^{\uptau_2}\int_{\Sigma_{\uptau}}\Big(\frac{1}{\sqrt{|h|}}PW_i\partial_\nu\big(\sqrt{|h|}\frac{(h^{-1})^{0\nu}}{(h^{-1})^{00}}\big)+\frac{(h^{-1})^{0\nu}}{(h^{-1})^{00}}(W_i\partial_\nu P+P\partial_\nu W_i)
\Big)\sqrt{|h|}\ud x\ud\uptau.
\end{split}
\end{align}
Comparing \eqref{eq:txtgtemp1} and \eqref{eq:txtgtemp2} with the goal of achieving the orthogonality condition 
\begin{align}\label{eq:Porthtemp1}
\begin{split}
\int_{\Sigma_{\uptau}}PW_i\sqrt{|h|}\ud x=0
\end{split}
\end{align}
leads to the requirements
\begin{align}\label{eq:txtgtemp3alt}
\begin{split}
&\int_{\Sigma_{\uptau}}\txtg_2 Y_i\sqrt{|h|}\ud x \\
&= \int_{\Sigma_{\uptau}}\Big[(h^{-1})^{0\nu}((Y_i+\frac{W_i}{(h^{-1})^{00}})\partial_\nu P+P(\frac{1}{(h^{-1})^{00}}\partial_\nu W_i-\partial_\nu Y_i))\\
&\phantom{= \int_{\Sigma_{\uptau}}\Big[}+\frac{1}{\sqrt{|h|}}PW_i\partial_\nu\big(\sqrt{|h|}\frac{(h^{-1})^{0\nu}}{(h^{-1})^{00}}\big)
\Big]\sqrt{|h|}\ud x
\end{split}
\end{align}
and
\begin{align}\label{eq:txtgtemp4}
\begin{split}
\int_{\Sigma_{\uptau}}\txtg_1 Y_i\sqrt{|h|}\ud x=\int_{\Sigma_{\uptau}}\Big[P\calP_hY_i-g_\Ext Y_i+\frac{\partial_\uptau(\sqrt{|h|}Y_i)}{\sqrt{|h|}}\txtg_2\Big]\sqrt{|h|}\ud x.
\end{split}
\end{align}
However, to avoid any potential loss of regularity, it is convenient to introduce extra smoothing in time in the definition of $\txtg_2$.\footnote{We could also introduce smoothing for $\txtg_1$ but this is not necessary as the only potential loss of regularity occurs in estimating $\RbfT\partial_\uptau g_2$ in the energy estimate, where the extra $\RbfT$ comes from the degeneracy of the $LE$ norm due to trapping. See Propositions~\ref{prop:LED1} and~\ref{prop:EILEDfinal1} below.} This is similar to the smoothing used for the parameters $\wp$ in \cite[Section~3]{LuOS1}, which we will also use in defining our parameters $\wp$ in Section~\ref{subsec:wpeq} below. It will lead to a modification of the orthogonality condition \eqref{eq:Porthtemp1} for $P$ that will still be acceptable for our estimates. Let $k_P\in C^\infty_c(\bbR)$ be a smooth bump function supported in $[0,1]$ with $\int_0^1k_P(s) \ud s=1$, and let $S$ be the smoothing operator (for $h(t)$ defined for~$t\geq -1$)
\begin{align}\label{eq:Sdef1}
\begin{split}
(S_Ph)(t)=\int_{\bbR}\chi_{[-1,\infty)}(s)h(s)k_P(t-s)\ud s,\qquad t\geq -1.
\end{split}
\end{align}
It follows that $S_P-I=\frac{\ud}{\ud t}\tilS_P$ with 
\begin{align}\label{eq:tilSdef1}
\begin{split}
(\tilS_P h)(t)=\int_\bbR\chi_{[-1,\infty)}(s)h(s)\tilk_P(t-s)\ud s,\qquad t\geq-1,
\end{split}
\end{align}
where $\tilk_P(r)=0$ for $r<0$ and $\tilk_P(r):=-\int_r^\infty k_P(s)\ud s$ for $r\geq0$. It will become clear (see \eqref{eq:Porth1} below) that it is beneficial to assume that $k_P$ satisfies the moment condition
\begin{align}\label{eq:kPmoment1}
\begin{split}
\int_0^1 sk_P(s)\ud s=0.
\end{split}
\end{align}
It then follows that
\begin{align*}
\begin{split}
(S_P-I)=\frac{\ud}{\ud t}\tilS_P=\frac{\ud^2}{\ud t^2}\tiltilS_P,
\end{split}
\end{align*}
where 
\begin{align*}
\begin{split}
(\tiltilS_Ph)(h)=\int_{\bbR}\chi_{[-1,\infty)}(s)h(s)\tiltilk_P(t-s)\ud s,\qquad t\geq -1,
\end{split}
\end{align*}
and $\tiltilk$ with $\supp\tiltilk_P\subseteq [0,1]$ is given by
\begin{align*}
\begin{split}
\tiltilk_P(r)=\begin{cases}\int_r^\infty(s-r) k(s)\ud s,&\qquad r\geq0 \\0,&\qquad r\leq 0\end{cases}.
\end{split}
\end{align*}
Indeed, the moment condition implies that
\begin{align*}
\begin{split}
\tiltilk_P(r)&:=\int_0^r\tilk_P(s)\ud s  = -\int_0^r(1-\int_0^sk_P(s')\ud s')\ud s=-r+\int_0^r(r-s')k_P(s')\ud s'\\
&=\int_r^\infty (s-r)k_P(s')\ud s'.
\end{split}
\end{align*}
We then replace the condition \eqref{eq:txtgtemp3alt} by the requirement that
\begin{align}\label{eq:txtgtemp3}
\begin{split}
&\int_{\Sigma_{\uptau}}\txtg_2 Y_i\sqrt{|h|}\ud x \\
&= S_P\int_{\Sigma_{\uptau}}\Big[(h^{-1})^{0\nu}((Y_i+\frac{W_i}{(h^{-1})^{00}})\partial_\nu P+P(\frac{1}{(h^{-1})^{00}}\partial_\nu W_i-\partial_\nu Y_i))\\
&\phantom{= \int_{\Sigma_{\uptau}}\Big[}+\frac{1}{\sqrt{|h|}}PW_i\partial_\nu\big(\sqrt{|h|}\frac{(h^{-1})^{0\nu}}{(h^{-1})^{00}}\big)
\Big]\sqrt{|h|}\ud x
\end{split}
\end{align}
Assuming $P$ is zero initially, it follows that the orthogonality condition \eqref{eq:Porthtemp1} for $P$ gets replaced by
\begin{align}\label{eq:Porth1}
\begin{split}
\int_{\Sigma_{\uptau}}PW_i\sqrt{|h|}\ud x&=-\tiltilS_P\partial_\uptau \int_{\Sigma_{\uptau}}\Big[(h^{-1})^{0\nu}((Y_i+\frac{W_i}{(h^{-1})^{00}})\partial_\nu P+P(\frac{1}{(h^{-1})^{00}}\partial_\nu W_i-\partial_\nu Y_i))\\
&\phantom{=-\tiltilS_P\partial_\uptau  \int_{\Sigma_{\uptau}}\Big[}+\frac{1}{\sqrt{|h|}}PW_i\partial_\nu\big(\sqrt{|h|}\frac{(h^{-1})^{0\nu}}{(h^{-1})^{00}}\big)
\Big]\sqrt{|h|}\ud x.
\end{split}
\end{align}
Note that the reason we required $k_P$ in the definition of $S_P$ to satisfy the moment condition \eqref{eq:kPmoment1} is to have the derivative $\partial_\uptau$ on the right-hand side of \eqref{eq:Porth1} above. 
To ensure the conditions \eqref{eq:txtgtemp4} and \eqref{eq:txtgtemp3} we let 
\begin{align}\label{eq:giorthdef1}
\begin{split}
\txtg_i = \sum_{j=\mu,1,2,3}c_{i,j}(\uptau)\chi \fy_j^\stat,\qquad i=1,2,
\end{split}
\end{align}
so that \eqref{eq:txtgtemp3} and \eqref{eq:txtgtemp4} become an invertible linear system for the coefficients $c_{i,j}$, provided $Y_i$ and $W_i$ are approximations to $\fy_i^\stat$. Here $\chi$ is a cutoff to a large compact region in the interior which is equal to one for $|\uprho|\leq R_\ctf/3$ and equal to zero for $|\uprho|\geq 2R_\ctf/3$. Note that in view of \eqref{eq:txtgtemp3}, in terms of decay in $\uptau$, $c_{2,j}$ are at best comparable to a small multiple of $P$ (which, as explained earlier, cannot decay faster than $\uptau^{-2}$). However, only $\partial_\uptau c_{2,j}$, not $c_{2,j}$, appear as errors in equation \eqref{eq:Pdef1}, and we can expect $\partial_{\uptau}$ to gain extra decay. For $c_{1,j}$ the only term on the right-hand side of \eqref{eq:txtgtemp4} that a priori may not have better decay than $P$ is the term  involving $\calP_h Y_i$. To extract extra decay for this term we define $Y_i=\chi_{\lesssim \uptau} \fy_i^\stat$, where $\chi_{\lesssim \uptau}$ is a cutoff to the region $\{|\uprho|\lesssim \uptau\}$. That is, we define
\begin{align}\label{eq:WYdef1}
\begin{split}
&W_\mu:=Y_\mu:=\chi_{\lesssim \uptau} \fy_\mu^\stat;\qquad W_i:=\chi_{\leq R_\ctf}\fy_i^\stat\mand Y_i:=\chi_{\lesssim \uptau} \fy_i^\stat,\qquad i=1,2,3,
\end{split}
\end{align}
with $\chi_{\leq R_\ctf}$ denoting a cutoff to the region $\{|\uprho|\leq R_\ctf\}$ for a large constant $R_\ctf$. Note that we have chosen $W_\mu=Y_\mu$ to obtain a homogeneous ODE for $\int_{\Sigma_{\uptau}}PW_\mu\sqrt{|h|}\ud x$ with zero initial conditions. Although $W_\mu$ is not compactly supported, we will be able to use this orthogonality condition in view of the exponential decay of $\fy_\mu$. We summarize our construction in the following proposition and with $P$ defined there, we let $\varepsilon:=\phi-P$. 
\begin{proposition}\label{prop:propdefP1}
Given smooth parameters $\ell$ and $\xi$ defined for $\uptau \in[0,\uptau_0]$ and satisfying\footnote{There is nothing special this decay rate and the only reason we have made this choice is to be consistent with our bootstrap assumptions later.}
\begin{align*}
\begin{split}
|\dotwp^{(k)}|\leq C_k \epsilon \uptau^{-\frac{9}{4}+\frac{\kappa}{2}},\quad \forall  k\geq 1,
\end{split}
\end{align*} 
let $g_\Ext$ be as in \eqref{eq:gextdef1} and $W_i$ and $Y_i$ as in \eqref{eq:WYdef1}. Then for $\uptau\in[-1,\uptau_0]$ there exist a smooth $P$, $g_1$, $g_2$, and $c_{ij}$ such that $g_i$, $c_{ij}$ satisfy \eqref{eq:giorthdef1}, \eqref{eq:txtgtemp3}, and \eqref{eq:txtgtemp4}, and $P$ satisfies \eqref{eq:Pdef1} with vanishing data at $\uptau=0$. 
\end{proposition}
\begin{proof}
Since \eqref{eq:Pdef1} is linear, the only difficulty is the dependence of $g_i$ on $P$ through \eqref{eq:txtgtemp3} and \eqref{eq:txtgtemp4}. To deal with this we iteratively construct $P$ in small time steps of length $\uptau_\ast\ll1$, and use the smallness of $\uptau_\ast$ to close the existence iteration scheme in each step. Turning to the details let $\tilP$ be the solution with vanishing data to
\begin{align*}
\begin{split}
\calP_h \tilP=g_\Ext
\end{split}
\end{align*}
on $[-1,\uptau_\ast]$. Next we claim that if $\uptau_\ast$ is sufficiently small, then there is a solution $\tiltilP$ on $[0,\uptau_\ast]$ with zero initial data to
\begin{align*}
\begin{split}
\calP_h \tiltilP=\txtg_1+\partial_\uptau\txtg_2,
\end{split}
\end{align*}
where $\txtg_j$ are defined so that \eqref{eq:txtgtemp4} and \eqref{eq:txtgtemp3} hold with $P$ replaced by $\tilP+\tiltilP$ for $\uptau\geq0$ and by $P$ for $\uptau<0$. To see this we set up a Picard iteration by solving
\begin{align*}
\begin{split}
\calP_h \tiltilP_j=\txtg_{1,j}+\partial_\uptau\txtg_{2,j},
\end{split}
\end{align*}
where $g_{i,j}$ are defined by \eqref{eq:txtgtemp4} and \eqref{eq:txtgtemp3} with $P$ replaced by $\tilP+\tiltilP_{j-1}$ for $\uptau\geq0$ and by $\tilP$ for $\uptau<0$. To close the iteration, we need to prove uniform Sobolev estimates for $\tiltilP_j$ and the differences $\tiltilP_i-\tiltilP_j$. This is achieved by the energy estimates for $\calP_h$ coming from multiplying the equation by $\partial_\uptau \tiltilP_j$ (see for instance Section~\ref{sec:ILED}, specifically, Proposition~\ref{prop:LED1} for a more elaborate version). Here we note that $\calP_h$ is almost stationary in the sense that the coefficients depend on $\uptau$ only through $\ell$ and $\xi$ which satisfy the estimates in the statement of the proposition. Moreover, the contribution of the lower order terms in $\calP_h$ as well as the part of $\txtg_1$ and $\txtg_2$ that depend on $\tiltilP_{j-1}$ can be made small by choosing $\uptau_\ast$ small. Note that the smallness of $\uptau_\ast$ is independent of the size of $\tilP$. It follows that $P=\tilP+\tiltilP$ satisfies \eqref{eq:Pdef1} on $[0,\uptau_\ast]$. 

Now suppose we have constructed $P$ satisfying \eqref{eq:Pdef1} on $[0,\uptau_1]$, with $\uptau_1<\uptau_0$.
We will use the same scheme as above to extend $P$ to $\min\{\uptau_0,\uptau_1+\uptau_\ast\}$. For simplicity of notation we assume that $\uptau_1+\uptau_\ast\leq \uptau_0$. First we let $\tilP$ be the solution to 
\begin{align*}
\begin{split}
\calP_h \tilP=g_\Ext
\end{split}
\end{align*}
on $[\uptau_1,\uptau_1+\uptau_\ast]$, with data at $\uptau=\uptau_1$ induced by $P(\uptau_1)$. Next we construct $\tiltilP_2$ as the solution with zero data at $\uptau=\uptau_1$ to
\begin{align*}
\begin{split}
\calP_h \tiltilP=\txtg_1+\partial_\uptau\txtg_2,
\end{split}
\end{align*} 
where again $\txtg_1$, $\txtg_2$ are defined by \eqref{eq:txtgtemp4} and \eqref{eq:txtgtemp3} with $P$ replaced by $\tilP+\tiltilP$. Arguing as above, $\tiltilP$ can be defined on $[\uptau_1,\uptau_1+\uptau_\ast]$ provided $\uptau_\ast$. Note that as mentioned above the smallness of $\uptau_\ast$ is independent of the size of $\tilP$, and hence of the size of $P(\uptau_1)$, so the size of the new existence interval for $\tiltilP$ is always the fixed small number $\uptau_\ast$. It follows that $P=\tilP+\tiltilP$ satisfies \eqref{eq:Pdef1} as desired.
\end{proof}
\begin{definition}\label{def:refprof1}
Given parameters $\ell$ and $\xi$ satisfying the conditions in Proposition~\ref{prop:propdefP1} we define $P$ to be the solution to \eqref{eq:Pdef1} from Proposition~\ref{prop:propdefP1}  and let $\varepsilon:=\phi-P$. 
\end{definition}
\begin{remark}
Note that by Proposition~\ref{prop:propdefP1} the conditions
\eqref{eq:giorthdef1}, \eqref{eq:txtgtemp3}, and \eqref{eq:txtgtemp4} are satisfied. In view of the discussion preceding the statement of Proposition~\ref{prop:propdefP1}, this implies that $P$ satisfies the orthogonality condition~\eqref{eq:Porth1}.
\end{remark}
Returning to \eqref{eq:txtgtemp3}--\eqref{eq:giorthdef1}, we derive an expression for $\partial_\uptau c_{2,j}$, $j=1,2,3$,  which reveals their smallness more clearly. In what follows we write 
\begin{align*}
\begin{split}
X_j=\chi\varphi_j^\stat,\quad j=\mu,1,2,3,
\end{split}
\end{align*}
in \eqref{eq:giorthdef1} and use the convention that repeated indices are summed over $\{\mu,1,2,3\}$. We also write $\angles{f}{g}=\int_{\Sigma_\uptau}fg\sqrt{|h|}\ud x$. Then $\angles{g_2}{Y_i}=\angles{X_j}{Y_i} c_{2,j}$. Therefore, with the notation $A_{ji}= \angles{X_j}{Y_i}$ and $A^{-1}_{ji}$ denoting the components of the inverse matrix,
\begin{align*}
\begin{split}
\partial_\uptau c_{2,j}=A^{-1}_{ij}\partial_\uptau\angles{g_2}{Y_i}+(\partial_\uptau A^{-1}_{ij})\angles{g_2}{Y_i}.
\end{split}
\end{align*}
Plugging in \eqref{eq:txtgtemp3} for $\angles{g_2}{Y_i}$ we get
\begin{align}\label{eq:c2prime1}
\begin{split}
\partial_\uptau c_{2,j}=(\partial_\uptau A^{-1}_{ij})\angles{g_2}{Y_i}+A^{-1}_{ij}S_P(\mrI_i+\secondff_i),
\end{split}
\end{align}
where
\begin{align}\label{eq:Iic1}
\begin{split}
\mrI_i&= \partial_\uptau\int_{\Sigma_\uptau}\big(\frac{1}{(h^{-1})^{00}}+1\big)W_i(h^{-1})^{0\nu}\partial_\nu P \sqrt{|h|}\ud x+\partial_\uptau\int_{\Sigma_\uptau}PW_i\partial_\nu\big(\frac{(h^{-1})^{0\nu}}{(h^{-1})^{00}}\sqrt{|h|}\big)\ud x\\
&\quad+\partial_\uptau\int_{\Sigma_\uptau}P(h^{-1})^{0\nu}\big[\big(\frac{1}{(h^{-1})^{00}}+1\big)\partial_\nu W_i-\partial_\nu(Y_i+W_i)\big]\sqrt{|h|}\ud x\\
&\quad+\int_{\Sigma_\uptau}\partial_\uptau(Y_i-W_i)(h^{-1})^{0\nu}\partial_\nu P \sqrt{|h|}\ud x,
\end{split}
\end{align}
and
\begin{align*}
\begin{split}
\secondff_i = \int_{\Sigma_\uptau}(Y_i-W_i)\partial_\uptau((h^{-1})^{0\nu}\sqrt{|h|}\partial_\nu P)\ud x.
\end{split}
\end{align*}
This already has the right form for proving improved pointwise bounds on $\varepsilon$ in Section~\ref{sec:tails}. But, for the ILED estimate for $P$ itself, we need to write $\secondff_i$ in a different form. Recalling equation \eqref{eq:Pdef1} for $P$, we have
\begin{align*}
\begin{split}
\secondff_i&= \int_{\Sigma_\uptau}(h^{-1})^{k\nu}\partial_k(Y_i-W_i)\partial_\nu P\sqrt{|h|}\ud x+\int_{\Sigma_\uptau}(Y_i-W_i)(g_\ext+\txtg_1)\sqrt{|h|}\ud x\\
&\quad +\int_{\Sigma_\uptau}(Y_i-W_i)c_{2,k}\partial_\uptau X_k\sqrt{|h|}\ud x+\int_{\Sigma_\uptau}(Y_i-W_i)X_k\sqrt{|h|}\ud x \partial_\uptau c_{2,k}\\
&=:\widetilde{\secondff}_i+\angles{Y_i-W_i}{X_k}\partial_{\uptau}c_{2,k},
\end{split}
\end{align*}
where
\begin{align}\label{eq:tilIIic1}
\begin{split}
\widetilde{\secondff}_i&= \int_{\Sigma_\uptau}(h^{-1})^{k\nu}\partial_k(Y_i-W_i)\partial_\nu P\sqrt{|h|}\ud x+\int_{\Sigma_\uptau}(Y_i-W_i)(g_\ext+\txtg_1)\sqrt{|h|}\ud x\\
&\quad +\int_{\Sigma_\uptau}(Y_i-W_i)c_{2,k}\partial_\uptau X_k\sqrt{|h|}\ud x.
\end{split}
\end{align}
Plugging back into \eqref{eq:c2prime1} we arrive at the following equation for $\partial_\uptau c_{2,j}$,
\begin{align}\label{eq:c2prime2}
\begin{split}
\partial_\uptau c_{2,i}- A_{ij}^{-1}S_P\angles{Y_i-W_i}{X_k}\partial_\uptau c_{2,k}= A_{ij}^{-1}S_P(\mrI_i+\widetilde{\secondff}_i)+(\partial_\uptau A^{-1}_{ij})\angles{g_2}{Y_i}.
\end{split}
\end{align}

We end by introducing two related versions of the profile $P$ in the first order formulation from Section~\ref{subsubsec:firstorder} and in the exterior renormalized graph formulation from  Section~\ref{subsubsec:2ndordereq}. In the first order formulation we will use the following definition for $\vecP$:
\begin{align}\label{eq:vecPdef1}
\begin{split}
\vecP_\wp\equiv\vecP= \pmat{P\\\dotP}:=\pmat{P\\ \sqrt{|h|}(h^{-1})^{0\nu}\partial_\nu P}.
\end{split}
\end{align}

In the renormalized graph formulation from Section~\ref{subsubsec:2ndordereq}, using the notation from \eqref{eq:sgraphdef1}, we define $P_g=sP$. It follows from \eqref{eq:calPgcalP1} that
\begin{align}\label{eq:upvarepsilondef1}
\begin{split}
\boldsymbol\upvarepsilon:=\fy-P_g=s\varepsilon
\end{split}
\end{align}
satisfies
\begin{align*}
\begin{split}
\calP_g\boldsymbol\upvarepsilon=s\calP\varepsilon,
\end{split}
\end{align*}
where $\calP\varepsilon$ is given in \eqref{eq:varepeq1}.

\section{Main decomposition and smoothing of modulation parameters}\label{sec:finaldecomp}

In this section we use the implicit function theorem to define the parameters $\wp$ and $a_\pm$. As a consequence of the definition, we derive ODEs describing their evolution. As in \cite{LuOS1}, the construction is designed so that the parameter are smooth (more precisely, so that their higher derivatives can be controlled in terms of only finitely many derivatives of the decaying component of the perturbation). The main difference with \cite{LuOS1} is that, since the profile correction $\vecP_\wp$ (see \eqref{eq:vecPdef1}) depends on $\dotwp$, we have to be careful that the implicit function theorem can be set up properly without loss of regularity. 
\subsection{Determination of $\wp$}\label{subsec:wpeq} We start with the determination of $\wp$ for which we use equation \eqref{eq:bfOmega2}. Let $\vectilupphi$ be defined as
\begin{align*}
\begin{split}
\vecpsi=\vecP_\wp+\vectilupphi,
\end{split}
\end{align*}
and according let
\begin{align*}
\begin{split}
\vecbfOmega= \vecbfOmega_P+\vecbfOmega_\tilupphi,\qquad \vecN=\vecN_P+\vecN_\tilupphi.
\end{split}
\end{align*}
Let $k\in C^\infty_c(\bbR)$ be a smooth bump function supported in $[0,1]$, and let $S$ be the smoothing operator\footnote{We could also use the same definition as in \eqref{eq:Sdef1} and \eqref{eq:tilSdef1} but since the moment condition $\int_0^1k(s)\ud s=0$ is no longer needed, we have distinguished the choices of the smoothing operators.} (for $h(t)$ defined for $t\geq -1$)
\begin{align*}
\begin{split}
(Sh)(t)=\int_{\bbR}\chi_{[-1,\infty)}(s)h(s)k(t-s)\ud s,\qquad t\geq -1.
\end{split}
\end{align*}
It follows that $S-I=\frac{\ud}{\ud t}\tilS$ with 
\begin{align*}
\begin{split}
(\tilS h)(t)=\int_\bbR\chi_{[-1,\infty)}(s)h(s)\tilk(t-s)\ud s,\qquad t\geq-1,
\end{split}
\end{align*}
where $\tilk(r)=0$ for $r<0$ and $\tilk(r):=-\int_r^\infty k(s)\ud s$ for $r\geq0$.
To motivate the final orthogonality condition, we want to choose $\vecbfOmega_\tilupphi$ so that $\wp$ satisfies\footnote{To understand the process better, consider the simplified situation when the first order formulation leads to an equation of the form $(\partial_t-M)\vecpsi=\vecF_0+\vecF_1$, where $\vecF_0$ contains the contribution of $\dotwp$ and $\vecF_1$ is the nonlinearity. Assuming time independence of $\vecZ$ and using the notation from the text, this leads to $\vecbfOmega(\vecF_0,\vecZ)=\partial_t\vecbfOmega_\tilupphi+\partial_t\vecbfOmega_P+\vecN-\vecbfOmega(\vecF_1,\vecZ)$. In view of this, the required orthogonality conditions are $\vecbfOmega_\tilupphi=\tilS(\vecN-\vecbfOmega(\vecF_1,\vecZ))-(I-S)\vecbfOmega_P$, or equivalently $\vecbfOmega(\vecPhi-\vecPsi_\wp,\vecZ)=\tilS(\vecN-\vecbfOmega(\vecF_1,\vecZ))+S\vecbfOmega_P$, which imply $\vecbfOmega(\vecF_0,\vecZ)=S(\vecbfOmega((\partial_t-M)\vecP,\vecZ)+\vecN_\tilupphi-\vecbfOmega(\vecF_1,\vecZ))$ (here we have in mind that $\vecbfOmega(\vecF_0,\vecZ)\approx \dotwp$).}
\begin{align}\label{eq:parODE1}
\begin{split}
\dotwp=\vecG(S\partial_\Sigma^{\leq 2}\vecpsi,\ell,S(\vecN_\tilupphi+\vecN_P+\partial_t\vecbfOmega_P)-\beta \vecomega)
\end{split}
\end{align}
for some constant $\beta>0$ and with $\vecomega$ to be determined below (see the discussion leading to equation (3.34) in \cite{LuOS1} for the motivation for introducing $\vecomega$). Comparing with \eqref{eq:bfOmega2}, and defining
\begin{align*}
\begin{split}
\vecF_\omega(\partial_\Sigma^{\leq2}\vecpsi,\ell,S(\vecN+\partial_t\vecbfOmega_P)-\beta\vecomega)&=\vecF(0,\ell,\vecG(0,\ell,S(\vecN+\partial_t\vecbfOmega_P)-\beta\vecomega))\\
&\quad-\vecF(\partial_\Sigma^{\leq2}\vecpsi,\ell,\vecG(S\partial_\Sigma^{\leq2}\vecpsi,\ell,S(\vecN+\partial_t\vecbfOmega_P)-\beta\vecomega)),
\end{split}
\end{align*}
so that 
\begin{align}\label{eq:Fomegaquad1}
\begin{split}
|\vecF_\omega|\lesssim |\partial_\Sigma^{\leq2}\vecpsi|(|S\partial_\Sigma^{\leq2}\vecpsi|+|S(\vecN+\partial_t\vecbfOmega_P)-\beta\vecomega|),
\end{split}
\end{align}
we get
\begin{align*}
\begin{split}
\partial_t\vecbfOmega_\tilupphi=(S-I)\vecN+\partial_t(S-I)\vecbfOmega_P-\beta\vecomega-\vecF_\omega.
\end{split}
\end{align*}
To achieve this orthogonality we impose a further decomposition
\begin{align}\label{eq:orth1}
\begin{split}
\vecbfOmega_\tilupphi=\vecUpomega+\vecomega,
\end{split}
\end{align}
such that
\begin{align}\label{eq:vecomegaODE1}
\begin{split}
\partial_t\vecUpomega=(S-I)(\vecN+\vecF_\omega)+\partial_t(S-I)\vecbfOmega_P,\qquad \partial_t\vecomega+\beta\vecomega=-S\vecF_\omega.
\end{split}
\end{align}
With this preparation we are ready to set up the implicit function theorem to define $\wp$. Let $\Phi$ be a regular (say $C^5$) solution of the HVMC, and consider $C^2$ curves $\xi(t)$ and $\ell(t)$ defined on some time interval $J=[0,\tau_0]$ in the domain of definition of $\Phi$. Let $\Psi_\wp$ denote the profile in \eqref{eq:Psiwpdef1}, and define $\psi_\wp:=\bfeta(\Phi-\Psi_\wp,n_\wp)$, where we use the notation introduced in Section~\ref{subsec:mainprofile}. Similarly, we define $\dotpsi_\wp$ as in \eqref{eq:psidotdef1} and let $\vecpsi_\wp=(\psi_\wp,\dotpsi_\wp)$. We extend $\xi$, $\ell$, and $\partial_\Sigma^{\leq2}\vecpsi_\wp$ to $-1\leq t\leq0$ and define $\vecomega$ to be the unique solution to
\begin{align}\label{eq:vecomegaODE2}
\begin{cases}
\vecomega(t)=-\int_0^te^{-\beta(t-s)}S\vecF_\omega(\partial_\Sigma^{\leq 2}\vecpsi_\wp,\ell,S\vecN-\vecomega)\ud s,\qquad &t\geq 0\\
\vecomega(t)=0,\qquad&t<0
\end{cases}.
\end{align}
The existence of a unique solution $\vecomega$ follows from the quadratic estimate \eqref{eq:Fomegaquad1} on $\vecF_\omega$ and a fixed point argument. Let
\begin{align}\label{eq:vecUpomegadef1}
\begin{split}
\vecUpomega=\tilS(\vecN+\vecF_\omega)+(S-I)\vecbfOmega_P,
\end{split}
\end{align}
and define
\begin{align}\label{eq:Upsilondef1}
\begin{split}
\vec\Upsilon(\Phi,\xi,\ell):=\vecbfOmega(\vecpsi_\wp)-\vecbfOmega_P-\vecUpomega-\vecomega=\vecbfOmega(\vecpsi_\wp)-\tilS(\vecN+\vecF_\omega)-S\vecbfOmega_P-\vecomega.
\end{split}
\end{align}
Observe that for $\Psi_0(t,\rho,\omega)=(t,F(\rho,\omega))$ we have $\vec\Upsilon(\Psi_0,0,0)=0$. We want to view $\vec\Upsilon$ as a map into $C^2(J)$ and verify that the Fr\'echet derivative $D_{(\xi,\ell)}\vec\Upsilon(\Psi_0,0,0)$ is invertible. We could then conclude from the implicit function theorem that given a solution $\Phi$, there are $C^2$ curves $\xi(t)$ and $\ell(t)$ for which $\vec\Upsilon(\Phi,\xi,\ell)=0$. Tracing back the definitions, we get that the parameters satisfy the ODEs \eqref{eq:parODE1}. In order to execute this plan, we should first check that $\vec\Upsilon$ indeed defines a $C^2$ curve. To see this, observe that by the smoothing property of $S$, the last two terms on the right-hand side of \eqref{eq:Upsilondef1} are in fact smooth. The first term $\vecbfOmega(\vecpsi_\wp)$ is $C^2$ because we have assumed that $\xi$ and $\ell$ are $C^2$ functions. So it remains to consider $\tilS(\vecN+\vecF_\omega)$. Note that $\vecP$ is defined in terms of one derivative of the parameters. Since $\tilS$ is smoothing of order one (recall that $\frac{\ud}{\ud t}\tilS=S-I$), it follows that $\frac{\ud^2}{\ud t^2}\tilS\vecN$ depends on two derivatives of $\wp$, and since $\wp$ was assumed to be $C^2$, this concludes the proof that $\vec\Upsilon$ defines a $C^2$ curve. The invertibility of the Fr\'echet derivative $D_{(\xi,\ell)}\vec\Upsilon(\Psi_0,0,0)$ follows from the same argument as in \cite[Section~3.6]{LuOS1}, so we will not reproduce the proof. The only point that needs extra attention is the smallness of the contribution from the terms involving $\vecP$. But this smallness follows from the smallness of the source term, which is in turned guaranteed by choosing $\uprho_0$ and $R_\ctf$ (see \eqref{eq:gextdef1} and \eqref{eq:WYdef1}) sufficiently large.
\subsection{Determination of $a_\pm$}\label{subsec:apm} We now assume that the parameters $\xi$ and $\ell$ are already determined and continue to fix $a_\pm$. The choice of parameters $\xi$ and $\ell$, gives a decomposition of the local solution $\Phi$ to the HVMC equation in first order formulation as $\vecPhi=\vecPsi_\wp+\vecpsi$. We will impose orthogonality conditions for the additional decomposition $\vecpsi=\vecphi+a_{+}\vecZ_{+}+a_{-}\vecZ_{-}$ (see \eqref{eq:psiphi1}) so as to obtain smoothing in the equations for $a_\pm$ in \eqref{eq:apm1}. More precisely, we have to incorporate the contribution of the refined profile $\vecP_\wp$ as well, for which we write
\begin{align}\label{eq:vecvarepdef1}
\begin{split}
\vecphi=\vecP_\wp+\vecvarepsilon.
\end{split}
\end{align}
Note that comparing with our earlier definition of $\vectilupphi$ we have
\begin{align}\label{eq:vecvarepdef2}
\begin{split}
\vecvarepsilon=\vectilupphi-a_{+}\vecZ_{+}-a_{-}\vecZ_{-}.
\end{split}
\end{align}
Plugging into \eqref{eq:apm1} we arrive at
\begin{align}\label{eq:apm2}
\begin{split}
\frac{\ud}{\ud t}(e^{\mp\mu t}a_{\pm})=\mp\frac{\ud}{\ud t}\big(e^{\mp\mu t}\bfOmega(\vecvarepsilon,\vecZ_{\mp})+e^{\mp\mu t}\bfOmega(\vecP_\wp,\vecZ_{\mp})\big)\mp e^{\mp\mu t}F_{\pm}.
\end{split}
\end{align}
Parameter smoothing for these equations can be achieved by requiring the orthogonality conditions
\begin{align}\label{eq:orthpm1}
\begin{split}
&\bfOmega(\vecvarepsilon,\vecZ_{-})=e^{\mu t}\tilS(e^{-\mu t}F_{+})-e^{\mu t}(I-S)(e^{-\mu t}\bfOmega(\vecP_\wp,\vecZ_{-})),\\
&\bfOmega(\vecvarepsilon,\vecZ_{+})=e^{-\mu t}\tilS(e^{\mu t}F_{-})-e^{-\mu t}(I-S)(e^{\mu t}\bfOmega(\vecP_\wp,\vecZ_{+})),
\end{split}
\end{align}
which, in view of \eqref{eq:apm1}, lead to
\begin{align}\label{eq:apmfinal1}
\begin{split}
&\frac{\ud}{\ud t}(e^{-\mu t}a_{+})=-S(e^{-\mu t}F_{+}+\partial_t(e^{-\mu t}\bfOmega(\vecP_\wp,\vecZ_{-})))=:-S(e^{-\mu t}\tilF_{+}),\\ 
&\frac{\ud}{\ud t}(e^{\mu t}a_{-})=S(e^{\mu t}F_{-}+\partial_t(e^{\mu t}\bfOmega(\vecP_\wp,\vecZ_{+})))=:S(e^{\mu t}\tilF_{-}).
\end{split}
\end{align}
Note that for the purpose of the implicit function theorem used to define the parameters, we replace $\vecphi$ appearing in the definition of $F_{\pm}$ in \eqref{eq:F+1} and \eqref{eq:F-1} by $\vecphi=\vecP_\wp+\vecvarepsilon$. On the other hand, to estimate $a_{\pm}$ we will rearrange terms to be able to use the equation for $\vecP_\wp$. More precisely, we can write
\begin{align}\label{eq:F+2}
\begin{split}
\tilF_{+}&=F_{+}+e^{\mu t}\partial_t(e^{-\mu t}\bfOmega(\vecP_\wp,\vecZ_{-})) \\
&= -\bfOmega(\vecF_1,\vecZ_{-})+\bfOmega(\vecvarepsilon,M\vecZ_{-}+\mu\vecZ_{-})-a_{+}\bfOmega(M\vecZ_{+}-\mu\vecZ_{+},\vecZ_{+})\\
&\quad-a_{-}\bfOmega(M\vecZ_{-}+\mu\vecZ_{-},\vecZ_{-})+\bfOmega((\partial_t-M)\vecP_\wp,\vecZ_{-})
\end{split}
\end{align}
and
\begin{align}\label{eq:F-2}
\begin{split}
\tilF_{-}=F_{-}+e^{-\mu t}\partial_t(e^{\mu t}\bfOmega(\vecP_\wp,\vecZ_{+}))&=-\bfOmega(\vecF_1,\vecZ_{+})+\bfOmega(\vecvarepsilon,M\vecZ_{+}-\mu\vecZ_{+})-a_{+}\bfOmega(M\vecZ_{+}-\mu\vecZ_{+},\vecZ_{+})\\
&\quad-a_{-}\bfOmega(M\vecZ_{-}+\mu\vecZ_{-},\vecZ_{+})+\bfOmega((\partial_t-M)\vecP_\wp,\vecZ_{+}).
\end{split}
\end{align}
To arrive at the orthogonality conditions \eqref{eq:orthpm1} we argue as follows. First, given $\vecpsi(t)$ and $C^2$ curves $\xi$, $\ell$, $a_\pm(t)$ defined on some time interval $J=[0,\tau_0]$ we extend them to $-1\leq t\leq0$. We then define
\begin{align*}
\begin{split}
\Upsilon_\mu(\vecpsi,a_{+},a_{-},\xi,\ell)=\pmat{\bfOmega(\vecpsi-a_{+}\vecZ_{+}-a_{-}\vecZ_{-},\vecZ_{-})-e^{\mu t}\tilS(e^{-\mu}F_{+})-e^{\mu t} S (e^{-\mu t}\bfOmega(\vecP_\wp,\vecZ_{-}))\\\bfOmega(\vecpsi-a_{+}\vecZ_{+}-a_{-}\vecZ_{-},\vecZ_{+})+e^{-\mu t}\tilS(e^{\mu t}F_{-})-e^{-\mu t}S(e^{\mu t}\bfOmega(\vecP_\wp,\vecZ_{+}))}.
\end{split}
\end{align*}
Observe that $\Upsilon_\mu(0,\dots,0)=0$. Therefore, to be able to use the implicit function theorem to achieve the orthogonality conditions \eqref{eq:orthpm1} it suffices to show that the Fr\'echet derivative $D_{(a_{+},a_{-})}\Upsilon_\mu(0,\dots,0)$ is invertible.  Note that in view of the presence of $\vecP_\wp$ in the definition of $\Upsilon_\mu$, one has to check that $\Upsilon_\mu$ is map to $C^2$, but this holds by the same considerations as for $\Upsilon$ in the definition of the parameters $\wp$ above. Invertibility of $D_{(a_{+},a_{-})}\Upsilon_\mu(0,\dots,0)$ holds by the same exact argument as in \cite[Section~3.7]{LuOS1} and the same considerations as for the parameters $\wp$ above for the contribution of $\vecP_\wp$.
 
\section{The bootstrap argument and the proof of the main theorem}\label{sec:bootstrap}
In this section we set up the overall bootstrap argument. The main results are stated in Propositions~\ref{prop:bootstrappar1},~\ref{prop:bootstrapphi1}, and~\ref{prop:boostrapvareptail}, whose proofs will occupy most of the remainder of the paper. In the final part of this section we will prove Theorem~\ref{thm:main} assuming the validity of Propositions~\ref{prop:bootstrappar1},~\ref{prop:bootstrapphi1}, and~\ref{prop:boostrapvareptail}.

We assume the existence of parameters $\xi,\ell, a_{\pm}$, defined on a $\uptau$ interval $[0,\tau_f)$, and a parameterization \eqref{eq:psidotdef1} on this interval. With the refined profile $P$ given by Definition~\ref{def:refprof1} we assume that the orthogonality conditions \eqref{eq:orth1} and \eqref{eq:orthpm1} are satisfied. Our bootstrap assumptions consist of a \emph{trapping assumption} for $a_{+}$ and a series of decay estimates $a_{\pm}$, $\dotwp$, $\varepsilon$, and $P$. The trapping assumption is for the second derivative of $a_{+}$ as we do not need improvements in decay for higher derivatives. But to to apply a topological selection argument for the growing mode, the estimate needs to be stated at the level of $a_{+}$ itself. For this reason, we have used the equation to express $\ddot{a}_{+}$ in terms of $a_{+}$ in the trapping estimate which can be stated as (see \eqref{eq:F+1} and \eqref{eq:F+2} for the expression of $\tilF_{+}$ and \eqref{eq:psiphi1}, \eqref{eq:vecvarepdef1}, \eqref{eq:vecvarepdef2} for the relation between $\vecvarepsilon$ and $\vecphi$)
\begin{equation} \label{eq:a+trap}
|\mu(\mu a_{+}(\tau)-e^{\mu \tau}S(e^{-\mu \tau}\tilF_{+}))-e^{\mu \tau}S(e^{-\mu \tau}\dottilF_{+})|\leq C_\trap\delta_\wp \epsilon \jap{\uptau}^{-3}.
\end{equation}
Here, it should be understood that every appearance of $\dota_{\pm}$ in $\dottilF_{+}$ on the right-hand side should be replaced by 
\begin{align*}
\begin{split}
\dota_{\pm}\quad\to\quad \pm\mu a_{\pm}\mp e^{\pm\mu t}S(e^{-\mu t}\tilF_{\pm}).
\end{split}
\end{align*}
Our remaining bootstrap assumptions are that the following estimates hold for all $\tau,\sigma_1,\sigma_2\in[0,\tau_f)$ where we use $\delta_\wp$ to denote a constant that  is small in terms of $\wp$ or inverse powers of $R_\ctf$ (recall from Section~\ref{subsec:prelimvfs} that $\tilpartial_\Sigma$ denotes size one tangential derivatives $\partial_\Sigma$ or $\jap{\uprho}^{-1} \RbfT$ and that in the exterior $X$ denotes any of the vectorfields $\tilr L$, $\Omega$, or $T$ from Appendix~\ref{subsec:appMink} equation~\eqref{eq:VFdef0}):%
\begin{align}
&|a_{+}|\leq 2 C(R_\ctf^2 R^2) \delta_{\wp}\epsilon \tau^{-\frac{9}{4}+\kappa}.\label{eq:a+b1}\\
&|a_{-}-a_{-}(0)e^{-\mu\tau}|\leq 2C (R_\ctf^2 R^2) \delta_{\wp} \epsilon\tau^{-\frac{9}{4}+\kappa}.\label{eq:a-b1}\\
&|\dotwp|\leq 2C (R_\ctf^2 R^2) \delta_{\wp}\epsilon \tau^{-\frac{9}{4}+\kappa}.\label{eq:wpb1}\\
&|a_{+}^{(k)}|\leq 2 C_{k} \delta_{\wp}\epsilon \tau^{-\frac{9}{4}+\frac{\kappa}{2}}, \quad \forall  k\geq 1.\label{eq:a+b2}\\
&|a_{-}^{(k)}-(-\mu)^ka_{-}(0)e^{-\mu \tau}|\leq 2C_k \delta_{\wp} \epsilon\tau^{-\frac{9}{4}+\frac{\kappa}{2}},\quad \forall  k\geq 1.\label{eq:a-b2}\\
&|\dotwp^{(k)}|\leq 2C_k \delta_{\wp}\epsilon \tau^{-\frac{9}{4}+\frac{\kappa}{2}},\quad \forall  k\geq 2.\label{eq:wpb2}\\
&\|\chi_{\leq R}\partial_\Sigma \partial^{k}\RbfT^j\varepsilon\|_{L^2(\Sigma_\tau)}+\|\chi_{\geq R} \tilpartial_\Sigma X^{k}\RbfT^j\varepsilon\|_{L^2(\Sigma_\tau)}\leq 2C_{j,k} \epsilon\tau^{-1-j},\quad 0\leq j\leq \NT,\nonumber\\
&\phantom{\|\chi_{\leq R}\partial_\Sigma \partial^{k}\RbfT^j\varepsilon\|_{L^2(\Sigma_\tau)}+\|\chi_{\geq R} \tilpartial_\Sigma X^{k}\RbfT^j\varepsilon\|_{L^2(\Sigma_\tau)}\leq 2C_{j,k} \epsilon\tau^{-1-j},\quad}k+3j \leq M_\varepsilon-2.\label{eq:Tjphienergyb1}\\
&\|\chi_{\leq R}\partial_\Sigma \partial^{k}\RbfT^jP\|_{L^2(\Sigma_\tau)}+\|\chi_{\geq R} \tilpartial_\Sigma X^{k}\RbfT^jP\|_{L^2(\Sigma_\tau)}\leq 2C_{j,k} \epsilon\tau^{-1-j},\quad 0\leq j\leq \NT,\nonumber\\ 
&\phantom{\|\chi_{\leq R}\partial_\Sigma \partial^{k}\RbfT^jP\|_{L^2(\Sigma_\tau)}+\|\chi_{\geq R} \tilpartial_\Sigma X^{k}\RbfT^jP\|_{L^2(\Sigma_\tau)}\leq 2C_{j,k} \epsilon\tau^{-1-j},\quad}k+3j \leq M_P-2.\label{eq:TjPenergyb1}\\
&\|\chi_{\geq R}r^{p} (\partial_r+\frac{1}{r}) X^{k}\RbfT^j\varepsilon\|_{L^2(\Sigma_\tau)}\leq 2C_{j,k} \epsilon\tau^{-1-j+\frac{p}{2}},\quad 0\leq p\leq 2,~0\leq j\leq \NT,~ k+3j \leq M_\varepsilon-2.\label{eq:Tjphienergyb2}\\
&\|\chi_{\geq R}r^{p} (\partial_r+\frac{1}{r}) X^{k}\RbfT^j P\|_{L^2(\Sigma_\tau)}\leq 2C_{j,k} \epsilon\tau^{-1-j+\frac{p}{2}},\quad 0\leq p\leq 2,~0\leq j\leq \NT,~ k+3j \leq M_P-2.\label{eq:TjPenergyb2}\\
&\| \chi_{\leq R}\partial^k\RbfT^j\varepsilon\|_{L^2(\Sigma_\tau)}+\|\jap{r}^{-\frac{3}{2}+\frac{\kappa}{2}}(\chi_{\geq R}X^k\RbfT^j\varepsilon)\|_{L^2(\Sigma_\tau)}\leq 2C_{j,k} \epsilon\tau^{-\frac{3}{2}-j+\frac{\kappa}{2}},\quad 0\leq j \leq 1,\nonumber\\
&\phantom{\| \chi_{\leq R}\partial^k\RbfT^j\varepsilon\|_{L^2(\Sigma_\tau)}+\|\jap{r}^{-\frac{3}{2}+\frac{\kappa}{2}}(\chi_{\geq R}X^k\RbfT^j\varepsilon)\|_{L^2(\Sigma_\tau)}\leq 2C_{j,k} \epsilon\tau^{-\frac{3}{2}-j+\frac{\kappa}{2}},\quad}0\leq k+3j\leq M_\varepsilon-5.\label{eq:dphiL2b1}\\
&\| \chi_{\leq R}\partial^k\RbfT^jP\|_{L^2(\Sigma_\tau)}+\|\jap{r}^{-\frac{3}{2}+\frac{\kappa}{2}}(\chi_{\geq R}X^k\RbfT^jP)\|_{L^2(\Sigma_\tau)}\leq 2C_{j,k} \epsilon\tau^{-\frac{3}{2}-j+\frac{\kappa}{2}},\nonumber\\
&\phantom{\| \chi_{\leq R}\partial^k\RbfT^jP\|_{L^2(\Sigma_\tau)}+\|\jap{r}^{-\frac{3}{2}+\frac{\kappa}{2}}(\chi_{\geq R}X^k\RbfT^jP)\|_{L^2(\Sigma_\tau)}}0\leq j \leq 1,~0\leq k+3j\leq M_P-5.\label{eq:dPL2b1}\\
&|\RbfT^j\varepsilon|+\chi_{\geq R}|X^k\RbfT^j\varepsilon|\leq 2C_{j,k}\epsilon\tau^{-\frac{5}{4}-j+\frac{\kappa}{2}}, \quad0\leq j \leq 1,~ 0\leq k+3j\leq M_\varepsilon-7,\label{eq:varepptwiserpb1}\\
&|\RbfT^jP|+\chi_{\geq R}|X^k\RbfT^jP|\leq 2C_{j,k}\epsilon\tau^{-\frac{5}{4}-j+\frac{\kappa}{2}}, \quad 0\leq j \leq 1,~0\leq k+3j\leq M_P-7.\label{eq:Pptwiserpb1}\\
&|\partial^k\RbfT^j\varepsilon|+\chi_{\geq R}|\partial^{k-m}X^m\RbfT^j\varepsilon|\leq 2C_{j,k}\epsilon\tau^{-\frac{3}{2}-j+\frac{\kappa}{2}}, \quad0\leq j \leq 1,~1\leq k\leq M_\varepsilon-8,\nonumber\\
&\phantom{|\partial^k\RbfT^j\varepsilon|+\chi_{\geq R}|\partial^{k-m}X^m\RbfT^j\varepsilon|\leq 2C_{j,k}\epsilon\tau^{-\frac{5}{4}-j+\frac{\kappa}{2}},\quad}m<k,~0\leq k+3j\leq M_\varepsilon-8.\label{eq:varepptwiserpb2}\\
&|\partial^k\RbfT^j P|+\chi_{\geq R}|\partial^{k-m}X^m\RbfT^j P|\leq 2C_{j,k}\epsilon\tau^{-\frac{3}{2}-j+\frac{\kappa}{2}}, \quad0\leq j \leq 1,~1\leq k\leq M_P-8,\nonumber\\
&\phantom{|\partial^k\RbfT^j\varepsilon|+\chi_{\geq R}|\partial^{k-m}X^m\RbfT^j\varepsilon|\leq 2C_{j,k}\epsilon\tau^{-\frac{5}{4}-j+\frac{\kappa}{2}},\quad}m<k,~0\leq k+3j\leq M_P-8.\label{eq:Pptwiserpb2}\\
&\chi_{\geq R}|\partial^{k-m}X^m\RbfT^j\varepsilon|\leq 2C_{j,k}\epsilon\jap{r}^{-\frac{1}{2}}\tau^{-1-j+\frac{\kappa}{2}}, \quad0\leq j \leq 1,~1\leq k\leq M_\varepsilon-8,\nonumber\\
&\phantom{|\partial^k\RbfT^j\varepsilon|+\chi_{\geq R}|\partial^{k-m}X^m\RbfT^j\varepsilon|\leq 2C_{j,k}\epsilon\jap{r}^{-\frac{1}{2}}\tau^{-1-j+\frac{\kappa}{2}},\quad}m<k,~0\leq k+3j\leq M_\varepsilon-8.\label{eq:varepptwiserpb3}\\
&\chi_{\geq R}|\partial^{k-m}X^m\RbfT^j P|\leq 2C_{j,k}\epsilon\jap{r}^{-\frac{1}{2}}\tau^{-1-j+\frac{\kappa}{2}}, \quad0\leq j \leq 1,~1\leq k\leq M_P-8,\nonumber\\
&\phantom{|\partial^k\RbfT^j\varepsilon|+\chi_{\geq R}|\partial^{k-m}X^m\RbfT^j\varepsilon|\leq 2C_{j,k}\epsilon\jap{r}^{-\frac{1}{2}}\tau^{-1-j+\frac{\kappa}{2}},\quad}m<k,~0\leq k+3j\leq M_P-8.\label{eq:Pptwiserpb3}\\
&\chi_{\geq R}|\partial^{k-m}X^m\RbfT^j\varepsilon|\leq 2C_{j,k}\epsilon\jap{r}^{-1}\tau^{-\frac{1}{2}-j+\frac{\kappa}{2}}, \quad0\leq j \leq 1,~1\leq k\leq M_\varepsilon-8,\nonumber\\
&\phantom{|\partial^k\RbfT^j\varepsilon|+\chi_{\geq R}|\partial^{k-m}X^m\RbfT^j\varepsilon|\leq 2C_{j,k}\epsilon\jap{r}^{-1}\tau^{-\frac{1}{2}-j+\frac{\kappa}{2}},\quad}m<k,~0\leq k+3j\leq M_\varepsilon-8.\label{eq:varepptwiserpb4}\\
&\chi_{\geq R}|\partial^{k-m}X^m\RbfT^j P|\leq 2C_{j,k}\epsilon\jap{r}^{-1}\tau^{-\frac{1}{2}-j+\frac{\kappa}{2}}, \quad0\leq j \leq 1,~1\leq k\leq M_P-8,\nonumber\\
&\phantom{|\partial^k\RbfT^j\varepsilon|+\chi_{\geq R}|\partial^{k-m}X^m\RbfT^j\varepsilon|\leq 2C_{j,k}\epsilon\jap{r}^{-1}\tau^{-\frac{1}{2}-j+\frac{\kappa}{2}},\quad}m<k,~0\leq k+3j\leq M_P-8.\label{eq:Pptwiserpb4}\\
&|\partial^k\varepsilon|+\chi_{\geq R}|X^k\varepsilon|\leq 2C_{k}(R_\ctf^2R^2)\epsilon \tau^{-\frac{9}{4}+\kappa},\quad 0\leq k\leq M_\varepsilon-\tilC_\varepsilon.\label{eq:varepptwisetailb1}
\end{align}
Here $\tilC_\varepsilon$ is an absolute constant that corresponds to the vectorfield regularity we lose in obtaining the improved pointwise estimate \eqref{eq:varepptwisetailb1}. We have not attempted to optimize this constant. See the proofs of Proposition~\ref{prop:boostrapvareptail} and Lemma~\ref{eq:Boxmhugens1}.
\begin{remark}\label{rem:varepbotstrap1}
Note that in the pointwise bounds \eqref{eq:varepptwisetailb1} and \eqref{eq:a+b1}--\eqref{eq:wpb1} we have used $\kappa$ instead of $\frac{\kappa}{2}$. The reason is that in deriving the improved bound \eqref{eq:varepptwisetailb1}, which is used in deriving \eqref{eq:a+b1}--\eqref{eq:wpb1}, we will use the improved estimate for $\RbfT\varepsilon$ from \eqref{eq:varepptwiserpb1}, but will lose a factor of $\log\tau$ because of a slow $r$ decay. See the proof of Proposition~\ref{prop:boostrapvareptail} in Section~\ref{sec:tails}. We expect that by a more careful analysis we could avoid this loss, but the current estimate is sufficient for closing the bootstrap.

\end{remark}
\begin{remark}\label{rem:a+bootstrap}
The trapping assumption \eqref{eq:a+trap} is stated at the level of the $\NT^{\mathrm{nd}}$ derivative $a_{+}$, because we prove improved decay for up to two time derivatives of $\varepsilon$. The assumption \eqref{eq:a+trap} is used to close the corresponding bootstrap assumptions. See Lemma~\ref{lem:orthTk} and its proof.
\end{remark}
\begin{remark}\label{rem:a-bootstrap}
The small constant $\delta_\wp$ is important for the following reason. The parameters $\dotwp$, $a_{\pm}$ and $\varepsilon$ depend linearly on each other. So to break the circularity in estimating $\varepsilon$ in terms of $\dotwp$, $a_{\pm}$ and vice versa, we need to use the fact that the linear dependence of the parameters on $\varepsilon$ is because we have introduced a cutoff in imposing orthogonality conditions. Therefore the estimates for the parameters come with the extra smallness $\delta_\wp$ compared with the corresponding estimates for $\varepsilon$. Also note that for $a_{-}$ we have separated the contribution of the initial data which evolves simply as an exponential. This linear evolution is completely independent of $\varepsilon$ and therefore its contribution to the estimates for $\varepsilon$ can be handled by choosing the bootstrap constants in the estimates for $\varepsilon$ sufficiently large.
\end{remark}
\begin{remark}\label{rem:MpMvarep}
The reason we are allowed to commute a higher number, $M_P$, of derivatives with the equation for $P$ compared with $M_\varepsilon$ for the equation for $\varepsilon$ is that $P$ depends on $\varepsilon$ only through the parameters $\wp$ which are more regular than $\varepsilon$ (in view of the smoothing introduced in their definition; see Section~\ref{subsec:wpeq}).This allows us to treat the quasilinear terms in $P$ as lower order when deriving top order energy estimates for $\varepsilon$. Alternatively we could have commuted an equal number of derivatives and worked with $P$ and $\phi$ for the derivation of energy estimates.
\end{remark}
We close our bootstrap assumptions in a few steps. First, in Proposition~\ref{prop:bootstrappar1}, we close the bootstrap assumptions for the parameters, but with a suboptimal rate for $|\ddot{a}|$ in \eqref{eq:a+trap}. We then use this in Proposition~\ref{prop:bootstrapphi1} to improve the bootstrap bounds on $\varepsilon$ and $P$ that come from the $r^p$ energy method, specifically, \eqref{eq:Tjphienergyb1}--\eqref{eq:Pptwiserpb4}.  The improved pointwise estimate  \eqref{eq:varepptwisetailb1} is closed in Proposition~\ref{prop:boostrapvareptail}. Finally, in the proof of Theorem~\ref{thm:main} we show that the initial data and parameters can be chosen such that the trapping assumption \eqref{eq:a+trap} is satisfied.
\begin{proposition}\label{prop:bootstrappar1}
Suppose the estimates \eqref{eq:a+trap}--\eqref{eq:varepptwisetailb1} and orthogonality conditions \eqref{eq:orth1} and \eqref{eq:orthpm1} are satisfied. If $\epsilon$ is sufficiently small and $C, C_{j,k}, C_k$ appearing on the right-hand side of \eqref{eq:a+trap}--\eqref{eq:varepptwisetailb1} are sufficiently large (compared to $C_{trap}$), then the following improved estimates hold:
\begin{align} 
&|a_{+}|\leq C \delta_{\wp}(R_\ctf^2R^2) \epsilon\tau^{-\frac{9}{4}+\kappa}.\label{eq:a+1}\\
&|a_{-}|\leq C \delta_{\wp}(R_\ctf^2R^2) \epsilon\tau^{-\frac{9}{4}+\kappa}.\label{eq:a-1}\\
&|\dotwp|\leq C \delta_{\wp}(R_\ctf^2R^2) \epsilon\tau^{-\kappa}.\label{eq:wp1}\\
&|a_{+}^{(k)}|\leq C_{k} \delta_{\wp} \epsilon\tau^{-\frac{9}{4}+\frac{\kappa}{2}}, \quad \forall  k\geq 1.\label{eq:a+2}\\
&|a_{-}^{(k)}|\leq C_k \delta_{\wp} \epsilon\tau^{-\frac{9}{4}+\frac{\kappa}{2}},\quad \forall  k\geq 1.\label{eq:a-2}\\
&|\dotwp^{(k)}|\leq C_k \delta_{\wp} \epsilon\tau^{-\frac{9}{4}+\frac{\kappa}{2}},\quad \forall  k\geq 2.\label{eq:wp2}
\end{align}
\end{proposition}
\begin{proposition}\label{prop:bootstrapphi1}
Suppose the estimates \eqref{eq:a+trap}--\eqref{eq:varepptwisetailb1} and orthogonality conditions \eqref{eq:orth1} and \eqref{eq:orthpm1} are satisfied. If $\epsilon$ is sufficiently small and $C, C_{j,k}, C_k$ appearing on the right-hand side of \eqref{eq:a+trap}--\eqref{eq:varepptwisetailb1} are sufficiently large (compared to $C_{trap}$), then the following improved estimates hold:
\begin{align}
&\|\chi_{\leq R}\partial_\Sigma \partial^{k}\RbfT^j\varepsilon\|_{L^2(\Sigma_\tau)}+\|\chi_{\geq R} \tilpartial_\Sigma X^{k}\RbfT^j\varepsilon\|_{L^2(\Sigma_\tau)}\leq C_{j,k} \epsilon\tau^{-1-j},\quad 0\leq j\leq \NT,\nonumber\\ &\phantom{\|\chi_{\leq R}\partial_\Sigma \partial^{k}\RbfT^j\varepsilon\|_{L^2(\Sigma_\tau)}+\|\chi_{\geq R} \tilpartial_\Sigma X^{k}\RbfT^j\varepsilon\|_{L^2(\Sigma_\tau)}\leq C_{j,k} \epsilon\tau^{-1-j},\quad}k+3j \leq M_\varepsilon-2.\label{eq:Tjphienergy1}\\
&\|\chi_{\leq R}\partial_\Sigma \partial^{k}\RbfT^jP\|_{L^2(\Sigma_\tau)}+\|\chi_{\geq R} \tilpartial_\Sigma X^{k}\RbfT^jP\|_{L^2(\Sigma_\tau)}\leq C_{j,k} \epsilon\tau^{-1-j},\quad 0\leq j\leq \NT,\nonumber\\
&\phantom{\|\chi_{\leq R}\partial_\Sigma \partial^{k}\RbfT^jP\|_{L^2(\Sigma_\tau)}+\|\chi_{\geq R} \tilpartial_\Sigma X^{k}\RbfT^jP\|_{L^2(\Sigma_\tau)}\leq C_{j,k} \epsilon\tau^{-1-j},\quad}k+3j \leq M_P-2.\label{eq:TjPenergy1}\\
&\|\chi_{\geq R}r^{p} (\partial_r+\frac{1}{r}) X^{k}\RbfT^j\varepsilon\|_{L^2(\Sigma_\tau)}\leq C_{j,k} \epsilon\tau^{-1-j+\frac{p}{2}},\quad 0\leq p\leq 2,~0\leq j\leq \NT,\nonumber\\
&\phantom{\|\chi_{\geq R}r^{p} (\partial_r+\frac{1}{r}) X^{k}\RbfT^j\varepsilon\|_{L^2(\Sigma_\tau)}\leq C_{j,k} \epsilon\tau^{-1-j+\frac{p}{2}},\quad}k+3j \leq M_\varepsilon-2.\label{eq:Tjphienergy2}\\
&\|\chi_{\geq R}r^{p} (\partial_r+\frac{1}{r}) X^{k}\RbfT^j P\|_{L^2(\Sigma_\tau)}\leq C_{j,k} \epsilon\tau^{-1-j+\frac{p}{2}},\quad 0\leq p\leq 2,~0\leq j\leq \NT,\nonumber \\
&\phantom{\|\chi_{\geq R}r^{p} (\partial_r+\frac{1}{r}) X^{k}\RbfT^j P\|_{L^2(\Sigma_\tau)}\leq C_{j,k} \epsilon\tau^{-1-j+\frac{p}{2}},\quad}k+3j \leq M_P-2.\label{eq:TjPenergy2}\\
&\| \chi_{\leq R}\partial^k\RbfT^j\varepsilon\|_{L^2(\Sigma_\tau)}+\|\jap{r}^{-\frac{3}{2}+\frac{\kappa}{2}}(\chi_{\geq R}X^k\RbfT^j\varepsilon)\|_{L^2(\Sigma_\tau)}\leq C_{j,k} \epsilon\tau^{-\frac{3}{2}-j+\frac{\kappa}{2}}, \nonumber\\
&\phantom{\| \chi_{\leq R}\partial^k\RbfT^j\varepsilon\|_{L^2(\Sigma_\tau)}+\|\jap{r}^{-\frac{3}{2}+\frac{\kappa}{2}}(\chi_{\geq R}X^k\RbfT^j\varepsilon)\|_{L^2(\Sigma_\tau)}}0\leq j \leq 1,~0\leq k+3j\leq M_\varepsilon-5.\label{eq:dphiL21}\\
&\| \chi_{\leq R}\partial^k\RbfT^jP\|_{L^2(\Sigma_\tau)}+\|\jap{r}^{-\frac{3}{2}+\frac{\kappa}{2}}(\chi_{\geq R}X^k\RbfT^jP)\|_{L^2(\Sigma_\tau)}\leq C_{j,k} \epsilon\tau^{-\frac{3}{2}-j+\frac{\kappa}{2}},\nonumber\\
&\phantom{\| \chi_{\leq R}\partial^k\RbfT^jP\|_{L^2(\Sigma_\tau)}+\|\jap{r}^{-\frac{3}{2}+\frac{\kappa}{2}}(\chi_{\geq R}X^k\RbfT^jP)\|_{L^2(\Sigma_\tau)}}0\leq j \leq 1,~0\leq k+3j\leq M_P-5.\label{eq:dPL21}\\
&|\RbfT^j\varepsilon|+\chi_{\geq R}|X^k\RbfT^j\varepsilon|\leq C_{j,k}\epsilon\tau^{-\frac{5}{4}-j+\frac{\kappa}{2}}, \quad0\leq j \leq 1,~ 0\leq k+3j\leq M_\varepsilon-7.\label{eq:varepptwiserp1}\\
&|\RbfT^jP|+\chi_{\geq R}|X^k\RbfT^jP|\leq C_{j,k}\epsilon\tau^{-\frac{5}{4}-j+\frac{\kappa}{2}}, \quad 0\leq j \leq 1,~0\leq k+3j\leq M_P-7.\label{eq:Pptwiserp1}\\
&|\partial^k\RbfT^j\varepsilon|+\chi_{\geq R}|\partial^{k-m}X^m\RbfT^j\varepsilon|\leq C_{j,k}\epsilon\tau^{-\frac{3}{2}-j+\frac{\kappa}{2}}, \quad0\leq j \leq 1,~1\leq k\leq M_\varepsilon-8,\nonumber\\
&\phantom{|\partial^k\RbfT^j\varepsilon|+\chi_{\geq R}|\partial^{k-m}X^m\RbfT^j\varepsilon|\leq C_{j,k}\epsilon\tau^{-\frac{5}{4}-j+\frac{\kappa}{2}},\quad}m<k,~0\leq k+3j\leq M_\varepsilon-8.\label{eq:varepptwiserp2}\\
&|\partial^k\RbfT^j P|+\chi_{\geq R}|\partial^{k-m}X^m\RbfT^j P|\leq C_{j,k}\epsilon\tau^{-\frac{3}{2}-j+\frac{\kappa}{2}}, \quad0\leq j \leq 1,~1\leq k\leq M_P-8,\nonumber\\
&\phantom{|\partial^k\RbfT^j\varepsilon|+\chi_{\geq R}|\partial^{k-m}X^m\RbfT^j\varepsilon|\leq C_{j,k}\epsilon\tau^{-\frac{5}{4}-j+\frac{\kappa}{2}},\quad}m<k,~0\leq k+3j\leq M_P-8.\label{eq:Pptwiserp2}\\
&\chi_{\geq R}|\partial^{k-m}X^m\RbfT^j\varepsilon|\leq C_{j,k}\epsilon\jap{r}^{-\frac{1}{2}}\tau^{-1-j+\frac{\kappa}{2}}, \quad0\leq j \leq 1,~1\leq k\leq M_\varepsilon-8,\nonumber\\
&\phantom{|\partial^k\RbfT^j\varepsilon|+\chi_{\geq R}|\partial^{k-m}X^m\RbfT^j\varepsilon|\leq C_{j,k}\epsilon\jap{r}^{-\frac{1}{2}}\tau^{-1-j+\frac{\kappa}{2}},\quad}m<k,~0\leq k+3j\leq M_\varepsilon-8.\label{eq:varepptwiserp3}\\
&\chi_{\geq R}|\partial^{k-m}X^m\RbfT^j P|\leq C_{j,k}\epsilon\jap{r}^{-\frac{1}{2}}\tau^{-1-j+\frac{\kappa}{2}}, \quad0\leq j \leq 1,~1\leq k\leq M_P-8,\nonumber\\
&\phantom{|\partial^k\RbfT^j\varepsilon|+\chi_{\geq R}|\partial^{k-m}X^m\RbfT^j\varepsilon|\leq C_{j,k}\epsilon\jap{r}^{-1}\tau^{-\frac{1}{2}-j+\frac{\kappa}{2}},\quad}m<k,~0\leq k+3j\leq M_P-8. \label{eq:Pptwiserp4}
\end{align}
\end{proposition}
\begin{proposition}\label{prop:boostrapvareptail}
Suppose the estimates \eqref{eq:a+trap}--\eqref{eq:varepptwisetailb1} and orthogonality conditions \eqref{eq:orth1} and \eqref{eq:orthpm1} are satisfied. If $\epsilon$ is sufficiently small and $C, C_{j,k}, C_k$ appearing on the right-hand side of \eqref{eq:a+trap}--\eqref{eq:varepptwisetailb1} are sufficiently large (compared to $C_{trap}$), then the following improved estimate holds:
\begin{align} 
&|\partial^k\varepsilon|+\chi_{\geq R}|X^k\varepsilon|\leq C_{k}(R_\ctf^2R^2)\epsilon \tau^{-\frac{9}{4}+\kappa},\quad 0\leq k\leq M-\tilC_\varepsilon.\label{eq:varepptwisetail1}
\end{align}
\end{proposition}
Assuming the conclusions of Propositions~\ref{prop:bootstrappar1},~\ref{prop:bootstrapphi1}, and~\ref{prop:boostrapvareptail}, we present the proof of Theorem~\ref{thm:main}.
\begin{proof}[Proof of Theorem~\ref{thm:main}]
Given $(\psi_0,\psi_1)$ we artificially define $\psi$ on the time interval $[-1,0]$ as $\psi(t)=\psi_0+t\psi_1$.  For each $b$ (see the statement of Theorem~\ref{thm:main}) we let $\tau_f(b)$ be the maximal time on which there is a solution parameterized as in \eqref{eq:psidef1} such that the bootstrap assumptions \eqref{eq:a+b1}--\eqref{eq:varepptwisetailb1} and orthogonality conditions \eqref{eq:orth1} and \eqref{eq:orthpm1} are satisfied. By Proposition~\ref{prop:LWP} and the implicit function theorem arguments in Sections~\ref{subsec:wpeq} and~\ref{subsec:apm}, we know that $\tau_f(b)$ is strictly positive for each choice of $b$. We want to show that $\tau_f(b)$ is infinite for some choice of $b$. Suppose not. First we show that \eqref{eq:a+trap} must get saturated, that is, the inequality must be an equality, at $\tau=\tau_f$. Indeed, fix $b$ and let $\tau^\ast\in(0,\infty)$ be such that the bootstrap conditions (including the orthogonality conditions and the parameterization \eqref{eq:psidef1}) are satisfied on $[0,\tau^\ast]$. Suppose \eqref{eq:a+trap} is strict on $[0,\tau^\ast]$. By Propositions~\ref{prop:bootstrappar1} and~\ref{prop:bootstrapphi1} (with $\tau_f$ replaced by $\tau^\ast$) we can improve the bootstrap assumptions \eqref{eq:a+b1}--\eqref{eq:varepptwisetailb1} on $[0,\tau^\ast]$. By Proposition~\ref{prop:LWP} applied with $\ell_0$ and $\xi_0$ fixed at values of $\ell$ and $\xi$ close to $\tau^\ast$, we can extend the solution on an interval of size of order one beyond $\tau^\ast$. By the implicit function theorem as in Sections~\ref{subsec:wpeq} and~\ref{subsec:apm} we can extend $\xi$ and $\ell$ and the parameterization \eqref{eq:psidef1} beyond $\tau^\ast$ such that the orthogonality conditions \eqref{eq:orth1} and \eqref{eq:orthpm1} are still satisfied. Now since \eqref{eq:a+trap} is strict on $[0,\tau^\ast]$, by continuity it is still satisfied on a larger interval. It follows that on this larger interval all the bootstrap conditions are satisfied and hence $\tau_f(b)>\tau^\ast$. This shows that condition \eqref{eq:a+trap} must get saturated at some time.

Arguing by contradiction, assume that $\tau_f(b)$ is finite for every choice of $b$. Let
\begin{align*}
\begin{split}
q(\tau):=\mu(\mu a_{+}(\tau)-e^{\mu \tau}S(e^{-\mu\tau}\tilF_{+}))-e^{\mu \tau}S(e^{-\mu\tau}\dottilF_{+}),\qquad \uplambda(\tau):=C_\trap\delta_\wp \epsilon \jap{\uptau}^{-3},\qquad \uplambda_0:=C_\mathrm{trap}\delta_\wp\epsilon.
\end{split}
\end{align*}
Note that $q=\ddot{a}_{+}$. We claim that if $(\psi_0,\psi_1)$ satisfy the orthogonality condition \eqref{eq:codim1} below, then for each $|q_0|\leq \lambda_0$ there is a choice of $b$ in a neighborhood of zero for which $q(0)=q_0$.  Recall that the orthogonality condition \eqref{eq:orthpm1} determines $a_{+}$ by
\begin{align*}
\begin{split}
a_+(t)=\bfOmega(\vectilupphi,\vecZ_{-})-e^{\mu t}\tilS(e^{-\mu t}\tilF_{+}).
\end{split}
\end{align*}
Define  $\calZ:C^\infty(\barcalC)\times C^\infty(\barcalC)\times I\to \bbR$, where $I$ is a neighborhood of zero in $\bbR$, by
\begin{align*}
\begin{split}
\calZ(\psi_0,\psi_1,b)=q(0).
\end{split}
\end{align*}
Here $q$ is determined using initial data
\begin{align}\label{eq:shootingbdata1}
\begin{split}
 \Phi\vert_{\{t=0\}}=\Phi_0[\epsilon(\psi_0+b\tilvarphi_\mu)]\mand \partial_t\Phi\vert_{\{t=0\}}=\Phi_1[\epsilon(\psi_1-\mu b\tilvarphi_\mu)],
\end{split}
\end{align}
as in the statement of Theorem~\ref{thm:main}. We then restrict attention to $(\psi_0,\psi_1)$ satisfying the codimension one condition
\begin{align}\label{eq:codim1}
\begin{split}
\calZ(\psi_0,\psi_1,0)=0.
\end{split}
\end{align}
Arguing as for the implicit function theorem in Section~\ref{subsec:apm}, we see that $\big|\frac{\partial q_0}{\partial b}\vert_{(\psi_0,\psi_1,0)}\big|\gtrsim1$. Our claim then follows from from the (calculus) implicit function theorem and \eqref{eq:codim1}.
 
It follows that for every such $q_0$ there is $\tau_\trap(q_0)$ and a solution with $q(0)=q_0$ that satisfies $|q(\tau)|<\uplambda(\tau)$ for $\tau<\tau_\trap(q_0)$ and $|q(\tau_\trap(q_0))|=\uplambda(\tau_\trap(q_0))$. We use a standard shooting argument  to derive a contradiction from this. The main observation is that if $\frac{1}{2}\uplambda(\tau)<| q(\tau)|<\uplambda(\tau)$ for some $\tau\leq \tau_f$, then 
\begin{align}\label{eq:outgoing1}
\begin{split}
\frac{\ud}{\ud\tau}q^2(\tau)\geq \mu q^2(\tau).
\end{split}
\end{align}
Indeed, rewriting equation~\eqref{eq:apmfinal1} for $a_{+}$ as
\begin{align*}
\begin{split}
\dot{q}(\tau)= \mu q(\tau)-e^{\mu\tau} S(e^{-\mu \tau}\dottilF_{+}(\tau)),
\end{split}
\end{align*}
and multiplying by $2q(\tau)$, the first term on the right gives $2\mu q^2(\tau)$. For the error term, by the same arguments as in the proofs of Proposition~\ref{prop:bootstrappar1} and Lemma~\ref{lem:orthTk} in Sectin~\ref{sec:ODEanalysis} below,
\begin{align*}
\begin{split}
|e^{\mu\tau} S(e^{-\mu \tau}\dotF_{+}(\tau))|\leq c \uplambda(\tau)< c q(\tau),
\end{split}
\end{align*}
for some $c\ll \mu$, proving \eqref{eq:outgoing1}.  To derive the desired contradiction it suffices to show that the map $\Lambda:(-\uplambda_0,\uplambda_0)\to\{\pm\uplambda_0\}$, $\Lambda(q_0)=q(\tau_\trap(q_0))\jap{\tau_\trap(q_0)}^3$ is continuous. Indeed, by \eqref{eq:outgoing1}, $\Lambda(q_0)=-\uplambda_0$ if $q_0$ is close to $-\uplambda_0$ and $\Lambda(q_0)=\uplambda_0$ if $q_0$ is close to $\uplambda_0$, so the continuity of $\Lambda$ contradicts the intermediate value theorem. By continuous dependence on initial data, continuity of $\Lambda$ follows from that of $\tau_\trap$. Fix $q_0\in(-\uplambda_0,\uplambda_0)$ and let $q$ denote the corresponding solution. By \eqref{eq:outgoing1}, given $\tilde\epsilon>0$ there exists $\tilde\delta\in(0,1)$ such that if $(1-\tilde\delta)\uplambda(\tau)<|q(\tau)|<\uplambda(\tau)$ for some $\tau<\tau_f$, then $|\tau_\trap(q_0)-\tau|<\tilde\epsilon$. Let $\tau_1<\tau_f$ be such that $(1-{\tilde\delta}^2)\uplambda(\tau_1)<|q(\tau_1)|<(1-{\tilde{\delta}}^3)\uplambda(\tau_1)$, and note that if $q_1$ is sufficiently close to $q_0$ then the solution $\tilq$ corresponding to $q_1$ satisfies $(1-\tilde\delta)\uplambda(\tau_1)<|\tilq(\tau_1)|<\uplambda(\tau_1)$, and hence $|\tau_\trap(q_0)-\tau_\trap(q_1)|\leq |\tau_\trap(q_0)-\tau_1|+|\tau_\trap(q_1)-\tau_1|<2\tilde\epsilon$.
\end{proof}
\section{Analysis of modulation equations}\label{sec:ODEanalysis}
In this section we derive estimates for the parameter derivatives as well as $\bfOmega_i(\RbfT^k\varepsilon):=\bfOmega(\RbfT^k\vecvarepsilon,\vecZ_i)$, $i=\pm,1,\dots,6$. Specifically, our goal is to prove Proposition~\ref{prop:bootstrappar1} as well as the following lemma.
\begin{lemma}\label{lem:orthTk}
Suppose the bootstrap assumptions \eqref{eq:a+trap}--\eqref{eq:varepptwisetailb1} hold. If $R_\ctf$ is sufficiently large, then for $k=0,1,2$, $j\geq0$, and $i=\pm,1,\dots,6$,  and any $t_1\leq t_2$, 
\begin{align*}
\begin{split}
&\|\bfOmega_i(\RbfT^{k}\varepsilon)\|_{L^2[t_1,t_2]}+\|a_{\pm}^{(k+j)}\|_{L^2[t_1,t_2]}+\|\dotwp^{(k+1+j)}\|_{L^2[t_1,t_2]}\\
&\lesssim o_{\wp,R_\ctf}(1)(\|\RbfT^k\varepsilon\|_{LE[t_1,t_2]}+\|\RbfT^kP\|_{LE[t_1,t_2]}\\
&\phantom{\lesssim o_{\wp,R_\ctf}(1)(}+\sup_{t_1\leq \tau\leq t_2}(\|\RbfT^k\varepsilon\|_{E(\Sigma_\tau)}+\|\RbfT^kP\|_{E(\Sigma_\tau)}))+o_{\wp,R_\ctf}(1)\epsilon\jap{t_1}^{-3}.
\end{split}
\end{align*}
For $k>2$,
\begin{align*}
\begin{split}
\|\bfOmega_i(\RbfT^{k}\varepsilon)\|_{L^2[t_1,t_2]}&\lesssim o_{\wp,R_\ctf}(1)\epsilon\jap{t_1}^{-3}\\
&\quad +o_{\wp,R_\ctf}(1)\sum_{j=2}^k(\|\RbfT^j\varepsilon\|_{LE[t_1,t_2]}+\|\RbfT^jP\|_{LE[t_1,t_2]}\\
&\phantom{\quad +o_{\wp,R_\ctf}(1)\sum_{j=2}^k(}+\sup_{t_1\leq \tau\leq t_2}(\|\RbfT^j\varepsilon\|_{E(\Sigma_\tau)}+\|\RbfT^jP\|_{E(\Sigma_\tau)})).
\end{split}
\end{align*}
\end{lemma}
We start with Proposition~\ref{prop:bootstrappar1}, where we assume the result of Lemma~\ref{lem:orthTk}. The proof is for the most part the same as that of [LuOS, Proposition~5.1], so we will be brief in details. The main difference is the presence of the modified profile $P$ and our use of Lemma~\ref{lem:orthTk} to prove \eqref{eq:a+2}, \eqref{eq:a-2}, and \eqref{eq:wp2}.
\begin{proof}[Proof of Proposition~\ref{prop:bootstrappar1}]
We start with estimate \eqref{eq:wp1} for $\dotwp$. Recall that $\wp$ satisfies \eqref{eq:parODE1} where $\vecomega$ is the solution of the second equation in \eqref{eq:vecomegaODE1} given by \eqref{eq:vecomegaODE2}. Since $\vecF_\omega$ is quadratic (see \eqref{eq:Fomegaquad1}), the contribution of $\vecomega$ is a quadratic error (see [LuOS, Lemma~6.1] for further details) and, recalling \eqref{eq:vecGbound1}, the main contribution to $\dotwp$ in \eqref{eq:parODE1} comes from $S(\partial_t\vecbfOmega_P+\vecN)$. Using the notation of Section~\ref{subsec:wpeq}, we consider the contributions of $\vecbfOmega_{\partial_tP}+\vecN_P$ and $\vecN_\tilupphi$ separately. Here we have noted that the difference between $\partial_t\vecbfOmega_P$ and $\vecbfOmega_{\partial_tP}$ is quadratic and can be bounded by the right-hand side of \eqref{eq:wp1} using the bootstrap assumptions. For $\bfOmega(\partial_t\vecP-M\vecP,\vecZ_i)$, we want to use equation \eqref{eq:Pdef1} for $P$. In view of the definition \eqref{eq:vecPdef1}, we have (note that $g_\ext$ and $\vecZ_i$ have disjoint supports)
\begin{align}\label{eq:vecPtemp1}
\begin{split}
(\partial_t-M)\vecP=\pmat{0\\\sqrt{|h|}(\txtg_1+\partial_\uptau\txtg_2)},\qquad \mathrm{in~}\supp \vecZ_i.
\end{split}
\end{align}
Recalling \eqref{eq:bfOmegadef1}, we need to estimate
\begin{align*}
\begin{split}
\int Z_i\sqrt{|h|}(\txtg_1+\partial_\uptau\txtg_2)\ud\omega\ud\rho.
\end{split}
\end{align*}
This integral is bounded by the right-hand side of \eqref{eq:wp1} in view of \eqref{eq:txtgtemp4}, \eqref{eq:txtgtemp3},  \eqref{eq:giorthdef1}, and the bootstrap assumptions, where we also use the support properties of $Y_i$ in \eqref{eq:WYdef1} to estimate $\int P\calP_hY_i\sqrt{|h|}\ud x$ in \eqref{eq:txtgtemp4}. Here in the contribution of $\partial_\tau\txtg_2$ for the term $(h^{-1})^{0\nu}(Y_i+\frac{W_i}{(h^{-1})^{00}})\partial^2_{\tau\nu} P$ we use the energy estimates \eqref{eq:TjPenergyb1} and \eqref{eq:dPL2b1} (see the proof of Proposition~\ref{prop:boostrapvareptail} below for a similar estimate). For the contribution of $N_\tilupphi$ we first write $\vectilupphi=\vecvarepsilon+a_{+}\vecZ_{+}+\vecZ_{-}$. For $a_{\pm}$ we use the bootstrap assumptions and gain extra smallness from the almost vanishing of $\bfOmega(M\vecZ_{i},\vecZ_{\pm})$, $i=1,\dots,6$. For $\bfOmega(\vecvarepsilon,M\vecZ_i)$, $i=1,2,3$, we use the bootstrap assumptions, in particular \eqref{eq:varepptwisetailb1}, and the almost vanishing of $M\vecZ_i$ to estimate
\begin{align*}
\begin{split}
|\bfOmega(\vecvarepsilon,M\vecZ_i)|\lesssim |\ell| (R_\ctf^2R^2)\epsilon \tau^{-\frac{9}{4}+\kappa}+(R_\ctf^2R^2)\epsilon \tau^{-\frac{9}{4}+\kappa}\int_{\{\rho\simeq R_\ctf\}}\rho^{-2}\ud\rho\lesssim \delta_\wp(R_\ctf^2R^2)\epsilon \tau^{-\frac{9}{4}+\kappa}.
\end{split}
\end{align*}
For $i=1,2,3$ the contribution of $\bfOmega(\vecvarepsilon,M\vecZ_{3+i}-\vecZ_{i})$ is estimated similarly, while for $\bfOmega(\vecvarepsilon,\vecZ_i)$ we use the orthogonality conditions \eqref{eq:orth1}. The quadratic error $\vecomega$ in \eqref{eq:orth1} can be estimated as above. For $\vecUpomega_i$, $i=1,2,3$, we rewrite \eqref{eq:vecUpomegadef1} as
\begin{align*}
\begin{split}
\vecUpomega  = \tilS\vecN_\tilupphi+\tilS(\partial_t\vecOmega_P+\vecN_P).
\end{split}
\end{align*}
Note that here we need to bound the first three components of $\vecUpomega$, but these components are bounded exactly as above, where again we use equation \eqref{eq:vecPtemp1} for $\vecP$. This completes the proof of \eqref{eq:wp1}. For \eqref{eq:wp2} we use Lemma~\ref{lem:orthTk} and Sobolev estimates in time. Specifically, note that for any~$\tau$,
\begin{align*}
\begin{split}
|\dotwp^{(k)}(\tau)|\lesssim \tau^{-\frac{1}{2}} \|\dotwp^{(k)}\|_{L^2[\tau,2\tau]}+\tau^{\frac{1}{2}}\|\dotwp^{(k+1)}\|_{L^2[\tau,2\tau]}.
\end{split}
\end{align*}
We can now estimate the right-hand side using Lemma~\ref{lem:orthTk} and the bootstrap assumptions \eqref{eq:Tjphienergyb1} and \eqref{eq:TjPenergyb1}. 
The estimates for $a_{\pm}$ are similar where for $a_{-}$ we write
\begin{align}\label{eq:a-temp1}
\begin{split}
a_{-}(t)= a_{-}(0)e^{-\mu t}+e^{-\mu t}\int_0^t S(e^{\mu s}\tilF_{-}(s))\ud s,
\end{split}
\end{align}
while for $a_{+}$ we use \eqref{eq:a+trap} to express $a_{+}$ in terms of $\tilF_{\pm}$ and a term with better decay (by the assumption~\eqref{eq:a+trap}). See \cite[Proposition~5.1]{LuOS1} for more details. Here we provide some more details on how to estimate the contribution of $\tilF_{-}$ in \eqref{eq:a-temp1}. Recall the expression for $\tilF_{-}$ from \eqref{eq:F-2}. As above for the contribution of $\vecP_\wp$ we use equation~\eqref{eq:vecPtemp1}, where the extra smallness comes from the smallness of $c_j$. For the remaining terms on the right-hand side of \eqref{eq:F-2}, except the first term involving $\vecF_1$, the smallness comes from the smallness of $M\vecZ_{\pm}\mp\mu\vecZ_{\pm}$. For the first term on the right-hand side of \eqref{eq:F-2} recall the definition of $\vecF_1$ from \eqref{eq:firstorder1}. The smallness from the contribution of $\vecf$ comes from estimate \eqref{eq:vecfbound1}, while for $\vecK$ the smallness comes from the almost orthogonality of \eqref{eq:K1storder1} with $\vecZ_\pm$.
\end{proof}
Next, we turn to the proof of Lemma~\ref{lem:orthTk}. This lemma will be needed in absorbing the contribution of the parameters in the energy estimates.
\begin{proof}[Proof of Lemma~\ref{lem:orthTk}]
The proof is for the most part similar to that of Proposition~\ref{prop:bootstrappar1} above, so we only highlight the main differences. The estimates for $\bfOmega_i(\RbfT^k\vecvarepsilon)$ are similar to how we bounded $\bfOmega_i(\vecvarepsilon)$ in the proof of Proposition~\ref{prop:bootstrappar1}, where we use the orthogonality conditions, with the difference that instead of using pointwise bounds on $\varepsilon$, we now use the spatial part of the local energy norm. That is, we estimate
\begin{align*}
\begin{split}
\int_{\{\rho\simeq R_\ctf\}} \epsilon \rho^{-4}\ud V_{\Sigma_\tau}\lesssim R_\ctf^{-1+\frac{\alpha}{2}} \Big(\int_{\{\rho\lesssim R_\ctf\}}\frac{\epsilon^2}{\jap{\rho}^{3+\alpha}}\ud V_{\Sigma_\tau}\Big)^{\frac{1}{2}}.
\end{split}
\end{align*}
The negative power of $R_\ctf$ is the desired small factor. Similarly, in estimating the contribution of $\partial_\uptau\txtg_2$ in \eqref{eq:vecPtemp1} we use \eqref{eq:Iic1},~\eqref{eq:tilIIic1}, and \eqref{eq:c2prime2} instead of \eqref{eq:txtgtemp3}. Note that, using the first order formulation of the equations, to estimate $\bfOmega_i(\RbfT^k\varepsilon)$ in this way at most $k$ time derivatives, $\RbfT$, of $\varepsilon$ appear on the right-hand side and spatial derivatives can be integrated by parts to $\vecZ_i$. Thus, even though $\tilS$ is not (infinitely) smoothing, there is no loss of regularity in this process. With the estimates for  $\bfOmega_i(\RbfT^k\varepsilon)$ in hand, the estimates for $\dotwp^{(k)}$ and $a_{-}^{(k)}$ follow as in the proof of Proposition~\ref{prop:bootstrappar1} and with similar modifications as above. For $a_{+}$ we have to argue a bit differently and use the trapping assumption~\eqref{eq:a+trap}. Here, for $\ddot a_{+}$ and any $t \leq \tau_f$ we write (using the notational convention introduced below \eqref{eq:a+trap})
\begin{align}\label{eq:atemp1}
\begin{split}
\ddot a_{+}(t) &= \big(\mu(\mu a_{+}(\tau_f)-e^{\mu \tau_f}S(e^{-\mu \tau_f}F_{+}(\tau_f)))-e^{\mu \tau_f}S(e^{-\mu \tau_f}\dotF_{+}(\tau_f)\big)e^{-\mu(\tau_f-t)}\\
&\quad+\int_t^{\tau_f}e^{-\mu(s-t)}(e^{\mu s} S(e^{-\mu s}\frac{\ud^2}{\ud s^2} \tilF_{+}(s)))\ud s.
\end{split}
\end{align}
The desired estimate then follows by applying Schur's test. For the higher order derivatives, we simply differentiate \eqref{eq:atemp1} and absorb any excess time derivatives by the smoothing operator $S$. For $a_{+}$ we rearrange the equation for $\ddot a_{+}$ as
\begin{align*}
\begin{split}
\mu^2 a_{+}(t)=\ddot a_{+}(t)+\mu(e^{\mu t}S(e^{-\mu t}\tilF_{+}))+e^{\mu t}S(e^{-\mu t}\dot \tilF_{+}(t)),
\end{split}
\end{align*}
and use the estimates we have already established for $\ddot a_{+}$. Similarly, the estimate for $\dota_{+}$ follows from differentiating this relation.
\end{proof}
\begin{remark}\label{rem:deltawp1}
An inspection of the proofs of Proposition~\ref{prop:bootstrappar1} and Lemma~\ref{lem:orthTk} reveals that the small constants $\delta_\wp$ and $o_{\wp,R_\ctf}(1)$ depend on $|\ell|$ and $R_\ctf$ in the following way: either there is smallness of order $O(|\ell|)$ or the smallness is $O(R_\ctf^{-1+\frac{\alpha}{2}})$.
\end{remark}
\section{Uniform energy bound and integrated local energy decay}\label{sec:ILED}
We continue to use the notation $L^p_x$ for $L^p(\Sigma_\uptau)$ in this section. Our goal is to prove energy and ILED estimates for the equation (written in the non-geometric global coordinates $(\uptau,\uprho,\uptheta)$; see also \eqref{eq:calPdef1})
\begin{align}\label{eq:genILEDeqn1}
\begin{split}
\calP \uppsi = f_1+f_2.
\end{split}
\end{align}
Here $\calP$ is given in \eqref{eq:calPdef1} and \eqref{eq:calPhdef1}. It is assumed to satisfy the properties described in Section~\ref{subsubsec:2ndordereq}, in particular~\eqref{eq:calPP0Ppertdecomp1} and~\eqref{eq:tilDelta1}. $\uppsi$, $f_1$, and $f_2$ are functions defined on $\Sigma_{t_1}^{t_2}:=\cup_{\tau=t_1}^{t_2}\Sigma_\tau$ for some $t_1<t_2$. In our applications, $f_2$ is the main part of the source term in the interior which has the structure $\dotwp^{(\leq2)}\fybar_i$, where $\fybar_i$ are the eigenfunctions of $\uH$. See Lemma~\ref{lem:calG01}. $f_1$ contains the nonlinearity as well as the part of the source term has extra smallness of order $O(|\ell|)$ in the interior. To use the decay of the parameter derivatives $\dotwp$, we will need to place $f_2$, and the part of $f_1$ which contains the source term, in $L^p_\uptau X$ spaces with $p\geq 2$, $X$ being some (possibly weighted) $L^2_x$ space. To get sufficient spatial decay for $f_2$ we will perform some integration by parts in the proof of the energy estimates, and use the Darboux transform introduced in Section~\ref{subsubsec:2ndordereq} (see \eqref{eq:Dardef1}) in the proof of the ILED estimate. These considerations are analogous to the case of the corresponding estimates on the product catenoid which we studied in Proposition~\ref{prop:ILEDprod1}. The main result of this section is stated in Proposition~\ref{prop:LED1} below. Before stating this proposition we recall and define some necessary notation. For the orthogonality conditions, since the statements in this section is for the general linear equation~\eqref{eq:genILEDeqn1}, we use the following linear substitutes for our $\bfOmega_k$ (see Remark~\ref{rem:orthcomparison1} below):
\begin{align*}
\begin{split}
&\bfUpomega_k(\uppsi)(c)=-\int_{\{\uptau= c\}}Z_kn^\alpha\partial_\alpha \uppsi \sqrt{|h|}\ud y,\quad \bfUpomega_{3+k}(\uppsi(c))=\int_{\{\uptau= c\}}\uppsi Z_kn^\alpha\partial_\alpha \uptau \sqrt{|h|}\ud y,\quad k=1,2,3,\\
&\bfUpomega_{\pm\mu}(\uppsi)(c)=\int_{\{\uptau= c\}}(\pm\mu \uppsi Z_\mu \partial_\alpha \tiluptau-Z_\mu \partial_\alpha \uppsi)n^\alpha \sqrt{|h|}\ud y.
\end{split}
\end{align*} 
Here $y$ denotes the spatial variables $(\uprho,\uptheta)$ on $\Sigma_c$ and $n$ denotes the normal vector with respect to $h$. For the energy we use the notation
\begin{align}\label{eq:standardenergydef1}
\begin{split}
&E[\uppsi](\uptau)\equiv\|\uppsi\|_{E(\Sigma_\uptau)}^2\\
&:=\int_{\Sigma_\uptau}\Big[\chi_{\leq \tilR}(|\partial\uppsi|^2+\jap{\uprho}^{-2}|\uppsi|^2)\sqrt{|h|}+\chi_{\geq \tilR}(|\partial_\Sigma \uppsi|^2+\uprho^{-2}|\RbfT\uppsi|^2+\uprho^{-2}|\uppsi|^2) \Big]\ud y.
\end{split}
\end{align}
Here $\chi_{\geq \tilR}$ is a cutoff function supported in $\calC_\hyp$, for some fixed large $\tilR\gg1$, and $\chi_{\leq \tilR}=1-\chi_{\geq \tilR}$. The local energy norm and its dual on any (space-time) region $\calR$ of the domain of definition of $\uppsi$ is defined by
\begin{align}\label{eq:standardLEdef1}
\begin{split}
&\|\uppsi\|_{LE(\calR)}^2:=\int_{\calR}\chi_{\leq \tilR} ((\uprho\tilpsi)^2+(\uprho\partial \tilpsi)^2+( \partial_\uprho\tilpsi)^2)\sqrt{|h|}\ud y+\int_{\calR}\chi_{\geq \tilR}( \uprho^{-3-\alpha}\uppsi^2+\uprho^{-1-\alpha}(\partial\uppsi)^2)\sqrt{|h|}\ud y,\\
&\|f\|_{LE^\ast(\calR)}^2:=\int_{\calR}\chi_{\leq \tilR} f^2\ud V+\int_{\calR}\chi_{\geq \tilR}\uprho^{1+\alpha}f^2\sqrt{|h|}\ud y.
\end{split}
\end{align}
Here $0<\alpha\ll 1$ is a fixed small positive number. We use the notation
$$
\|\uppsi\|_{L^pL^q(\Sigma_{t_1}^{t_2})}=\Big(\int_{t_1}^{t_2}\|\uppsi\|_{L^q(\Sigma_\uptau)}^p\ud \tau\Big)^{\frac{1}{p}},
$$
with the usual modificaion when $p=\infty$. When $p=q$ we simply write $\|\uppsi\|_{L^p(\Sigma_{t_1}^{t_2})}$, and similarly with $\Sigma_{t_1}^{{t_2}}$ replaced by any other region. 
\begin{proposition}\label{prop:LED1}
Suppose $\uppsi$ satisfies $\calP\uppsi=f_1+f_2$. For any $t_1<t_2$, $\uppsi$ satisfies the estimate
\begin{align}\label{eq:ILEDlingeneral1}
\begin{split}
\|\uppsi\|_{LE(\Sigma_{t_1}^{t_2})}+\sup_{t_1\leq t\leq t_2}\|\uppsi\|_{E(\Sigma_{t})}
&\lesssim \sum_{k\in\{\pm\mu,1,\dots,6\}}\|\bfUpomega_k(\uppsi)\|_{L^2([t_1,t_2])}+\|\jap{\RbfT}f_2\|_{LE^\ast(\Sigma_{t_1}^{t_2})}\\
&\quad+R_\ctf^{2+\frac{\alpha}{2}}(\|f_1\|_{L^1 L^2(\Sigma_{t_1}^{t_2})}+\|\uppsi\|_{E(\Sigma_{t_1})}+\|\Dar \calS_1f_2\|_{LE^\ast(\Sigma_{t_1}^{t_2})}).
\end{split}
\end{align}
\end{proposition}
\begin{remark}\label{rem:orthcomparison1}
For the purpose of the linear estimate \eqref{eq:ILEDlingeneral1} we only need $\|\bfUpomega_k(\uppsi)\|_{L^2([t_1,t_2])}$, $k=\mu,1,2,3$, on the right-hand side. But, in practice these are controlled using our orthogonality conditions which also involve $k=-\mu,4,5,6$. Also as mentioned above $\bfUpomega_k(\uppsi)$ are proxies for our orthogonality conditions in terms of $\bfOmega_k(\uppsi)$, where $\uppsi$ corresponds to the perturbation $\varepsilon$. In our applications it will be easy to pass from one to the other, because the difference $\|\bfUpomega_k(\uppsi)-\bfOmega_k(\uppsi)\|_{L^2[t_1,t_2]}$ is bounded by a small factor of order $O(|\ell|)$ times the $LE$ norm of $\psi$, which can be absorbed in the left-hand side in the estimate.
\end{remark}
\begin{proof}[Proof of Proposition~\ref{prop:LED1}]
For $\tilM>0$ sufficiently large, let $V_\far=\tilM V_\temp $ where $V_\temp$ is a compactly supported potential with $\supp V_\temp\subseteq \{|\uprho|\leq \Rled\ll R_\ctf \}$ and
\begin{align*}
\begin{split}
(\sgn\uprho)\partial_\uprho V_\temp(\uprho)\leq -v_0<0 \mathrm{~for~}|\uprho|\leq \Rled/2,\quad (\sgn\uprho)\partial_\uprho V_\temp(\uprho)\leq 0\mathrm{~for~all~}|\uprho|.
\end{split}
\end{align*}
Let $\uppsi_\far$ be the solution of
\begin{align*}
\begin{split}
(\calP-V_\far)\uppsi_\far=f_1,\qquad (\uppsi_\far,\partial_\uptau\uppsi_\far)\vert_{\Sigma_{t_1}}=(\uppsi,\partial_\uptau\uppsi)\vert_{\Sigma_{t_1}}.
\end{split}
\end{align*}
By an identical argument as in \cite[Lemma~7.6]{LuOS1} (see also \cite[Proposition~7.1]{LuOS1} but treating the contribution of $f_{1} \bfT \uppsi_{\far}$ as in the proof of Lemma~\ref{lem:ILEDprodj}), $\uppsi_\far$ satisfies
\begin{align*}
\begin{split}
\|\uppsi_\far\|_{LE(\Sigma_{t_1}^{t_2})}\lesssim \|\uppsi\|_{E(\Sigma_{t_1})}+\|f_1\|_{L^1L^2(\Sigma_{t_1}^{t_2})}+\|\jap{\RbfT}f_1\|_{LE^\ast(\Sigma_{t_1}^{t_2})}.
\end{split}
\end{align*}
Here we have added a multiple of the energy estimate for $\uppsi_\far$. Next, observe that $\uppsi_\near:=\uppsi-\uppsi_\far$ satisfies $\calP\uppsi_\near=f_2-V_\far\uppsi_\far$. As in \cite[Lemma~7.6]{LuOS1}, it follows by the same argument as above that with $K_{t_1}^{t_2}=[t_1,t_2]\times K$, $K$ a compact region in $\{|\uprho|\leq \Rled\}$,
\begin{align*}
\begin{split}
\|\uppsi_\near\|_{LE(\Sigma_{t_1}^{t_2})}&\lesssim \|\uppsi\|_{E(\Sigma_{t_1})}+\|f_1\|_{L^1 L^2(\Sigma_{t_1}^{t_2})}+ \|\jap{\RbfT}f_2\|_{LE^\ast(\Sigma_{t_1}^{t_2})}+\|\uppsi_\near\|_{L^2(K_{t_1}^{t_2})}.
\end{split}
\end{align*}
This follows by adding the ILED estimate for $\uppsi_\far$ and the corresponding one, with an $L^2$ error in $K_{t_1}^{t_2}$, for $\uppsi$. See the proof of \eqref{eq:psi1prodtemp1} in the proof of Proposition~\ref{prop:ILEDprod1} for a similar argument and how we treat the contribution of $V\uppsi$. To deal with the $L^2$ error in $K_{t_1}^{t_2}$ we use the operator with frozen coefficients in the coordinates $(\uptau,\uprho,\uptheta)$ as in Section~\ref{subsubsec:2ndordereq} (see~\eqref{eq:calPP0Ppertdecomp1}). That is, we write
\begin{align*}
\begin{split}
\calP= \calP_0+\calP_\pert,
\end{split}
\end{align*}
with $\calP_0$ having $\uptau$ independent coefficients, frozen at $\uptau=t_2$. See~\eqref{eq:calPP0Ppertdecomp1}. We also introduce the time frequency projection
\begin{align*}
\begin{split}
P_{\leq N_0}u(\uptau) = \int_{\bbR}2^{N_0}\chi(2^{N_0}\uptau')u(\uptau-\uptau')\ud \uptau', \qquad P_{>N_0}u = u-P_{\leq N_0}u,
\end{split}
\end{align*}
where $N_0$ is a fixed large number and $\hat\chi(\hat\uptau):=\int_{\bbR}e^{-i\uptau\hat\uptau}\chi(\uptau)\ud \uptau$ is a cutoff to the region $\{|\hat\uptau|\lesssim 1\}$. As in \cite{LuOS1}, to define the frequency projection on $\uppsi_\near$ we first extend the coefficients of $\calP$ globally in time and extend $\uppsi_\near$ by requiring it to solve a homogeneous equation for $\uptau\notin [t_1,t_2]$. See the paragraph preceding \cite[Lemma~7.7]{LuOS1} for more details. By \cite[Lemma~7.7]{LuOS1},
\begin{align*}
\begin{split}
\|P_{>N_0}\uppsi_\near\|_{L^2(K_{t_1}^{t_2})}\leq o_{N_0}(1)\|\uppsi_\near\|_{LE},
\end{split}
\end{align*}
where $LE$ denotes the global in time norm. Therefore, by choosing $N_0$ sufficiently large, it remains to control $\|P_{\leq N_0}\uppsi_\near\|_{L^2(K_{t_1}^{t_2})}$. Note that the coefficients of $\calP_\pert$ have $\uptau$ decay but not necessarily $\uprho$ decay, so their contribution needs to be handled carefully. For this we introduce another near-far decomposition $\uppsi_{\near,\far}$ and $\uppsi_{\near,\near}$ which are solutions to the following equations in $\Sigma_{t_1}^{t_2}$ (with suitable modifications for $\uptau\notin[t_1,t_2]$, see \cite[equations (7.21) and (7.22)]{LuOS1})
\begin{align*}
\begin{split}
(\calP_0-V_\far)\uppsi_{\near,\far}= -\calP_\pert \uppsi_\near-V_\far \uppsi_\far,\qquad \calP_0\uppsi_{\near,\near}=f_2-V_\far\uppsi_{\near,\far}.
\end{split}
\end{align*}
Since $V_\far$ and the coefficients of $\calP_0$ are independent of $\uptau$, the frequency projection $P_{\leq N_0}$ commutes with $\calP_0$ and $\calP_0-V_\far$. It follows from \cite[Lemma~7.8]{LuOS1} that (see~\eqref{eq:Ppertcoeffbound1} and~\eqref{eq:calPP0Ppertdecomp2})
\begin{align*}
\begin{split}
\|P_{\leq N_0}\uppsi_{\near,\far}\|_{L^2(\bbR \times K)}&\lesssim \|f_1\|_{L^1L^2(\Sigma_{t_1}^{t_2})}+\|\jap{\RbfT}f_2\|_{LE^\ast(\Sigma_{t_1}^{t_2})}\\
&\quad+\epsilon \big(\sup_{t_1\leq t\leq t_2}\|\uppsi\|_{E(\Sigma_{t})}+\|\uppsi_\near\|_{LE(\Sigma_{t_1}^{t_2})}\big).
\end{split}
\end{align*}
Finally we turn to the equation 
\begin{align*}
\begin{split}
\calP_0 P_{\leq N_0}\uppsi_{\near,\near}= P_{\leq N_0}f_2 -P_{\leq N_0}V_\far\uppsi_{\near,\far},
\end{split}
\end{align*}
whose analysis is where our proof differs from that of \cite[Proposition~7.2]{LuOS1}. Let $(\tiluptau,\tiluprho,\tiluptheta)$ be the geometric global coordinates in which $\calP_0$ takes the product form $-\partial_\tiluptau^2+\tilDelta+V(\tiluprho)$. See \eqref{eq:tilDelta1}. Proceeding as in Section~\ref{sec:ILEDproduct}, let $\calS_j$ denote the spherical harmonic projections as introduced there, but with respect to the coordinates $(\tiluprho,\tiluptheta)$. Note that since $K_{t_1}^{t_2}$ is contained in $\{|\uprho|\leq \Rled\}$ it suffices to estimate the $LE$ norm of $P_{\leq N_0}\uppsi_{\near,\near}$ in $\{\tilt_1\leq \tiluptau\leq \tilt_2\}$,  which by definition is the smallest infinite rectangle in the $(\tiluptau,\tiluprho,\tiluptheta)$ coordinates that contains $K_{t_1}^{t_2}$. See the figure below.
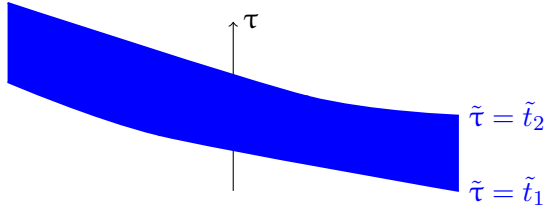
\begin{center}
\begin{tikzpicture}[scale=1,transform shape]
  \draw[->] (0,-0.25) -- (0,2) node[right] {$\uptau$};
  \draw[name path= C, red, very thick,decorate] (-1,0.5) -- (1,0.5) node[right] {$\uptau=t_1$};
  \draw[name path = D, red, very thick,-,decorate]  (-1,1) node[left] {$\uptau=t_2$}-- (1,1) ;
  \draw[red,very thick] (1,1) -- (1,0.5);
  \draw[red, very thick] (-1,1) -- (-1,0.5) (-0.5,0.76) node{$K_{t_1}^{t_2}$};
    \tikzfillbetween[of=C and D]{red, opacity=0.1};
    \coordinate  (A) at (-3,2.25);
\coordinate  (B) at (1,1);
\coordinate  (C) at (3,0.75);
\draw[name path=O, thick,blue] plot [smooth] coordinates { (A) (B) (C) };
    \coordinate  (D) at (-3,1.2);
\coordinate  (E) at (-1,0.5);
\coordinate  (F) at (3,-0.25);
\draw[name path = U, thick,blue] plot [smooth] coordinates { (D) (E) (F) };
\node[right] at (C) {$\Blue{\tiluptau=\tilt_2}$};
\node[right] at (F) {$\Blue{\tiluptau=\tilt_1}$};
\tikzfillbetween[of=O and U]{blue, opacity=0.1};
\end{tikzpicture}
\end{center}
Let
\begin{align*}
\begin{split}
\uppsi_1:=\calS_1P_{\leq N_0}\uppsi_{\near,\near},\quad \uppsi_{\neq1}:=\sum_{j\neq1}\calS_jP_{\leq N_0}\uppsi_{\near,\near}.
\end{split}
\end{align*}
The contribution of $\uppsi_{\neq1}$ can be estimated as in Lemma~\ref{lem:ILEDprodj} (see also \cite[equation~(7.24)]{LuOS1} for a similar computation). For the contribution of $\uppsi_1$ we use the Darboux transform exactly as in Lemma~\ref{lem:shred1} to get the desired estimate. The only remaining issue is to relate the last term on the right-hand side of \eqref{eq:shILEDtemp1}, which is now with respect to the $(\tiluptau,\tiluprho,\tiluptheta)$, to our orthogonality conditions. But this is done in the same manner as in \cite[Proposition~7.2]{LuOS1}. Indeed, the last term on the right-hand side of \eqref{eq:shILEDtemp1} corresponds exactly to $\tilbfOmegabar_k(\uppsi_1)$ in the notation of \cite[equation~(7.25)]{LuOS1}. This is then related to $\bfUpomega_k$ in the same way as in the argument treating equations (7.26) and (7.27) in \cite{LuOS1}.
\end{proof}
\section{Vector field method}\label{sec:VF}
In this section we prove various energy estimates for $\varepsilon$ and $P$  and their derivatives, and use these to derive preliminary pointwise estimates for $\RbfT^k\varepsilon$ and $\RbfT^kP$. For some of these estimates we will work  with the conjugated variable $\bsupvarepsilon$ introduced in \eqref{eq:upvarepsilondef1}. In view of \eqref{eq:sgraphdef1}, estimates for $\bsupvarepsilon$ are transferable to estimates on $\varepsilon$. Our starting point is the boundedness of the energy and local energy norms of $\varepsilon$ and $P$. We will use the following notation for the higher order energy and local energy norms (recall~\eqref{eq:standardenergydef1} and~\eqref{eq:standardLEdef1}):
\begin{align*}
\begin{split}
&\|\varepsilon\|^2_{LE_k[\tau_1,\tau_2]}:= \|\chi_{\lesssim 1}\partial^k\varepsilon\|^2_{LE[\tau_1,\tau_2]}+\|\chi_{\gtrsim 1}X^k\varepsilon\|^2_{LE[\tau_1,\tau_2]},\\
&\|\varepsilon(\tau)\|_{E_k}^2\equiv \|\varepsilon\|_{E_k(\Sigma_\tau)}^2:=\|\chi_{\lesssim 1}\partial^k\varepsilon\|^2_{E(\Sigma_\tau)}+\|\chi_{\gtrsim 1}X^k\varepsilon\|^2_{E(\Sigma_\tau)}.
\end{split}
\end{align*}
\begin{proposition}\label{prop:EILEDfinal1}
If the bootstrap assumptions \eqref{eq:a+trap}--\eqref{eq:varepptwisetailb1} are satisfied and $\epsilon$ is sufficiently small, then for $j=0,1,2$
\begin{align}\label{eq:EILEDfinal1}
\begin{split}
&\sup_{\tau_1\leq \tau\leq \tau_2}\|\RbfT^j\varepsilon(\tau)\|_{E_k}^2+\|\RbfT^j\varepsilon\|_{LE_k[\tau_1,\tau_2]}^2+\sup_{\tau_1\leq \tau\leq \tau_2}\|\RbfT^jP(\tau)\|_{E_k}^2+\|\RbfT^jP\|_{LE_k[\tau_1,\tau_2]}^2\\
&\lesssim \sum_{i\leq k} (\|\RbfT^j\varepsilon(\tau_1)\|_{E_i}^2+\|\RbfT^jP(\tau_1)\|_{E_i}^2)+\epsilon^2\bfupsigma(\tau_1),\qquad k\leq M_\varepsilon-2,
\end{split}
\end{align}
where $\bfupsigma(\tau_1)=\tau_1^{-2-2j}$ if $k+3j\leq M_\varepsilon-2$ and $\bfupsigma(\tau_1)=1$ otherwise. For $P$, the estimates remain valid in the range where $M_\varepsilon$ is replaced by $M_P$.
\end{proposition}
\begin{proof}
The proof consists of several steps. First, we use Proposition~\ref{prop:LED1} to establish the desired estimates for $X^k\RbfT^j\varepsilon$ and $X^k\RbfT^j P$ where all the vectorfields $X$ are equal to $\RbfT$. By elliptic theory, this yields the same estimates for size one derivatives applied on $\RbfT^j\varepsilon$ and $\RbfT^jP$, so it remains to consider the weighted derivatives in the exterior region.  Here we can work with the operator $\calP_g$ in the graph formulation in the region $\{|\uprho|\gtrsim1\}$, and use \eqref{eq:higher1}, \eqref{eq:errorhigher1}, and \eqref{eq:vecorder1} to directly prove the desired estimate by similar multiplier estimates as in the proof of Proposition~\ref{prop:LED1}. The resulting errors in the region $\{|\uprho|\simeq 1\}$ are absorbed by the estimates on the size one derivatives on $\RbfT^j\varepsilon$ from the previous step. Except for the first step, the details are the same as in the proof of Proposition 8.8 in \cite{LuOS1}, so we focus on the first step. Here the main new aspect is that we need to use the structure $f_1+\partial_\uptau f_2$ of the inhomogeneous term in Proposition~\ref{prop:LED1} to estimate the source term in the equation for $\varepsilon$ (that is, the terms that are independent of $\varepsilon$).

Turning to the details, we apply Proposition~\ref{prop:LED1} to $\RbfT^jP$ and $\RbfT^j\varepsilon$ and add the resulting estimates to absorb the errors. For this we start with equations \eqref{eq:Pdef1} and \eqref{eq:varepeq1} and observe that the terms on the right-hand side of \eqref{eq:Pdef1} all come with extra smallness, either from the spatial decay and support of $g_\Ext$ or the expressions \eqref{eq:txtgtemp4}, \eqref{eq:txtgtemp3}, and \eqref{eq:giorthdef1}. Also note that the dependence of $P$ on $\varepsilon$ comes only through the parameters $\wp$ in $g_\Ext$. In view of the smoothness of $\wp$ this implies that there is no potential loss of regularity in estimating the higher derivatives of $P$ appearing in the equation for $\varepsilon$ (see also Remark~\ref{rem:MpMvarep}). For \eqref{eq:Pdef1} our goal is to estimate the contribution of $\txtg_1$ and $\partial_\uptau \txtg_2$ by a small multiple of the $LE$ norm of $P$. The contribution of $g_\ext$ is strictly easier because of the spatial decay and support of $g_\ext$. In the notation of Proposition~\ref{prop:LED1} we take $f_1=0$. In view of \eqref{eq:txtgtemp4},~\eqref{eq:Iic1},~\eqref{eq:tilIIic1}, and \eqref{eq:c2prime2}, and analysis similar to the proof of Lemma~\ref{lem:orthTk} shows that
\begin{align}\label{eq:Penergynonlintemp1}
\begin{split}
\|\txtg_1\|_{L^2_\uptau[t_1,t_2]}+\|\partial_\uptau \txtg_2\|_{L^2_{\uptau}[t_1,t_2]}\lesssim (|\ell|+R_\ctf^{-1+\frac{\alpha}{2}})R^{\frac{\alpha}{2}}\|P\|_{LE[t_1,t_2]}.
\end{split}
\end{align}
Since $R^{\frac{\alpha}{2}} R_\ctf^{-1+\frac{\alpha}{2}}\ll1$, this gives the desired smallness for $\|f_2\|_{LE}$ in the first line of \eqref{eq:ILEDlingeneral1}. The term $\|\RbfT f_2\|_{LE}$ is treated by the same argument in view of the regularity considerations discussed above (specifically the smoothness of the parameters $\wp$ and the smoothing operator $S_P$ in \eqref{eq:c2prime2}). 
For the first term on the right-hand side of \eqref{eq:ILEDlingeneral1} we use the orthogonality condition~\eqref{eq:Porth1} and its time derivative. The corresponding term is then bounded in exactly the same way as~\eqref{eq:Penergynonlintemp1} (but without the $R^{\frac{\alpha}{2}}$ factor which came from the spatial support of $\txtg_i$).  Here for the contribution of $W_\mu$ we use the fact that $\fy^\stat_\mu$ is exponentially decaying to compare \eqref{eq:Porth1} with $\bfUpomega_{\pm\mu}$ in Proposition~\ref{prop:LED1}. See Remark~\ref{rem:orthcomparison1}. We also note that in the first term on the right-hand side of \eqref{rem:orthcomparison1}, when $\partial_\uptau$ falls on $\partial_\nu P$ we argue in exactly the same way as we passed from $\secondff_i$ to $\tilde{\secondff}_i$ in \eqref{eq:tilIIic1}. For the second line of \eqref{eq:ILEDlingeneral1}, in view of \eqref{eq:giorthdef1}, by the same argument as for the source term in Lemma~\ref{lem:calG01}, 
\begin{align}\label{eq:Dartxtg2temp1}
\begin{split}
\|\Dar \txtg_1\|_{LE^\ast[t_1,t_2]}+\|\Dar \partial_\uptau \txtg_2\|_{LE^\ast[t_1,t_2]}\lesssim  (R^{-1+\frac{\alpha}{2}}+|\ell|C_{R})\|P\|_{LE[t_1,t_2]}.
\end{split}
\end{align}
Since  $R_\ctf^{2+\frac{\alpha}{2}}R^{-1+\frac{\alpha}{2}}\ll 1$, we gain an overall smallness in $R_\ctf$ and $R$ so the corresponding contributions can be absorbed by the energy and $LE$ norms of $P$. The estimates for $\RbfT^jP$ are similar. Note that there is no loss of regularity coming from $\RbfT f_2$ on the right-hand side of \eqref{eq:ILEDlingeneral1} (which appears because of trapping) when $f_2=\partial_\uptau \RbfT^kg_2$, as any excess time derivatives can be absorbed by the smoothing operator $S_P$ (see \eqref{eq:txtgtemp3}).

For \eqref{eq:varepeq1} our goal is to estimate the contribution of the main source term in the equation for $\RbfT^j\varepsilon$ by $\|\dotwp^{(\geq j)}\|_{L^2_\uptau\cap L^\infty_\uptau}$ and $\|a_\pm^{(\geq j)}\|_{L^2_\uptau\cap L^\infty_\uptau}$. The nonlinearity is simply placed in $L^1L^2[\tau_1,\tau_2]$ as in $f_1$ in Proposition~\ref{prop:LED1}. Estimate \eqref{eq:EILEDfinal1} then follows from the bootstrap assumptions \eqref{eq:a+trap}--\eqref{eq:varepptwisetailb1} as well as Proposition~\ref{lem:orthTk} and Lemma~\ref{prop:bootstrappar1} for the parameters. Returning to the source term for the equation of $\varepsilon$ (the case of $\RbfT^j\varepsilon$ is similar) we recall from Lemma~\ref{lem:calG01} that for $\{|\uprho|\geq R+C\}$ it has the decay $O((|\dotwp^{(\geq1)}|+|a_\pm^{\geq1}|)\jap{\uprho}^{-3})$ so its $LE^\ast$ norm is bounded by a small, of order $O(R^{-1+\frac{\alpha}{2}})$, multiple of the $LE$ norm of $\varepsilon$.  For $\{|\uprho|\leq R+C\}$, denoting the source term by $f_1$ we recall that by Lemma~\ref{lem:calG01}, $\Dar f_2$ either comes with smallness of order $O(|\ell|)$, or that it is $O(|\uprho|^{-2})$ and is supported in $\{R-C\leq |\uprho|\leq R+C\}$. Since 
\begin{align*}
\begin{split}
\|(|\dotwp|+|\dota_{\pm}|)|\rho|^{-2}\|_{LE^\ast(\{R-C\leq |\uprho|\leq R+C\})}\lesssim (\|\dotwp\|_{L^2_\uptau}+\|\dota_{\pm}\|_{L^2_\uptau})R^{-\frac{1-\alpha}{2}},
\end{split}
\end{align*}
the contribution of the source term in $\{|\uprho|\leq R+C\}$ gives smallness of order $O(|\ell|+R^{-\frac{1-\alpha}{2}})$. Also note that the contribution of the source term in $LE^\ast$ coming from $f_2$ in Proposition~\ref{prop:LED1} is estimated in exactly the same way as in \eqref{eq:Penergynonlintemp1}. Since $R^{\frac{\alpha}{2}} R_\ctf^{-1+\frac{\alpha}{2}},R_\ctf^{2+\frac{\alpha}{2}}R^{-1+\frac{\alpha}{2}}\ll 1$, the in $LE^\ast[\tau_1,\tau_2]$ is bounded by a small multiple of the $LE$ norm of $\varepsilon$. 
Finally for the contribution of $\bfOmega_k(\varepsilon)$ we use Lemma~\ref{lem:orthTk} (see also Remark~\ref{rem:orthcomparison1}).
\end{proof}
We are now ready to close the energy estimates for $\varepsilon$ and $P$ and deduce the resulting pointwise estimates. Specifically, we can now prove Proposition~\ref{prop:bootstrapphi1}.
\begin{proof}[Proof of Proposition~\ref{prop:bootstrapphi1}]
The main ingredients of the proof are are the energy estimates in Lemma~\ref{lem:rpmult1} and Proposition~\ref{prop:EILEDfinal1}. As in the proof of Proposition~\ref{prop:EILEDfinal1} we derive the energy estimates for $\varepsilon$ and $P$ in tandem, by adding the corresponding estimates to absorb the errors. In this process we use Proposition~\ref{prop:bootstrappar1} and Lemma~\ref{lem:orthTk} to handle the contribution of the parameters. For a large constant $\tilR$ and $p\in[0,2]$, let
\begin{align*}
\begin{split}
&\scE_{k,j}^p(\tau):=\sum_{\uppsi=\varepsilon,P}\int_{\Sigma_\tau}\chi_{\leq\tilR}|\partial\partial^k\RbfT^j\uppsi|^2\ud V+\sum_{\uppsi=\varepsilon,P}\int_{\Sigma_\uptau}\chi_{\geq \tilR}((L+\frac{1}{r})X^k\RbfT^j\uppsi)^2r^p\ud V,\\
&\scB_{k,j}^p(\tau):=\sum_{\uppsi=\varepsilon,P}\int_{\Sigma_\tau}\chi_{\leq\tilR}|\partial\partial^k\RbfT^j\uppsi|^2\ud V\\
&\phantom{\scB_{k,j}^p(\tau):=}+\sum_{\uppsi=\varepsilon,P}\int_{\Sigma_\uptau}\chi_{\geq \tilR}\big[((L+\frac{1}{r})X^k\RbfT^j\uppsi)^2+r^{-p-\alpha}(TX^k\RbfT^j\uppsi)^2\\
&\phantom{\scE_{k,j}^p(\tau):=+\sum_{\uppsi=\varepsilon,P}\int_{\Sigma_\uptau}\chi_{\geq \tilR}\big[}+(2-p)((r^{-1}X^k\RbfT^j\uppsi)^2+(r^{-1}\Omega X^k\RbfT^j\uppsi)^2)\big]r^{p-1}\ud V.
\end{split}
\end{align*}
Here we consider the range $k+3j\leq M_\varepsilon-2$ in Proposition~\ref{prop:bootstrapphi1}, while for $\uppsi=P$ (but not $\uppsi=\varepsilon$) we allow $k+3j\leq M_P-2$. With this understanding, we simply use $M$ instead of $M_\varepsilon$ and $M_P$ in the remainder of this proof. Applying Lemma~\ref{lem:rpmult1} for $\uppsi=\varepsilon, P$ and  adding a suitable multiple of \eqref{eq:EILEDfinal1} (also at one higher order of $\RbfT$ higher because of the degeneracy of the $LE$ norm at $\uprho=0$) for any $\tau>0$ we get
\begin{align*}
\begin{split}
\sum_{k\leq M-1 }\scE_{k,0}^2(\tau)+\sum_{k\leq M-1}\int_{0}^{\tau}\scB^2_{k,0}(\tau')\ud\tau'\lesssim \sum_{k\leq M}\scE_{k,0}^2(0)\lesssim \epsilon^2.
\end{split}
\end{align*}
Here to estimate the contribution of the right-hand sides of the equations for $P$ and $\varepsilon$, we have used the bootstrap assumptions for the nonlinearities, and Proposition~\ref{prop:EILEDfinal1} for the parameters. Note that the main source term in the equation for $\varepsilon$ has spatial decay $\jap{\uprho}^{-3}$ which allows us to estimate it on the right-hand side of \eqref{eq:rpmult1} with $p=2$. It follows that for a sequence of dyadic times $(\tau_m)$, and for $k\leq M-1$ we have $\scB_{k,0}^2(\tau_m)\lesssim \epsilon^2\tau_m^{-1}$. Since $\scE^1_{k,0}\lesssim \scB_{k,0}^2$, we can again apply Lemma~\ref{lem:rpmult1}, this time on $[\tau_{m-1},\tau_m]$ and argue as above to conclude that
\begin{align*}
\begin{split}
\sum_{k\leq M-2 }\scE_{k,0}^1(\tau_m)+\sum_{k\leq M-2}\int_{\tau_{m-1}}^{\tau_m}\scB^1_{k,0}(\tau)\ud\tau\lesssim \epsilon^2\tau_m^{-1}.
\end{split}
\end{align*}
It follows that for a possibly different dyadic sequence $(\tau_m)$ we have $$\|\varepsilon(\tau_m)\|_{E_j}+\|P(\tau_m)\|_{E_j}\lesssim \scB_{j,0}^1(\tau_m)\lesssim \epsilon^2\tau_m^{-2},\qquad j\leq M-2.$$
Another application of the energy estimate \eqref{eq:EILEDfinal1} proves \eqref{eq:Tjphienergy1} and~\eqref{eq:TjPenergy1} for $j=0$, and  \eqref{eq:Tjphienergy2} and~\eqref{eq:TjPenergy2} for $j=0$ follow from another application of Lemma~\ref{lem:rpmult1} with $p=1$ on the dyadic interval $[\tau_{m-1},\tau_m]$. To prove \eqref{eq:Tjphienergy1}--\eqref{eq:TjPenergy2} with $j=1$, we first observe that in view of equation~\eqref{eq:higher1} and the estimates we have already established, for $k\leq M-3$
\begin{align*}
\begin{split}
\sum_{i\leq k} \scE_{i,1}^2(\tau)\lesssim \epsilon^2\tau^{-2}+\sum_{i\leq k+1}(\|\RbfT\varepsilon(\tau)\|_{E_i}+\|\RbfT P(\tau)\|_{E_i})\lesssim \epsilon^2\tau^{-2}.
\end{split}
\end{align*}
We can now apply Lemma~\ref{lem:rpmult1} on a dyadic intervals $[\tau_{m-1},\tau_m]$, and use the estimate above to get
\begin{align*}
\begin{split}
\sum_{k\leq M-4}\scE_{k,1}^2(\tau_m)+\sum_{k\leq M-4}\int_{\tau_{m-1}}^{\tau_m}\scB^2_{k,1}(\tau)\ud\tau\lesssim \sum_{k\leq M-3}\scE_{k,1}^2(\tau_{m-1})\lesssim \epsilon^2\tau_{m}^{-2}.
\end{split}
\end{align*}
Starting from this estimate to run the same argument as above we arrive at \eqref{eq:Tjphienergy1}--\eqref{eq:TjPenergy2} with $j=1$. The case  $j=2$ follows by repeating this procedure one more time. 

Estimates \eqref{eq:dphiL21} and~\eqref{eq:dPL21} now follow from \eqref{eq:Tjphienergy1}--\eqref{eq:TjPenergy2} and the elliptic estimates from Proposition~\ref{prop:uH}. Here the difference between $\calP$ and $\uH$ is treated in the same way as in Lemma 8.12 and Corollary 8.13 in \cite{LuOS1}. We refer the reader to their proofs for more details. Note that the contribution of $\angles{\phi}{Z_i}$ in Proposition~\ref{prop:uH} for $\phi=\RbfT^j\varepsilon,\RbfT^jP$ are treated by the orthogonality conditions as in the proof of Proposition~\ref{prop:bootstrappar1} and Lemma~\ref{lem:orthTk}. Estimates \eqref{eq:varepptwiserp1}, \eqref{eq:Pptwiserp1}, \eqref{eq:varepptwiserp2}, \eqref{eq:Pptwiserp2} now follow from what has already been established and the Gagliardo-Nirenberg estimate
\begin{align*}
\begin{split}
\|\uppsi\|_{L^\infty(\Sigma_\tau)}\lesssim \|\partial^2_\Sigma\uppsi\|_{L^2(\Sigma_\tau)}^{\frac{1}{2}}\|\partial_\Sigma\uppsi\|_{L^2(\Sigma)}^{\frac{1}{2}}.
\end{split}
\end{align*}
Finally, \eqref{eq:varepptwiserp3} and \eqref{eq:Pptwiserp4} are a consequence of the energy estimates above and the Sobolev estimate on spheres. The reader is referred to \cite{Moschidis1} or the proof of \cite[Proposition~5.2]{LuOS1} for the details of this step. A similar argument is also carried out in the proof of Lemma~\ref{eq:Boxmhugens1} below.
\end{proof}
\section{Improved late time tail bounds}\label{sec:tails}
In this section we prove Proposition~\ref{prop:boostrapvareptail}. This is the only remaining bootstrap estimate, and we can use the conclusions of Propositions~\ref{prop:bootstrappar1} and~\ref{prop:bootstrapphi1}. 
\begin{proof}[Proof of Proposition~\ref{prop:boostrapvareptail}]
We prove \eqref{eq:varepptwisetail1} for $\varepsilon$ itself, and the proof for the derivatives of $\varepsilon$ follows by similar arguments. In the process we use estimates on higher order derivatives of $\varepsilon$ (from Proposition~\ref{prop:bootstrapphi1}) and this is the source of the regularity loss $C_\varepsilon$ in \eqref{eq:varepptwisetail1}. We do not keep track of the numerical value of this constant. Our starting point is equation~\eqref{eq:varepeq1} for $\varepsilon$. We start by using Lemma~\ref{eq:Boxmhugens1} to prove improved decay bounds in the exterior. We will then use Proposition~\ref{prop:uH} to obtain the desired interior bounds. Let $\chi_{\geq \uprho^\ast}$ be a cutoff to the region $\{\uprho\geq \uprho^\ast\}$ where $\uprho^\ast$ is a sufficiently large constant (the region $\{\uprho\leq -\uprho^\ast\}$ is treated similarly). Then $\chi_{\geq \uprho^\ast}\varepsilon$ satisfies 
\begin{equation}\label{eq:varepimpexttemp1}
\Box_m (\chi_{\geq \uprho^\ast}\varepsilon)=(\Box_m-\calP)(\chi_{\geq \uprho^\ast}\varepsilon)+[\calP,\chi_{\geq\uprho^\ast}]\varepsilon+\chi_{\geq\uprho^\ast}\big(\calP\phi-g_\Ext-\txtg_1-\partial_{\uptau}\txtg_2-(\calP-\calP_h)P\big).
\end{equation}
Here, since $\chi_{\geq\uprho^\ast}\varepsilon$ is supported in $\{\uprho\geq \uprho^\ast\}$, we have used the global coordinates to identify $\chi_{\geq\uprho^\ast}\varepsilon$ with a function defined on $\bbR^{1+3}$, and the Minkowski wave operator $\Box_m$ is well-defined when applied to $\chi_{\geq\uprho^\ast}\varepsilon$. 
 In view of the spatial decay of $\calP-\Box_m$, the space and time decay of $\calP-\calP_h$, and the (already established) estimates \eqref{eq:varepptwiserp1}, \eqref{eq:Pptwiserp1}, and \eqref{eq:a+1}--\eqref{eq:wp2}, an application of Lemma~\ref{eq:Boxmhugens1} (specifically \eqref{eq:intchardecay3} and \eqref{eq:intchardecay4} with $\betabar=\frac{1}{4}-\frac{\kappa}{2}$) gives
\begin{align}\label{eq:varepimpexttemp2}
\begin{split}
\chi_{\geq\uprho^\ast}\varepsilon\lesssim \jap{\uprho}^{-1}\jap{\uptau}^{-\frac{5}{4}+\frac{\kappa}{2}}.
\end{split}
\end{align}
Here to estimate $\txtg_1$ and $\partial_\uptau\txtg_2$ we have used \eqref{eq:txtgtemp4}, \eqref{eq:txtgtemp3}--\eqref{eq:giorthdef1}. Similarly, differentiating \eqref{eq:varepimpexttemp1} with respect to $\uptau$ and using Lemma~\ref{eq:Boxmhugens1} (specifically \eqref{eq:intchardecay1} and \eqref{eq:intchardecay2}) we get
\begin{align}\label{eq:varepimpexttemp3}
\begin{split}
\chi_{\geq\uprho^\ast}\RbfT\varepsilon\lesssim \jap{\uprho}^{-1}\jap{\uptau}^{-\frac{9}{4}+\frac{\kappa}{2}}.
\end{split}
\end{align}
Note that \eqref{eq:varepimpexttemp2} already proves \eqref{eq:varepptwisetail1} in the region $\{\uprho\geq \uptau\}$, so it remains to consider $\chi_{\leq \uptau}\varepsilon$, where $\chi_{\leq \uptau}$ is a cutoff to the region $\{|\uprho|\leq \uptau\}$. For this we observe that (recall the notation from \eqref{eq:calPdef1} and \eqref{eq:calPhstat1})
\begin{align*}
\begin{split}
\calP_h^\stat(\chi_{\leq \uptau}\varepsilon)=[\calP,\chi_{\leq \uptau}]\varepsilon-(\calP-\calP_h^\stat)(\chi_{\leq \uptau}\varepsilon)+\chi_{\leq\uptau}\calP\varepsilon,
\end{split}
\end{align*}
where for the last term $\calP\varepsilon$ is given by \eqref{eq:varepeq1}. We now apply Proposition~\ref{prop:uH} on a fixed slice $\Sigma_\uptau$, where we recall from \eqref{eq:tilDelta1} that by a change of coordinate using the global geometric coordinates restricted to $\Sigma_\uptau$, the operator $\calP_h^\stat$ can be transformed into $\uH$. Note that by the Gagliardo-Nirenberg inequality, and with $\chi_k$ a partition of unity subordinate to $A_k$ in Definition~\ref{def:w-Sob}, 
\begin{align*}
\begin{split}
\|f\|_{L^\infty}\lesssim \sup_{k\geq0}\big( \|\chi_kf\|_{L^2}^{\frac{1}{4}}\|\partial^2(\chi_kf)\|_{L^2}^{\frac{3}{4}}\big)\lesssim \|f\|_{\ell^\infty\calH^{2,-\frac{3}{2}}}.
\end{split}
\end{align*}
Therefore, our task is reduced to proving that
\begin{align}\label{eq:varepimpexttemp4}
\begin{split}
&R_\ctf^2\nrm{[\calP,\chi_{\leq \uptau}]\varepsilon-(\calP-\calP_h^\stat)(\chi_{\leq \uptau}\varepsilon)+\chi_{\leq\uptau}\calP\varepsilon-\chi_{\leq R+C}\calG_0}_{\ell^{1} \calH^{0, \frac{1}{2}}} + \sum_{i=1,2,3} \abs{\brk{\varepsilon, \uZ_{i}}}+R_\ctf^{\frac{1}{2}}\abs{\brk{\varepsilon, \uZ_{\mu}}}\\
&+\|\chi_{\leq R+C}\calG_0\|_{\ell^{1} \calH^{0, \frac{1}{2}}}+R_\ctf^2\|\Dar\calS_1\calG_0\|_{\ell^{1} \calH^{0, \frac{1}{2}}}
\end{split}
\end{align}
is bounded by the right-hand side of \eqref{eq:varepptwisetail1}. Here $\calG_0$ denotes the source term as in Lemma~\ref{lem:calG01}. 
For $\sum_{i} \abs{\brk{\varepsilon, \uZ_{i}}}$, this follows as usual by the orthogonality conditions and the bootstrap assumption \eqref{eq:varepptwisetailb1}, where we use largeness of $R_\ctf$ and smallness of $\ell$ to get a constant that is independent of the one in \eqref{eq:varepptwisetailb1}. For instance, note that by \eqref{eq:varepptwisetailb1} and the $\uprho^{-2}$ spatial decay of the zero eigenfunctions (as usual the more slowly decaying parts come with factors of $\ell$)
\begin{align*}
\begin{split}
\int_{\Sigma_{\uptau}} |\varepsilon| |\partial^2 \chi_{\{|\uprho|\leq R_\ctf\}}| |\fy_i^\stat| \ud V_{\Sigma_\uptau}\lesssim R_\ctf^{-1} (R_\ctf^2 R^2)\uptau^{-\frac{9}{4}+\frac{\kappa}{2}}.
\end{split}
\end{align*}
For $\fy_\mu^\stat$ the exponential decay of the eigenfunction gives an improved factor of $R_\ctf$ that compensates for the large factor $R_\ctf^{\frac{1}{2}}$. For the second line in \eqref{eq:varepimpexttemp4} the estimate follows from Lemma~\ref{lem:calG01} by a similar argument as in \eqref{eq:Penergynonlintemp1} and \eqref{eq:Dartxtg2temp1}, and using the smallness $R^{\frac{\alpha}{2}} R_\ctf^{-1+\frac{\alpha}{2}},R_\ctf^{2+\frac{\alpha}{2}}R^{-1+\frac{\alpha}{2}}\ll 1$. Indeed, here we use estimate \eqref{eq:wpb1} on the parameters which already contains a small factor, and the relation above among $R$ and $R_\ctf$ ensures that this smallness does not get compensated by another large factor. It remains to estimate the first norm in \eqref{eq:varepimpexttemp4}, for which we consider a few representative terms. For $\chi_{\leq\uptau}\calP\varepsilon-\chi_{\leq R+C}\calG_0$ we recall equation~\eqref{eq:varepeq1}, where the main contribution is from the first three terms on the right-hand side. For $\txtg_1$ and $\txtg_2$ we notice that these are supported on a single dyadic region, so the $\ell^1$ sum in the $\ell^{1} \calH^{s, -\frac{n}{2}+2}$ norm does not play a role. The corresponding contributions are then bounded using  \eqref{eq:txtgtemp3}--\eqref{eq:WYdef1} (note that in this context we do not  use \eqref{eq:Iic1},~\eqref{eq:tilIIic1},~\eqref{eq:c2prime2}). Here for $\txtg_1$ we also use the decay coming from $\calP_h Y_i$, which comes from the support of $Y_i$. For $\partial_\uptau\txtg_2$ we can use the estimates \eqref{eq:TjPenergy1} and \eqref{eq:dPL21} to estimate the contribution of $(h^{-1})^{00}(Y_i+\frac{W_i}{(h^{-1})^{0\nu}})\partial^2_{\tau\nu} P$. That is,
\begin{align*}
\begin{split}
\int_{\{R_\ctf\leq|\uprho|\leq \uptau\}} |\uprho|^{-2}\partial_j \partial_\uptau P\sqrt{|h|}\ud x \lesssim \|\chi_{|\uprho|\geq R_\ctf} |\uprho|^{-\frac{3-\kappa}{2}}\RbfT P\|_{L^2(\Sigma_\uptau)}\lesssim \uptau^{-\frac{5}{2}+\kappa}.
\end{split}
\end{align*}
and
\begin{align*}
\begin{split}
\int_{\{R_\ctf\leq|\uprho|\leq \uptau\}} |\uprho|^{-2} \partial_\uptau^2 P\sqrt{|h|}\ud x \lesssim \uptau^{\frac{1}{2}}\|\chi_{|\uprho|\geq R_\ctf} |\uprho|^{-1}\RbfT^2 P\|_{L^2(\Sigma_\uptau)}\lesssim \uptau^{-\frac{5}{2}}.
\end{split}
\end{align*}
For $\calP\phi-g_\Ext$ on the right-hand side of \eqref{eq:varepeq1}, the main contribution is from the source terms that depend linearly on the parameter derivatives so we concentrate only on these. In the exterior (that is, for $|\uprho|\geq R+C$), the corresponding source terms have spatial decay $\jap{\uprho}^{-4}$, and their $\ell^{1} \calH^{s, -\frac{n}{2}+2}$ norm can be bounded using \eqref{eq:a+1}--\eqref{eq:wp2}. 
Returning to \eqref{eq:varepimpexttemp4}, we consider the contribution of $(\calP-\calP_h^\stat)(\chi_{\leq \uptau}\varepsilon)$. The main contribution is of the forms $\chi_{\leq \uptau}\partial\RbfT\varepsilon$, for which we use the estimate
\begin{align*}
\begin{split}
|\chi_{\leq \uptau}\partial\RbfT\varepsilon|\lesssim \chi_{\leq \uptau} \uptau^{-\frac{9}{4}+\frac{\kappa}{2}}\jap{r}^{-2}.
\end{split}
\end{align*}
This follows from \eqref{eq:varepimpexttemp3} (with higher angular derivatives commuted) and the equation for $\varepsilon$. We conclude that
\begin{align*}
\begin{split}
\nrm{(\calP-\calP_0)(\chi_{\leq \uptau}\varepsilon)}_{\ell^{1} \calH^{0, \frac{1}{2}}}\lesssim R^2\uptau^{-\frac{9}{4}+\frac{\kappa}{2}}\log\uptau\lesssim R^2\uptau^{-\frac{9}{4}+\kappa}.
\end{split}
\end{align*}
with a constant that is independent of \eqref{eq:varepptwisetailb1}. 
The estimate for the remaining term, $[\calP,\chi_{\leq \uptau}]\varepsilon$, in \eqref{eq:varepimpexttemp4} is similar. Indeed, using \eqref{eq:varepimpexttemp2} and \eqref{eq:varepimpexttemp3} we have
\begin{align*}
\begin{split}
|[\calP,\chi_{\leq \uptau}]\varepsilon|\lesssim \uptau^{-\frac{9}{4}+\frac{\kappa}{2}-2}\chi_{\{|\uprho|\simeq \uptau\}},
\end{split}
\end{align*}
which implies that
\begin{align*}
\begin{split}
\|[\calP,\chi_{\leq \uptau}]\varepsilon\|_{\ell^{1} \calH^{0, \frac{1}{2}}}\lesssim \uptau^{-\frac{9}{4}+\frac{\kappa}{2}}
\end{split}
\end{align*}
with a constant that is independent of \eqref{eq:varepptwisetailb1}, completing the proof.
\end{proof}
\appendix
\section{Minkowski computations}\label{sec:Appendix}
In this section we collect a number of computations on the Minkowski space $(\bbR^{1+3},m)$. In view of the asymptotic flatness of the catenoid metric, these computations are the basis of most of the exterior calculations. Unless otherwise specified, the estimates in this appendix are stated under the bootstrap assumptions \eqref{eq:a+trap}--\eqref{eq:varepptwisetailb1} (in particular \eqref{eq:wpb1} as the parameters enter in the definitions of various operators).
\subsection{Expressions for $\Box_m$ and vectorfields}\label{subsec:appMink}
We start by writing the expression for $m$ in the exterior polar coordinates $(\tau,r,\theta)$ (see \eqref{eq:exttaurcoords1}):
\begin{align}
m&=-\gamma^{-2}\ud\tau\otimes \ud\tau-(1-\jap{r}^{-2}+\Theta\cdot\dotxi)(\ud\tau\otimes\ud r +\ud r\otimes \ud \tau)+r\Theta_a\cdot\dotxi(\ud\tau\otimes \ud\theta^a+\ud \theta^a\otimes \ud\tau)\nonumber\\
&\quad +\jap{r}^{-2} \ud r\otimes \ud r + r^2\ringsg_{ab}\ud\theta^a\otimes \ud \theta^b,\label{eq:mform1}
\end{align}
and
\begin{align}\label{eq:mdvol1}
\begin{split}
| m|^{\frac{1}{2}} = (1-\Theta\cdot\dotxi)r^{n-1}|\ringsg|^{\frac{1}{2}}(1+\jap{r}^{-2}\frac{\gamma^{-2}+(1-(\Theta\cdot\dotxi)^2)}{(1-\Theta\cdot\dotxi)^2}).
\end{split}
\end{align}
In these expressions we have used the notation $\gamma=(1-|\dotxi|^2)^{-\frac{1}{2}}$. The inverse $m^{-1}$ can be calculated as
\begin{align}\label{eq:minvdecomp1}
\begin{split}
m^{-1}=m_0^{-1}+\jap{r}^{-2}m_1,
\end{split}
\end{align}
where 
\begin{align}\label{eq:m02inv1}
\begin{split}
m_0^{-1}&=\frac{-1}{1-\Theta\cdot\dotxi}(\partial_\tau\otimes\partial_r+\partial_r\otimes\partial_\tau)+\frac{1+\Theta\cdot\dotxi}{1-\Theta\cdot\dotxi}\partial_r\otimes\partial_r+\frac{\Theta^a\cdot\dotxi}{r(1-\Theta\cdot\dotxi)}(\partial_r\otimes\partial_{\theta^a}+\partial_{\theta^a}\otimes\partial_r)\\
&\quad+r^{-2}(\ringsg^{-1})^{ab}\partial_{\theta^a}\otimes\partial_{\theta^b},
\end{split}
\end{align}
and $m_1$ is a matrix of size $\callO(1)$. The expression for the wave operator $\Box_m$ is
\begin{align}\label{eq:Boxm1}
\begin{split}
\Box_m\psi&=2(m_0^{-1})^{\tau r}\partial^2_{\tau r}\psi+\frac{n-1}{r}(m_0^{-1})^{\tau r}\partial_\tau\psi+2(m_0^{-1})^{\theta r}\partial^2_{\theta r}\psi\\
&\quad + (m_0^{-1})^{rr}\partial^2_r\psi+\frac{n-1}{r}(m_0^{-1})^{rr}\partial_r\psi+\frac{1}{\sqrt{|m_0|}}\partial_\theta(\sqrt{|m_0|}(m_0^{-1})^{\theta r})\partial_r\psi\\
&\quad
+\frac{1}{r^2}\sDelta_{\bbS^{n-1}}\psi+\frac{1}{r}(m_0^{-1})^{\theta r}\partial_\theta\psi+\frac{\Theta_\theta\cdot\dotxi}{r^2(1+\Theta\cdot\dotxi)}\partial_\theta\psi+\Err_{\Box_m}(\psi),
\end{split}
\end{align}
where we have suppressed the index $a$ in $\theta^a$ and
\begin{align}\label{eq:ErrBoxm1}
\begin{split}
\Err_{\Box_m}\psi&= \jap{r}^{-2}m_1^{\mu\nu}\partial^2_{\mu\nu}\psi+(\callO(r^{-2}|\dotwp|)+\callO(r^{-3}))\partial_\tau\psi+(\callO(r^{-2}|\dotwp|) +\callO(r^{-3}))\partial_r\psi\\
&\quad+\callO((r^{-3}|\dotwp|)+\callO(r^{-4}))\partial_\theta\psi.
\end{split}
\end{align}
We refer the reader to \cite[Section~4.2.2]{LuOS1} for the derivation of these formulas. Next, we want to express $\Box_m$ in terms of geometric vectorfields $\Omega$, $L$, $\Lbar$ and  the geometric radial function $\tilr$. To define these let $(x^0,x')=(\tau+\jap{r},r\Theta+\xi)$ be a point on the hyperboloidal part of $\barbsUpsigma_\tau$ and let $y=(y^0,y')$ be related to $x=(x^0,x')$ by $x=\Lambda_{-\ell}y+(-R,\xi-\tau\ell)$ (see Section~\ref{subsec:mainprofile}). 
We define $L$, $\Lbar$, $\Omega$, and $T$ as the push forward by $\Lambda_{-\ell}$ of the corresponding vectorfields in the $y$ coordinates. That is, 
\begin{equation}\label{eq:VFdef0}
T=\Lambda_{-\ell}\partial_{y^0},\quad L= \Lambda_{-\ell}(\partial_{y^0}+\frac{y^i}{|y'|}\partial_{y^i}),\quad \Lbar = \Lambda_{-\ell}(\partial_{y^0}-\frac{y^i}{|y'|}\partial_{y^i}),\quad \Omega_{jk}=\Lambda_{-\ell}(y^j\partial_{y^k}-y^k\partial_{y^j}).
\end{equation}
Then it can be seen that (see \cite[Section~4.1]{LuOS1})
\begin{equation}\label{eq:VFexpansion1}
\begin{cases}
&T=\callO(1)\partial_\tau+\callO(\dotwp)\partial_r+\callO(\dotwp r^{-1})\partial_a,\\
&L=\callO(r^{-2})\partial_\tau+\callO(1)\partial_r+\callO(r^{-3})\partial_a,\\
&\Lbar = \callO(1)\partial_\tau+\callO(1)\partial_r+\callO(\dotwp r^{-1}+r^{-3})\partial_a,\\
&\Omega_{jk}=\callO(\wp r)\partial_r+\callO(1) \partial_a.
\end{cases},~~
\begin{cases}
&\partial_\tau=\callO(1)L+\callO(1)\Lbar,\\
&\partial_r= \callO(1)L+\callO(r^{-2})\Lbar+\callO( r^{-3})\Omega,\\
&\partial_\theta=\callO(\wp r)L+\callO(1)\Omega+\callO(\wp r^{-1})\Lbar.
\end{cases}
\end{equation}
Similarly, the geometric radial function is defined in terms of the $y$ variables as $\tilr=|y'|$ and satisfies $\tilr =|rA_\ell\Theta-\gamma \jap{r}\ell|=\gamma(1-\Theta\cdot\ell)r+\callO(r^{-1})\simeq r$ and $L\tilr = 1+\calO(r^{-1}\dotwp)$. As proved in \cite[Lemma~8.2]{LuOS1}, and, with $\tilupsi:=\tilr\uppsi$,
\begin{align}\label{eq:BoxVF2}
\begin{split}
\tilr\Box_m\uppsi&= -\Lbar L \tilupsi+\frac{1}{\tilr^2}\sum_{\Omega}\Omega^2\tilupsi+\callO(\dotwp^{\leq 2})L\tilupsi+\callO(\dotwp^{\leq 2} r^{-2})\Lbar\tilupsi+\callO(\dotwp^{\leq 2} r^{-2})\Omega\tilupsi +\callO(\dotwp^{\leq 2}r^{-2})\tilupsi.
\end{split}
\end{align}
Recalling the definition of $\callP_g$ from \eqref{eq:callP1}--\eqref{eq:callP2},  and defining $\tilcalP:=-\Lbar L+\frac{1}{\tilr^2}\sum\Omega^2$ and $\tilcalP_1=LL+\frac{1}{\tilr^2}\sum\Omega^2$, for  any integers $k_1,k_2,k_3\geq0$, and with $k=k_1+k_2+k_3$, we get (similarly one can replace $\callP_g$ by $\Box_m$)
\begin{equation}\label{eq:higher1}
(\tilr L+1)^{k_1}\Omega^{k_2}T^{k_3}(\tilr^{\frac{n-1}{2}}\callP_g \uppsi)=\tilcalP((\tilr L)^{k_1}\Omega^{k_2}T^{k_3}\tilupsi)+\Err_{k_1,k_2,k_3}[\tilupsi],\\
\end{equation}
where
\begin{align}
\Err_{k_1,k_2,k_3}[\tilupsi]&=\sum_{j=0}^{k_1-1}c_{j,k_1}\tilcalP_1((\tilr L)^{j}\Omega^{k_2}T^{k_3}\tilupsi)+\callO(r^{-4})L^2(\tilr L)^{k_1}\Omega^{k_2}T^{k_3}\tilupsi+\callO(r^{-4})\Lbar^2(\tilr L)^{k_1}\Omega^{k_2}T^{k_3}\tilupsi\nonumber\\
&\quad+\callO(r^{-5})\Omega L(\tilr L)^{k_1}\Omega^{k_2}T^{k_3}\tilupsi+\callO(r^{-5})\Lbar\Omega(\tilr L)^{k_1}\Omega^{k_2}T^{k_3}\tilupsi\nonumber\\
&\quad+\callO(\dotwp^{\leq 2k+2}+r^{-5})\sum_{j_1+j_2+j_3\leq k}L(\tilr L)^{j_1}\Omega^{j_2}T^{j_3}\tilupsi\nonumber\\
&\quad+\callO(\dotwp^{\leq 2k+2} r^{-2}+r^{-4})\sum_{j_1+j_2+j_3\leq k}\Lbar(\tilr L)^{j_1}\Omega^{j_2}T^{j_3}\tilupsi\label{eq:errorhigher1}\\
&\quad+\callO(\dotwp^{\leq 2k+2} r^{-2}+r^{-6})\sum_{j_1+j_2+j_3\leq k}\Omega(\tilr L)^{j_1}\Omega^{j_2}T^{j_3}\tilupsi\nonumber\\
&\quad+\callO(\dotwp^{\leq 2k+2} r^{-2}+r^{-5})\sum_{j_1+j_2+j_3\leq k}(\tilr L)^{j_1}\Omega^{j_2}T^{j_3}\tilupsi,\nonumber
\end{align}
for some constants $c_{j,k_1}$ (which are nonzero only if $k_1\geq 1$), with $c_{k_1-1,k_1}=-k_1$. See \cite[Lemma~8.4]{LuOS1} for a proof. The commutators among the vectorfields $\{\Omega,L,\Lbar,T\}$ are computed in \cite[Lemma~8.1]{LuOS1} as
\begin{align}
&[\Lbar,L]= \callO(\dotwp^{\leq 2})L+\callO(\dotwp^{\leq 2}r^{-2})\Lbar+\callO(\dotwp^{\leq 2} r^{-2})\Omega,\quad [\Omega_{ij},\Omega_{k\ell}]=\delta_{[ki}\Omega_{\ell j]},\nonumber\\
&[\Lbar,\Omega]= \callO(\dotwp)\Omega+\callO(\dotwp r) L+\callO(\dotwp)\Lbar,\quad [T,L]=-[T,\Lbar]= \callO(\dotwp^{\leq 2})L+\callO(\dotwp^{\leq 2}r^{-2})\Lbar+\callO(\dotwp^{\leq 2} r^{-2})\Omega,\nonumber\\
&[L,\Omega]=\callO(\dotwp r^{-1})L+\callO(\dotwp r^{-2})\Omega+\callO(\dotwp r^{-2})\Lbar,\quad [T,\Omega]=\callO(\dotwp)\Omega+\callO(\dotwp r) L+\callO(\dotwp)\Lbar.\label{eq:VFcomm1}
\end{align}
We also record the following relation (which is a consequence of \eqref{eq:VFcomm1}) from \cite[Lemma~8.3]{LuOS1} which allows us to change the order of commuted vectorfields. If $X_1,\dots,X_k\in\{\tilr L,\Omega,T\}$ are $k$ vectorfields with $k_1$ factors of $\tilr L$, $k_2$ factors of $\Omega$ and $k_3$ factors of $T$, then for any function $\uppsi$,
\begin{align}\label{eq:vecorder1}
\begin{split}
X_k\dots X_1\uppsi= (\tilr L)^{k_1}\Omega^{k_2} T^{k_3}\uppsi+\sum_{j_1+j_2+j_3\leq k-1}\callO(\dotwp^{\leq 2k-2(j_1+j_2+j_3)})(\tilr L)^{j_1}\Omega^{j_2}T^{j_3}\uppsi.
\end{split}
\end{align}
\subsection{Multiplier identities}\label{seq:apprp}
To state the main $r^p$ multiplier identity, we start by defining the relevant energies. For simplicity of notation we concentrate on one asymptotic end, where $r>0$, but similar relations hold at the other end. Fix $\chi_{\geq \tilR}$ to be a cutoff supported in the region $\{r\geq \tilR\}\subseteq \calC_\hyp$ and let $\chi_{\leq \tilR}=1-\chi_{\geq \tilR}$. In general, we use the notation $\tilupsi=\tilr\uppsi$ for any function $\uppsi$. If $X_1,\dots,X_k$, $k=k_1+k_2+k_3$, are a collection of vectorfields from $\{\tilr L, \Omega, T\}$, with $X_1,\dots X_{k_3}=T$, $X_{k_3+1},\dots X_{k_3+k_2}=\Omega$, and~$X_{k_3+k_2+1},\dots X_k=\tilr L$, we let $\tilupsi_k=X_k\dots X_1\tilupsi$, and if the precise choice of the vectorfields is important we write
$\tilupsi_k=\tilupsi_{k_1,k_2,k_2}$.
For $p\in[0,2]$, the basic $r^p$ boundary and bulk energies are defined as 
\begin{align}\label{eq:rpenergiesdef1}
\begin{split}
&\calE^p_k(\tau)\equiv\calE_k^p[\uppsi](\tau):=\int_{\Sigma_\tau} \chi_{\geq \tilR} \tilr^p (L\tilupsi_k)^2 \ud \theta \ud r,\\
&\calB_{k}^{p}(\tau_1,\tau_2)\equiv\calB_{k}^{p}[\uppsi](\tau_1,\tau_2):=\int_{\tau_1}^{\tau_2}\int_{\Sigma_\tau}\chi_{\geq \tilR}\tilr^{p-1}\big((L\tilupsi_k)^2+\big(\frac{2-p}{\tilr^2}\big)(|\Omega\tilupsi_k|^2+\tilupsi_k^2)\big)\ud S \ud r \ud\tau.
\end{split}
\end{align}
Here $\ud S$ denotes the standard volume form on $\bbS^2$. When there is a need to distinguish between the vectorfields applied to $\tilupsi$ we write $\calE_{k_1,k_2,k_3}^p$ and $\calB_{k_1,k_2,k_3}^p$ for the corresponding energies. We also define the unweighted energy
\begin{align*}
\begin{split}
E_k(\tau)&\equiv E_k[\uppsi](\tau)\\
&:=\int_{\Sigma_\tau}\chi_{\leq \tilR}|\partial\partial^k\uppsi|^2\ud V+\int_{\Sigma_\tau}\chi_{\geq \tilR}(|\partial_\Sigma X^k \uppsi|^2+r^{-2}|TX^k\uppsi|^2+r^{-2}|X^k\uppsi|^2) \ud V.
\end{split}
\end{align*}
The following lemma, whose proof is contained in \cite[Lemma~8.7]{LuOS1} contains the $r^p$ multiplier identity. Note that while the lemma is stated for the operator $\callP_g$, the same result holds with $\callP_g$ replaced by $\Box_m$.
\begin{lemma}\label{lem:rpmult1}
Suppose $\callP_g \uppsi=f$ in $\{r\geq \tilR\}$ and let  $k=k_1+k_2+k_3$ and $p\in [0,2]$. Then for any $\tau_1<\tau_2$, 
\begin{align}\label{eq:rpmult1}
\begin{split}
&\sum_{j\leq k_1}\big(\sup_{\tau\in[\tau_1,\tau_2]}\calE_{j,k_2,k_3}^p(\tau)+\calB_{j,k_2,k_3}^{p-1}(\tau_1,\tau_2)\big)\\
&\leq C\sum_{j\leq k_1}\calE_{j,k_2,k_3}^{p}(\tau_1)+C\sum_{j\leq k_1}\int_{\tau_1}^{\tau_2}\int_{\Sigma_\tau}\chi_{\geq \tilR}\tilr^p \tilf_k (L\tilupsi_{j,k_2,k_3} +\callO(r^{-5})\Omega\tilupsi_{j,k_2,k_3})\ud S \ud r \ud\tau\\
&\quad+C_{\tilR}\int_{\tau_1}^{\tau_2}\int_{\Sigma_\tau}|\partial\chi_{\geq \tilR}|(|\partial_{\tau,x}\tilupsi_k|^2+|\tilupsi_k|^2)\ud S \ud r \ud \tau\\
&\quad+C\sum_{j\leq k}\sup_{\tau\in[\tau_1,\tau_2]}E_j(\tau)+C\delta \sup_{\tau\in[\tau_1,\tau_2]}\sum_{j\leq k}\calE_j^p(\tau)\\
&\quad+C\delta\sum_{j\leq k}\int_{\tau_1}^{\tau_2}\int_{\Sigma_\tau}\chi_{\geq \tilR}(|\partial_{\tau,x}\tilupsi_j|^2+\tilr^{-2}|\tilupsi_j|^2)\tilr^{-1-\alpha}\ud S \ud r \ud\tau.
\end{split}
\end{align}
Here $C$ and $C_{\tilR}$ are large constants constants and $\delta=o(\epsilon)+o(\tilR)$ is a small constant that is independent of $C$ and $C_{\tilR}$.
\end{lemma}
\subsection{Decay estimates for $\Box_m\uppsi=f$}\label{sec:Appendixtail}
In this section we use Lemma~\ref{lem:rpmult1} and the sharp Huygens principle to prove late time tail decay estimates for the forward solution\footnote{By this, we mean the solution vanishes for sufficiently negative $\tau$.} to $\Box_m\uppsi=f$ in terms of decay properties of $f$. These estimates will be used in the proof of the improved decay estimate \eqref{eq:varepptwisetail1} for $\varepsilon$ in Section~\ref{sec:tails}. In our applications the functions of interest will always be localized to the exterior by means of a cutoff, and using global coordinates we will always view the equation $\Box_m\uppsi=f$ \emph{globally} as an equation on $\bbR^{1+3}$.  In this context, we view the coordinates $(\tau,r,\theta)$ as defined globally by \eqref{eq:exttaurcoords1}, and the vectorfields $\{\Lbar,L,T,\Omega\}$ to be defined globally by \eqref{eq:VFdef0}--\eqref{eq:VFexpansion1}. We use $\Sigma_\tau$ to denote the constant $\tau$ hypersurfaces and $\partial_\Sigma$ to denote size one derivatives that are tangential to $\Sigma_\tau$. The desired decay estimates are provided in the following lemma, where we continue to use the notation from Section~\ref{seq:apprp}.
\begin{lemma}\label{eq:Boxmhugens1}
Let $\uppsi$ be the forward solution to  $\Box_m\uppsi=f$, and suppose $f$ is supported in $\{\tau\geq0\}$ and satisfies the following estimates
\begin{align*}
\begin{split}
| f|+|\chi_{\{r\lesssim 1\}}\partial f|+|\chi_{\{r\gtrsim1\}}Xf|\lesssim \jap{r}^{-4}\jap{\tau}^{-2-\betabar}
\end{split}
\end{align*}
for some $\betabar\in(0,\frac{1}{2}]$ and any $X\in\{\Omega,T,rL\}$. Then $\uppsi$ satisfies the estimates
\begin{align}\label{eq:intchardecay1}
\begin{split}
|\uppsi|\lesssim \jap{\tau}^{-2-\betabar}.
\end{split}
\end{align}
If in addition
\begin{align*}
\begin{split}
|\chi_{\{r\lesssim 1\}}\partial^2 f|+|\chi_{\{r\gtrsim1\}}X^2f|\lesssim \jap{r}^{-4}\jap{\tau}^{-2-\betabar}
\end{split}
\end{align*}
then
\begin{align}\label{eq:intchardecay2}
\begin{split}
|\uppsi|\lesssim \jap{r}^{-\frac{1}{2}} \jap{\tau}^{-2-\betabar}\quad\mand \quad |\uppsi|\lesssim \jap{r}^{-1} \jap{\tau}^{-\frac{7}{4}-\frac{\betabar}{2}}.
\end{split}
\end{align}
Similarly, with $\betabar$ as above, if
\begin{align*}
\begin{split}
| f|+|\chi_{\{r\lesssim 1\}}\partial f|+|\chi_{\{r\gtrsim1\}}Xf|\lesssim \jap{r}^{-4}\jap{\tau}^{-1-\betabar}
\end{split}
\end{align*}
then
\begin{align}\label{eq:intchardecay3}
\begin{split}
|\uppsi|\lesssim \jap{\tau}^{-1-\betabar},
\end{split}
\end{align}
and if in addition
\begin{align*}
\begin{split}
|\chi_{\{r\lesssim 1\}}\partial^2 f|+|\chi_{\{r\gtrsim1\}}X^2f|\lesssim \jap{r}^{-4}\jap{\tau}^{-1-\betabar}
\end{split}
\end{align*}
then
\begin{align}\label{eq:intchardecay4}
\begin{split}
|\uppsi|\lesssim \jap{r}^{-1} \jap{\tau}^{-1-\betabar}.
\end{split}
\end{align}
\end{lemma}
\begin{proof}
The proofs of \eqref{eq:intchardecay3} and \eqref{eq:intchardecay4} are similar to those of \eqref{eq:intchardecay1} and \eqref{eq:intchardecay2}, so we only present the argument for  \eqref{eq:intchardecay1} and \eqref{eq:intchardecay2}. Fix a point $x_\ast$ with coordinates $(\tau_\ast,r_\ast,\theta_\ast)$. The key point of the argument is that by the sharp Huygens principle, to estimate $\uppsi(x_\ast)$ we only need the knowledge of $f$ in the region $\{(\tau,r,\theta)\vert \jap{r}\gtrsim \tau_\ast-\tau\}$. To see this observe that $\uppsi(x_\ast)$ is determined by the values of $f$ on the backwards light cone through $x_\ast$, that is, $x=(x^0,x')$ such that $|x_\ast'-x'|=x_\ast^0-x^0$. In the $(\tau,r,\theta)$ coordinates, this implies that
\begin{align*}
\begin{split}
\tau_\ast+\jap{r_\ast}-\tau-\jap{r}=|r_\ast\Theta(\theta_\ast)+\xi(\tau_\ast)-r\Theta(\theta)-\xi(\tau)|\leq \jap{r_\ast}+\jap{r}+|\xi(\tau_\ast)-\xi(\tau)|,
\end{split}
\end{align*}
which, since $|\dotxi|<1$, requires $\jap{r}\gtrsim \tau_\ast-\tau$. It follows that if $\chi_\ast$ is a cutoff to the region $\{\jap{r}\gtrsim \tau_\ast-\tau\}$ and $\upphi$ is the forward solution to $\Box\upphi = \chi_\ast f$, then $\upphi(x_\ast)=\uppsi(x_\ast)$, and it suffices to estimate $\upphi$. For this we go through the usual proof of decay estimates using the $r^p$ multiplier argument, and observe that the decay and support properties of $\chi_\ast f$ give decay, rather than boundedness, of the energy fluxes. The details are as follows. First we apply Lemma~\ref{lem:rpmult1} with $k=0$ and $p=2$. To absorb the contribution of the last three lines of \eqref{eq:rpmult1}, we add a multiple of the energy and ILED estimates for $\Box_m\upphi=\chi_\ast f$, which can be proved in the same way as the corresponding estimates for $\calP$, but now without any eigenfunctions or loss of derivatives due to trapping (see Section~\ref{sec:ILED}). For any $\tau\leq \tau_\ast$ this gives
\begin{align*}
\begin{split}
\calE_{0}^2[\upphi](\tau)+\calB_{0}^{1}[\upphi](0,\tau)\lesssim \|f\chi_\ast\jap{r}\|_{L^1_\tau L^2_x[0,\tau]}^2+\|f\chi_\ast\|_{LE^\ast[0,\tau]}^2,
\end{split}
\end{align*}
where we have used the shorthand notation $\|g\|_{L^1_\tau L^2_x[0,\tau]}=\|\|g(\tau')\|_{L^2(\sqrt{|m|}\ud r\ud\theta)}\|_{L^1_{\ud \tau'}[0,\tau]}$. Using the bound $|\tilf\chi_\ast\jap{r}|^2\lesssim \chi_{\{r\gtrsim \tau_\ast-\tau'\}}\jap{\tau}^{-4-2\betabar}\jap{r}^{-4}$ (recall that $\tilf=\tilr f$) we get
\begin{align}\label{eq:Boxmdecaytemp0}
\begin{split}
\calE_{0}^2[\upphi](\tau)+\calB_{0}^{1}[\upphi](0,\tau)\lesssim \tau_\ast^{-3}\lesssim \jap{\tau}^{-3}.
\end{split}
\end{align}
It follows that we can find dyadic times $\tau_n\leq \tau_\ast$ such that $\calE_{0}^1[\upphi](\tau_n)\lesssim \jap{\tau_n}^{-4}$. Applying Lemma~\ref{lem:rpmult1} again, but with $p=1$ and on $[\tau_{j-1},\tau_j]$, and adding a suitable multiple of the energy and ILED estimates we get
\begin{align*}
\begin{split}
\calE_{0}^1[\upphi](\tau_j)+\|\upphi\|_{L^2_\tau E[\tau_{j-1},\tau_{j}]}^2\lesssim \tau_j^{-4}+\|f\chi_\ast\jap{r}^{\frac{1}{2}}\|_{L^1_\tau L^2_x[\tau_{j-1},\tau_{j}]}^2+\|f\chi_\ast\|_{LE^\ast[0,\tau]}^2\lesssim \tau_j^{-4},
\end{split}
\end{align*}
where for the last estimate we have used the bound $|\tilf\chi_\ast\jap{r}^{\frac{1}{2}}|^2\lesssim \chi_{\{r\gtrsim \tau_\ast-\tau'\}}\jap{\tau}^{-4-2\betabar}\jap{r}^{-5}$, and where we have used the notation
\begin{align*}
\begin{split}
\|\upphi\|_E^2=\|\chi_{\{r\lesssim 1\}}\partial\upphi\|_{L^2_x}^2+\|\chi_{\{r\gtrsim1\}}\partial_\Sigma\upphi\|_{L_x^2}+\|\chi_{\{r\gtrsim1\}}r^{-2}\upphi\|_{L^2_x}^2
\end{split}
\end{align*}
for the energy norm. We conclude that for some, possibly different, sequence of dyadic times $\tau_n\leq \tau_\ast$, $\| \upphi(\tau_n)\|_{E}\lesssim \tau_n^{-\frac{5}{2}}$. Applying the energy estimate on $[\tau_{j-1},\tau_j]$, but  noting that by similar arguments as before we only have the weaker estimate 
\begin{align*}
\begin{split}
\|f\chi_\ast\|_{LE^\ast[0,\tau]}^2 \lesssim \tau_j^{-4-2\betabar},
\end{split}
\end{align*}
we conclude that $\| \upphi(\tau)\|_{E}\lesssim \tau^{-2-\betabar}$ for all $\tau\leq \tau_\ast$. Similarly, after commuting one derivative and using elliptic estimates (see \cite{Moschidis1} or the proof of \cite[Corollary~8.13]{LuOS1}) we get
\begin{align}\label{eq:Boxmdecaytemp1}
\begin{split}
\|\chi_{\{r\lesssim 1\}}\partial^2\upphi(\tau)\|_{L^2_x}+\|\partial_\Sigma X\upphi(\tau)\|_{L^2_x}\lesssim \tau^{-2-\betabar}.
\end{split}
\end{align}
for all $\tau\leq \tau_\ast$. Using the Gagliardo-Nirenberg estimate we get
\begin{align*}
\begin{split}
\|\upphi(\tau_\ast)\|_{L^\infty_x}\lesssim \|\partial_\Sigma\upphi(\tau_\ast)\|_{L^2_x}^{\frac{1}{2}}\|\partial_\Sigma^2\upphi(\tau_\ast)\|_{L^2_x}^\frac{1}{2}\lesssim \jap{\tau_\ast}^{-2-\betabar}.
\end{split}
\end{align*}
This proves \eqref{eq:intchardecay1}. For \eqref{eq:intchardecay2}, which is relevant only for large $r$, we assume that $r_\ast>\tilR$ and observe that by the fundamental theorem of calculus and the trace theorem, for any function $\upvarphi$ on $\{\tau=\tau_\ast\}$ (the integrations are on $\{\tau=\tau_\ast\}$)
\begin{align*}
\begin{split}
\int_{ \{r=r_\ast\}}(r^{\frac{1}{2}}\upvarphi)^2\ud S\lesssim \int_{ \{r=\tilR\}}(r^{\frac{1}{2}}\upvarphi)^2\ud S+\int_{\{\tilR\leq r\leq r_\ast\}}|r^{\frac{1}{2}}\upvarphi||\partial_r(r^{\frac{1}{2}}\upvarphi)|\ud S\ud r\lesssim \|\upvarphi(\tau_\ast)\|_E^2.
\end{split}
\end{align*}
Applying the Sobolev estimate on the non-geometric sphere $\{r=r_\ast\}$ (using \eqref{eq:VFexpansion1} to express angular derivatives in terms of the vectorfields $X$), the trace theorem on $\{r=\tilR\}$, and using the last estimate for $\upvarphi=\Omega^{\leq 2}\upphi$ we get
\begin{align*}
\begin{split}
|r_\ast^{\frac{1}{2}}\upphi(r_\ast)|^2\lesssim \sum_{j\leq 2}(\|\chi_{\{r\gtrsim1\}}X^j\upphi(\tau_\ast)\|_E^2+\|\chi_{\{r\lesssim1\}}\partial^j\upphi(\tau_\ast)\|_E^2).
\end{split}
\end{align*}
Using the assumed higher order decay estimates on $f$, we can repeat the earlier arguments to conclude that the right-hand side above is bounded by $\jap{\tau_\ast}^{-2-\betabar}$, proving the first estimate in \eqref{eq:intchardecay2}. The argument for the second estimate in \eqref{eq:intchardecay2} is similar. For any function $\upvarphi$ on $\{\tau=\tau_\ast\}$ (recall that $\tilupfy=\tilr\upvarphi$),
\begin{align*}
\begin{split}
\int_{ \{r=r_\ast\}}(\tilr\upfy)^2\ud S&\lesssim \int_{ \{r=\tilR\}}(\tilr\upfy)^2\ud S+\int_{\{\tilR\leq r\leq r_\ast\}}|\upfy||\tilr\partial_r(\tilr\upfy)|\ud S\ud r\\
&\lesssim \|\upfy(\tau_\ast)\|_E^2+\Big(\int_{\Sigma_{\tau_\ast}}\chi_{\geq \tilR}(L\tilupfy)^2 r^2\ud S\ud r\Big)^{\frac{1}{2}}\|\upfy(\tau_\ast)\|_E.
\end{split}
\end{align*}
The rest of the argument is the same as before, except that now we estimate $\int_{\Sigma_{\tau_\ast}}\chi_{\geq \tilR}(L\tilupfy)^2 r^2\ud S\ud r$ by $\jap{\tau_\ast}^{-3}$ as in \eqref{eq:Boxmdecaytemp0}. 
\end{proof}
\begin{remark}\label{rem:intcharcomparison}
Note that the estimates in Lemma~\ref{eq:Boxmhugens1} are consistent with what one gets by integration along characteristics (for instance if $f$ is radial or more generally applying Sobolev estimates on $\bbS^2$ and using the positivity of the fundamental solution of $\Box_{\bbR^{1+3}}$ as in \cite{MTT}). In general, if $|f|\lesssim \jap{r}^{-2-\alphabar}\jap{\tau}^{-2-\beta bar}$ satisfies suitable weighted derivative bounds, then similar arguments as in the proof of Lemma~\ref{eq:Boxmhugens1} give the estimate $|\uppsi|\lesssim \jap{\tau}^{-\min\{1+\alphabar,2+\betabar\}}$ , which is again consistent with integration along characteristics. In \eqref{eq:Boxmdecaytemp1}, integration along characteristics suggests the stronger estimate $|\uppsi|\lesssim \jap{r}^{-1}\jap{\tau}^{-2}$, and the deficit is caused by how we implement the $r^p$ energy estimates. A more careful implementation would yield the sharper estimate, where we would instead look at $\Box\upphi_1=\chi_\ast \chi_{\{r\leq \tau\}}f$ and $\Box\upphi_2=\chi_\ast\chi_{\{r\geq \tau\}}f$ and observe that $\uppsi(x
_\ast)=\upphi_1(x_\ast)+\upphi_2(x_\ast)$. Since the sharper estimate is not needed, we have not carried out this argument.
\end{remark}
\bibliographystyle{abbrv}
\bibliography{researchbib}

\begin{thebibliography}{10}

\bibitem{AW20}
L.~Abbrescia and W.~W.~Y. Wong.
\newblock Global nearly-plane-symmetric solutions to the membrane equation.
\newblock {\em Forum Math. Pi}, 8:e13, 71, 2020.

\bibitem{AIT21}
A.~Ai, M.~Ifrim, and D.~Tataru.
\newblock {The time-like minimal surface equation in Minkowski space: low
  regularity solutions}, 2021.

\bibitem{ABBM1}
L.~Andersson, T.~B\"ackdahl, P.~Blue, and S.~Ma.
\newblock {Stability for linearized gravity on the Kerr spacetime}, 2019.

\bibitem{AC2}
A.~Aurilia and D.~Christodoulou.
\newblock Theory of strings and membranes in an external field. {I}. {G}eneral
  formulation.
\newblock {\em J. Math. Phys.}, 20(7):1446--1452, 1979.

\bibitem{AC1}
A.~Aurilia and D.~Christodoulou.
\newblock Theory of strings and membranes in an external field. {II}. {T}he
  string.
\newblock {\em J. Math. Phys.}, 20(8):1692--1699, 1979.

\bibitem{BMP1}
H.~Bahouri, A.~Marachli, and G.~Perelman.
\newblock Blow up dynamics for the hyperbolic vanishing mean curvature flow of
  surfaces asymptotic to the {S}imons cone.
\newblock {\em J. Eur. Math. Soc. (JEMS)}, 23(12):3801--3887, 2021.

\bibitem{B1}
S.~Brendle.
\newblock Hypersurfaces in {M}inkowski space with vanishing mean curvature.
\newblock {\em Comm. Pure Appl. Math.}, 55(10):1249--1279, 2002.

\bibitem{DHR1}
M.~Dafermos, G.~Holzegel, and I.~Rodnianski.
\newblock The linear stability of the {S}chwarzschild solution to gravitational
  perturbations.
\newblock {\em Acta Math.}, 222(1):1--214, 2019.

\bibitem{DHRT1}
M.~Dafermos, G.~Holzegel, I.~Rodnianski, and M.~Taylor.
\newblock {The non-linear stability of the Schwarzschild family of black
  holes}, 2021.

\bibitem{DaHoRoTa2}
M.~Dafermos, G.~Holzegel, I.~Rodnianski, and M.~Taylor.
\newblock Quasilinear wave equations on asymptotically flat spacetimes with
  applications to {K}err black holes.
\newblock 2022.

\bibitem{DR1}
M.~Dafermos and I.~Rodnianski.
\newblock A new physical-space approach to decay for the wave equation with
  applications to black hole spacetimes.
\newblock In {\em X{VI}th {I}nternational {C}ongress on {M}athematical
  {P}hysics}, pages 421--432. World Sci. Publ., Hackensack, NJ, 2010.

\bibitem{DonKrie1}
R.~Donninger and J.~Krieger.
\newblock A vector field method on the distorted {F}ourier side and decay for
  wave equations with potentials.
\newblock {\em Mem. Amer. Math. Soc.}, 241(1142):v+80, 2016.

\bibitem{DKSW}
R.~Donninger, J.~Krieger, J.~Szeftel, and W.~W.~Y. Wong.
\newblock Codimension one stability of the catenoid under the vanishing mean
  curvature flow in {M}inkowski space.
\newblock {\em Duke Math. J.}, 165(4):723--791, 2016.

\bibitem{Ettinger}
B.~Ettinger.
\newblock {\em Well-posedness of the three-form field equation and the minimal
  surface equation in {M}inkowski space}.
\newblock ProQuest LLC, Ann Arbor, MI, 2013.
\newblock Thesis (Ph.D.)--University of California, Berkeley.

\bibitem{F-CS}
D.~Fischer-Colbrie and R.~Schoen.
\newblock The structure of complete stable minimal surfaces in {$3$}-manifolds
  of nonnegative scalar curvature.
\newblock {\em Comm. Pure Appl. Math.}, 33(2):199--211, 1980.

\bibitem{GiKlSz1}
E.~Giorgi, S.~Klainerman, and J.~Szeftel.
\newblock A general formalism for the stability of {K}err.
\newblock 02 2020.

\bibitem{GiKlSz2}
E.~Giorgi, S.~Klainerman, and J.~Szeftel.
\newblock Wave equations estimates and the nonlinear stability of slowly
  rotating {K}err black holes.
\newblock 05 2022.

\bibitem{HHV1}
D.~H\"{a}fner, P.~Hintz, and A.~Vasy.
\newblock Linear stability of slowly rotating {K}err black holes.
\newblock {\em Invent. Math.}, 223(3):1227--1406, 2021.

\bibitem{Hoppe13}
J.~Hoppe.
\newblock Relativistic membranes.
\newblock {\em J. Phys. A}, 46(2):023001, 30, 2013.

\bibitem{HKW1}
P.-K. Hung, J.~Keller, and M.-T. Wang.
\newblock Linear stability of {S}chwarzschild spacetime: decay of metric
  coefficients.
\newblock {\em J. Differential Geom.}, 116(3):481--541, 2020.

\bibitem{JNO15}
R.~L. Jerrard, M.~Novaga, and G.~Orlandi.
\newblock On the regularity of timelike extremal surfaces.
\newblock {\em Commun. Contemp. Math.}, 17(1):1450048, 19, 2015.

\bibitem{KlSz2}
S.~Klainerman and J.~Szeftel.
\newblock {\em Global nonlinear stability of {S}chwarzschild spacetime under
  polarized perturbations}, volume 210 of {\em Annals of Mathematics Studies}.
\newblock Princeton University Press, Princeton, NJ, 2020.

\bibitem{KlSz-GCM1}
S.~Klainerman and J.~Szeftel.
\newblock Construction of {GCM} spheres in perturbations of {K}err.
\newblock {\em Ann. PDE}, 8(2):Paper No. 17, 153, 2022.

\bibitem{KlSz-GCM2}
S.~Klainerman and J.~Szeftel.
\newblock Effective results on uniformization and intrinsic {GCM} spheres in
  perturbations of {K}err.
\newblock {\em Ann. PDE}, 8(2):Paper No. 18, 89, 2022.
\newblock With an appendix by Camillo De Lellis.

\bibitem{KlSz1}
S.~Klainerman and J.~Szeftel.
\newblock Kerr stability for small angular momentum.
\newblock {\em Pure Appl. Math. Q.}, 19(3):791--1678, 2023.

\bibitem{KMM17}
M.~Kowalczyk, Y.~Martel, and C.~Mu\~{n}oz.
\newblock On asymptotic stability of nonlinear waves.
\newblock In {\em S\'{e}minaire {L}aurent {S}chwartz---\'{E}quations aux
  d\'{e}riv\'{e}es partielles et applications. {A}nn\'{e}e 2016--2017}, pages
  Exp. No. XVIII, 27. Ed. \'{E}c. Polytech., Palaiseau, 2017.

\bibitem{KL1}
J.~Krieger and H.~Lindblad.
\newblock On stability of the catenoid under vanishing mean curvature flow on
  {M}inkowski space.
\newblock {\em Dyn. Partial Differ. Equ.}, 9(2):89--119, 2012.

\bibitem{Lin1}
H.~Lindblad.
\newblock A remark on global existence for small initial data of the minimal
  surface equation in {M}inkowskian space time.
\newblock {\em Proc. Amer. Math. Soc.}, 132(4):1095--1102, 2004.

\bibitem{LuOS1}
J.~{L\"uhrmann}, S.-J. {Oh}, and S.~{Shahshahani}.
\newblock {Stability of the Catenoid for the Hyperbolic Vanishing Mean
  Curvature Equation Outside Symmetry}.
\newblock 2022.

\bibitem{Luk1}
J.~Luk.
\newblock On the local existence for the characteristic initial value problem
  in general relativity.
\newblock {\em Int. Math. Res. Not. IMRN}, (20):4625--4678, 2012.

\bibitem{LuOh}
J.~Luk and S.-J. Oh.
\newblock Late time tail of waves on dynamic asymptotically flat spacetimes of
  odd space dimensions.
\newblock 04 2024.

\bibitem{MantonSutcliffe}
N.~Manton and P.~Sutcliffe.
\newblock {\em Topological solitons}.
\newblock Cambridge Monographs on Mathematical Physics. Cambridge University
  Press, Cambridge, 2004.

\bibitem{Marchali1}
A.~Marachli.
\newblock {\em {On the stability of certain minimal surfaces under the
  vanishing mean curvature flow in Minkowski space}}.
\newblock Theses, {Universit{\'e} Paris-Est}, Mar. 2019.

\bibitem{MMT1}
J.~Marzuola, J.~Metcalfe, and D.~Tataru.
\newblock Strichartz estimates and local smoothing estimates for asymptotically
  flat {S}chr\"{o}dinger equations.
\newblock {\em J. Funct. Anal.}, 255(6):1497--1553, 2008.

\bibitem{MST}
J.~Metcalfe, J.~Sterbenz, and D.~Tataru.
\newblock Local energy decay for scalar fields on time dependent non-trapping
  backgrounds.
\newblock {\em Amer. J. Math.}, 142(3):821--883, 2020.

\bibitem{MeTa}
J.~Metcalfe and D.~Tataru.
\newblock Global parametrices and dispersive estimates for variable coefficient
  wave equations.
\newblock {\em Math. Ann.}, 353(4):1183--1237, 2012.

\bibitem{MTT}
J.~Metcalfe, D.~Tataru, and M.~Tohaneanu.
\newblock Price's law on nonstationary space-times.
\newblock {\em Adv. Math.}, 230(3):995--1028, 2012.

\bibitem{Moschidis1}
G.~Moschidis.
\newblock The {$r^p$}-weighted energy method of {D}afermos and {R}odnianski in
  general asymptotically flat spacetimes and applications.
\newblock {\em Ann. PDE}, 2(1):Art. 6, 194, 2016.

\bibitem{NT13}
L.~Nguyen and G.~Tian.
\newblock On smoothness of timelike maximal cylinders in three-dimensional
  vacuum spacetimes.
\newblock {\em Classical Quantum Gravity}, 30(16):165010, 26, 2013.

\bibitem{OlSt}
J.~Oliver and J.~Sterbenz.
\newblock A vector field method for radiating black hole spacetimes.
\newblock {\em Anal. PDE}, 13(1):29--92, 2020.

\bibitem{RodSch1}
I.~Rodnianski and W.~Schlag.
\newblock Time decay for solutions of {S}chr\"{o}dinger equations with rough
  and time-dependent potentials.
\newblock {\em Invent. Math.}, 155(3):451--513, 2004.

\bibitem{Schlue1}
V.~Schlue.
\newblock Decay of linear waves on higher-dimensional {S}chwarzschild black
  holes.
\newblock {\em Anal. PDE}, 6(3):515--600, 2013.

\bibitem{Sh}
D.~Shen.
\newblock Construction of {GCM} hypersurfaces in perturbations of {K}err.
\newblock {\em Ann. PDE}, 9(1):Paper No. 11, 112, 2023.

\bibitem{Stefanov11}
A.~Stefanov.
\newblock Global regularity for the minimal surface equation in {M}inkowskian
  geometry.
\newblock {\em Forum Math.}, 23(4):757--789, 2011.

\bibitem{StuartHiggs}
D.~Stuart.
\newblock Dynamics of abelian {H}iggs vortices in the near {B}ogomolny regime.
\newblock {\em Comm. Math. Phys.}, 159(1):51--91, 1994.

\bibitem{Stuart1}
D.~M.~A. Stuart.
\newblock Modulational approach to stability of non-topological solitons in
  semilinear wave equations.
\newblock {\em J. Math. Pures Appl. (9)}, 80(1):51--83, 2001.

\bibitem{Tam-Zhou}
L.-F. Tam and D.~Zhou.
\newblock Stability properties for the higher dimensional catenoid in
  {$\mathbb{R}^{n+1}$}.
\newblock {\em Proc. Amer. Math. Soc.}, 137(10):3451--3461, 2009.

\bibitem{Tao09}
T.~Tao.
\newblock Why are solitons stable?
\newblock {\em Bull. Amer. Math. Soc.}, 46(1):1--33, 2009.

\bibitem{Tat1}
D.~Tataru.
\newblock Parametrices and dispersive estimates for {S}chr\"{o}dinger operators
  with variable coefficients.
\newblock {\em Amer. J. Math.}, 130(3):571--634, 2008.

\bibitem{Tat2}
D.~Tataru.
\newblock Local decay of waves on asymptotically flat stationary space-times.
\newblock {\em Amer. J. Math.}, 135(2):361--401, 2013.

\bibitem{Weinstein1}
M.~I. Weinstein.
\newblock Modulational stability of ground states of nonlinear
  {S}chr\"{o}dinger equations.
\newblock {\em SIAM J. Math. Anal.}, 16(3):472--491, 1985.

\bibitem{Wong1}
W.~W.~Y. Wong.
\newblock Regular hyperbolicity, dominant energy condition and causality for
  {L}agrangian theories of maps.
\newblock {\em Classical Quantum Gravity}, 28(21):215008, 23, 2011.

\bibitem{Wong2}
W.~W.~Y. Wong.
\newblock {Stability and instability of expanding solutions to the Lorentzian
  constant-positive-mean-curvature flow}, 2014.

\bibitem{Wong17}
W.~W.~Y. Wong.
\newblock Global existence for the minimal surface equation on
  {$\mathbb{R}^{1,1}$}.
\newblock {\em Proc. Amer. Math. Soc. Ser. B}, 4:47--52, 2017.

\bibitem{Wong18}
W.~W.~Y. Wong.
\newblock Singularities of axially symmetric time-like minimal submanifolds in
  {M}inkowski space.
\newblock {\em J. Hyperbolic Differ. Equ.}, 15(1):1--13, 2018.

\end{thebibliography}
\centerline{\scshape Sung-Jin Oh}
\smallskip
{\footnotesize
 \centerline{Department of Mathematics, UC Berkeley}
\centerline{Evans Hall 970, Berkeley, CA 94720-3840, U.S.A., and}
 \centerline{School of Mathematics, KIAS}
\centerline{80 Hoegi-ro, Seoul, 02455, Korea}
\centerline{\email{sjoh@math.berkeley.edu}}
} 

\medskip

\centerline{\scshape Sohrab Shahshahani}
\medskip
{\footnotesize
 \centerline{Department of Mathematics, University of Massachusetts, Amherst}
\centerline{710 N. Pleasant Street,
Amherst, MA 01003-9305, U.S.A.}
\centerline{\email{sshahshahani@umass.edu}}
}

\end{document}